\documentclass[11pt]{amsart}
\usepackage{amsmath,amsthm,amssymb,latexsym,color}
\usepackage{hyperref}
\usepackage{lipsum}

\theoremstyle{plain}
\newtheorem{theor}{Theorem}[section]
\newtheorem*{theorA}{Theorem~A}
\newtheorem*{theorB}{Theorem~B}
\newtheorem*{theorC}{Theorem~C}
\newtheorem*{cor*}{Corollary}

\newtheorem{prop}[theor]{Proposition}
\newtheorem{definition}[theor]{Definition}
\newtheorem{cor}[theor]{Corollary}
\newtheorem{lemma}[theor]{Lemma}
\newtheorem*{lemma*}{Lemma}
\newtheorem{rem}[theor]{Remark}
\theoremstyle{remark}

\def\R{{\mathbb R}}
\def\C{{\mathbb C}}
\def\CorR{{\bf \Theta}}
\def\cxi{{\bf i}}
\def\N{{\mathbb N}}
\def\Z{{\mathbb Z}}

\def\Event{{\mathcal E}}
\def\W{{\mathcal W}}
\def\Y{{\mathcal Y}}

\def\U{{\mathcal U}}
\def\oo{U}
\def\V{\widetilde{{\mathcal V}}}
\def\Proj{{\rm P}}
\def\Az{A-z}
\def\Prob{{\mathbb{P}}}

\def\col{{\rm Col}}
\def\n{{\bf{v}}}
\def\spn{{\rm span}\,}
\def\d{{\rm dist}}
\def\Exp{{\mathbb E}}
\def\ker{{\rm ker}}
\def\dim{{\rm dim\,}}
\def\codim{{\rm codim\,}}
\def\Vol{{\rm Vol}}

\def\Net{{\mathcal N}}
\def\M{{\mathcal M}}
\def\Id{{\rm I_n}}
\def\Idm{{\rm I_m}}
\def\cf{{\mathcal L}}
\def\LCD{{\rm LCD}}
\def\real{{\rm real}}
\def\supp{{\rm supp\,}}
\def\Incomp{{\rm Incomp\,}}
\def\HS{{\rm HS}}

\def\Ellips{{\bf E}}
\def\elclass{{\mathcal C}}

\def\disclass{{\mathcal D}}

\def\q{N}
\def\Az{A-z}
\DeclareMathOperator*{\esssup}{ess\,sup}
\def\disccollec{{\Upsilon}}

\title
{
On delocalization of eigenvectors of random non-Hermitian matrices
}
\author{Anna Lytova}
\address[Anna Lytova]{University of Opole, Poland}
\email{alytova@uni.opole.pl}

\author{Konstantin Tikhomirov}
\address[Konstantin Tikhomirov]{Georgia Institute of Technology}
\email{ktikhomirov6@gatech.edu}

\begin{document}

\maketitle

\begin{abstract}
We study delocalization of null vectors and eigenvectors of random matrices with i.i.d entries.
Let $A$ be an $n\times n$ random matrix with i.i.d real subgaussian entries of zero mean and unit variance.
We show that with probability at least $1-e^{-\log^{2} n}$
$$
\min\limits_{I\subset[n],\,|I|= m}\|\n_I\| \geq \frac{m^{3/2}}{n^{3/2}\log^Cn}\|\n\|
$$
for any {\it real} eigenvector $\n$ and any $m\in[\log^C n,n]$,
where $\n_I$ denotes the restriction of $\n$ to $I$.

Further, when the entries of $A$ are complex, with i.i.d real and imaginary parts, we show that
with probability at least $1-e^{-\log^{2} n}$ {\it all} eigenvectors of $A$ are delocalized in the sense that
$$
\min\limits_{I\subset[n],\,|I|= m}\|\n_I\| \geq \frac{m}{n\log^Cn}\|\n\|
$$
for all $m\in[\log^C{n},n]$.

Comparing with related results, in the range $m\in[\log^{C'}{n},n/\log^{C'}{n}]$ in the i.i.d setting
and with weaker probability estimates,
our lower bounds on $\|\n_I\|$ strengthen an earlier estimate
$\min\limits_{|I|= m}\|\n_I\| \geq c(m/n)^6\|\n\| $
obtained in [M. Rudelson, R. Vershynin, Geom. Func. Anal., 2016],
and bounds
$\min\limits_{|I|= m}\|\n_I\| \geq c(m/n)^2\|\n\|$
(in the real setting) and $\min\limits_{|I|= m}\|\n_I\| \geq c(m/n)^{3/2}\|\n\|$
(in the complex setting)
established in [K. Luh, S. O'Rourke, arXiv:1810.00489].

As the case of real and complex Gaussian matrices shows, our bounds are optimal up to the polylogarithmic multiples.
We derive stronger estimates without the polylogarithmic error multiples for null vectors of real $(n-1)\times n$ random matrices.

\end{abstract}

\tableofcontents

\section{Introduction}
\label{s:intro}

\subsection{Main definitions and earlier results}
\label{ss:intro1}

We say that a random variable $\xi$ is {\it subgaussian} ($K$-{\it subgaussian})
if $(\Exp|\xi|^p)^{1/p}\leq K\sqrt{p}$ for all $p\geq 1$ and some fixed $K>0$.
For any real or complex vector $X=(x_1,\dots,x_n)$, we define
its {\it non-increasing rearrangement} $X^*=(x_1^*,x_2^*,\dots,x_n^*)$
as follows: $x_i^* = \vert x_{\sigma(i)}\vert$, $i\le n$, where $\sigma$ is a permutation of
$[n]$ such that   $ \vert x_{\sigma(1)}\vert \ge  \vert x_{\sigma(2)}\vert \ge \ldots \geq  \vert x_{\sigma(n)}\vert$.
Equivalently,
$$x^*_i=\min\limits_{I\subset[n]:\,|I|=n-i+1}\max\limits_{j\in I}|x_j|\quad\mbox{for all $i=1,2,\dots,n$,}$$
where the minimum is taken over all subsets $I$ of $[n]=\{1,2,\dots,n\}$ of cardinality $n-i+1$.
In this paper, we are concerned with lower bounds for $x^*_i$ ($i\geq n/2$),
when $X$ is a unit eigenvector of a random square matrix with i.i.d subgaussian entries
or is a null vector of a rectangular $m\times n$ matrix (for $m<n$).
Following \cite{RV no-gap}, we will occasionally refer to those bounds as {\it no-gaps delocalization estimates}.

Delocalization of eigenvectors of random matrices has been actively studied,
especially in the setting of Wigner (and generalized Wigner) matrices.
The term {\it delocalization} usually refers to sup-norm ($\ell_\infty^n$--norm) delocalization or, more generally,
to upper bounds for inner products $|\langle Y,\n\rangle|$ for any fixed unit vector $Y$ and normalized eigenvectors $\n$.
With the $\ell_\infty^n$--delocalization, no single coordinate of the eigenvector
carries a significant mass.
In comparison, the no-gaps delocalization is the property of having few small coordinates,
so that every subset of vector components carries a non-negligible mass.
Both notions provide a way to measure how close distribution of an eigenvector is to the
uniform distribution on the sphere.
A related concept of {\it quantum unique ergodicity}
provides another viewpoint to this phenomenon.
Among many others, we refer to papers \cite{BHY, BY17, BYY1, EKYY, ESY, KY}
for eigenvectors of Wigner and band random matrices, as well as
to surveys \cite{RVW,Rud survey,Bourgade} for overview of some of existing results
on delocalization for various models of randomness and for further references.

No-gaps delocalization has found applications in both Hermitian and non-Hermitian settings.
The {\it Braess paradox} on random networks asserts that
under certain assumptions removing an edge in the random network
can decrease the traffic congestion (see, in particular, \cite{VR,Chung}).
A closely related phenomenon is a decrease in the spectral gap
of the normalized Laplacian of a random graph when adding a new edge
(see \cite{ERS,Rud survey}).
The probability of the decrease can be bounded using no-gaps delocalization estimates
for the second eigenvector of the Laplacian \cite{ERS,Rud survey}.
In particular, the no-gaps delocalization estimates obtained in \cite{RV no-gap}
allowed to strengthen the main result of \cite{ERS} and to show that randomly adding a new edge to a graph
decreases the spectral gap with probability at least $1/2-O(n^{-c})$ \cite{Rud survey}.
Further, in the non-symmetric setting, a weak form of no-gaps delocalization for
unit normals to linear spans of columns of $d$--regular random matrices
has been recently used in \cite{LLTTY18,LLTTY18b}
to establish the circular law for the limiting spectral distribution of
sparse random $d$--regular directed graphs (see also \cite{Cook, BCZ}).

In this paper, we consider non-Hermitian random matrices.
Before stating the main theorems, let us discuss several existing
results on delocalization of eigenvectors of such matrices,
including no-gaps and sup-norm delocalization.

In the non-Hermitian setting,
the sup-norm delocalization
for the eigenvectors was studied, in particular, in \cite{RV sup norm}, where it was shown that for an $n\times n$
random matrix $A$ with i.i.d subgaussian entries, with probability close to one every unit eigenvector ${\bf v}$
of $A$ satisfies $\|{\bf v}\|_\infty\leq C\sqrt{\log^{9}n\big /n}$.
Further, in \cite{NV} the bound $\|{\bf v}\|_\infty\leq C\sqrt{\log n\big /n}$ w.h.p
was derived for unit eigenvectors of $A$ corresponding to eigenvalues of small absolute value (of order $O(1)$).
The primary object of interest in \cite{NV} was a unit vector $\n$ in the kernel of an $(n-1)\times n$ random matrix $B$
with i.i.d subgaussian entries. It was shown that $\Prob\{\|\n\|_\infty\geq t/\sqrt{n}\}\leq e^{-ct^2}$ for all $t\geq C\sqrt{\log n}$,
with $C,c>0$ depending only on the subgaussian moment. It is not difficult to see, by considering the Gaussian matrix,
that, disregarding the values of $C,c$, this latter result is optimal.

{\it{}No-gaps delocalization} was considered in \cite{RV no-gap}
in a very general setting, including symmetric, skew-symmetric matrices
and matrices with i.i.d entries.
In particular, the main result of \cite{RV no-gap}
implies that, given an $n\times n$ random matrix $A$ with i.i.d $K$--subgaussian entries of zero
mean and unit absolute second moment, with probability at least $1-e^{-c(n-i)}$,
for every unit eigenvector ${\bf v}$ of $A$ and every $i\in [n/2,n-Cn^{-1/7}]$, one has
$v_i^*\geq {c(n-i)^{11/2}}/{n^6}$.
The results of \cite{RV no-gap} easily imply no-gaps delocalization of null vectors of an $(n-1)\times n$ matrix $B$
with i.i.d entries.
Recently, developing the approach from \cite{RV no-gap}, in work \cite{LuhRourke}
a stronger bound for the order statistics was established in the non-Hermitian case;
namely, it was shown in \cite{LuhRourke} that $v_i^*\geq {c(n-i)^{3/2}}/{n^2}$ (in the real case)
and $v_i^*\geq {c(n-i)}/{n^{3/2}}$ (in the complex case) for every $i\in[1,n-\log^2 n]$ and every unit eigenvector of $A$
with high probability. Moreover, the authors of \cite{LuhRourke} obtained
lower bounds for $v_i^*$ in the regime $i\in[n-\log^2 n,n]$, i.e.\ for smallest order statistics.
We note here that a variant of the no-gaps delocalization was considered in \cite{NV}
where the magnitude of the smallest component of the null unit vector was estimated.

If $G$ is an $n\times n$ Gaussian matrix with i.i.d real standard normal entries, strong estimates for
components of the non-increasing rearrangements of the real eigenvectors are immediately implied
by the orthogonal invariance of the matrix distribution. Indeed,
conditioned on the event that $G$ has at least $j$ real eigenvalues,
the unit eigenvector ${\bf v}$ corresponding to the $j$-th largest real eigenvalue
is uniformly distributed on the unit sphere $S^{n-1}(\R)$.
In particular, with probability close to one all unit real eigenvectors $\n=(v_1,\dots,v_n)$ of $G$ satisfy
${c(n-i)}/{n^{3/2}}\leq v_{i}^*\leq {C(n-i)}/{n^{3/2}}$
for all $n/2\leq i<n-C\log n$, for some universal constants $C,c>0$ (see e.g. \cite{David}).
A similar argument for a complex $n\times n$ matrix with i.i.d standard complex Gaussian entries
implies that with probability close to one every unit eigenvector $\n$ of such matrix
satisfies
${c\sqrt{n-i}}/{n}\leq v_{i}^*\leq {C\sqrt{n-i}}/{n}$ for all $n/2\leq i<n-C\log n$.
\medskip

Having in mind the Gaussian case, it is natural to formulate the following conjecture: for an $n\times n$ matrix $A$
with i.i.d subgaussian entries of zero mean and unit second absolute moment,
\begin{itemize}
\item In the real case,
all real unit eigenvectors $\n$ of $A$ should satisfy
$$v_{i}^*\geq \frac{c(n-i)}{n^{3/2}},\quad i\in [n/2,n-C\log n]$$
with probability close to one;

\item In the complex case,
all unit eigenvectors $\n$ of $A$ should satisfy
$$v_{i}^*\geq \frac{c\sqrt{n-i}}{n},\quad i\in [n/2,n-C\log n]$$
with high probability.
\end{itemize}

\subsection{Existing methods: singular value analysis, comparison to Gaussians, test projection}
\label{ss:intro2}

Comparing the aforementioned estimates in \cite{RV no-gap,LuhRourke}
with the above conjecture, one can see that they are suboptimal in the considered range of $i$.
Moreover, without significant new ingredients, the approach
of \cite{RV no-gap,LuhRourke} cannot produce the sharp bounds in that range for reasons explained further.

Consider a simpler setting of no-gaps delocalization estimates for a unit null vector ${\bf u}=(u_1,\dots,u_n)$ of $(n-1)\times n$
random matrix $B$ with i.i.d real subgaussian entries of zero mean and unit variance.
Given $i\in[n/2,n-1]$
and any $(n-i)$--element subset $I\subset[n]$, one has
$$s_{\max}(B_I)\Big\|\sum_{j\in I}u_j{\bf e}_j\Big\|_2
\geq \Big\|\sum_{j\in I}u_j\col_j(B)\Big\|_2=\Big\|\sum_{j\in I^c}u_j\col_j(B)\Big\|_2\geq
s_{\min}(B_{I^c})\Big\|\sum_{j\in I^c}u_j {\bf e}_j\Big\|_2,$$
where $B_{J}$ is the $(n-1)\times |J|$
matrix with columns $\col_j(B)$, $j\in J$, and $s_{\max},s_{\min}$ denote the largest/smallest singular values
of respective matrices. This immediately implies
$$\max\limits_{I\subset[n],\,|I|=n-i}s_{\max}(B_I)\,\sqrt{n-i}\,u^*_{i+1}
\geq c\min\limits_{I\subset[n],\,|I|=n-i}s_{\min}(B_{I^c}),$$
whence for any $t>0$, we have
\begin{align}
\Prob\big\{u^*_{i+1}\geq t\big\}&\geq 1-\Prob\Big\{c\min\limits_{I\subset[n],\,|I|=n-i}s_{\min}(B_{I^c})
\leq t\max\limits_{I\subset[n],\,|I|=n-i}s_{\max}(B_I)\,\sqrt{n-i}\Big\}\nonumber\\
&\geq 1-\sum\limits_{I\subset[n],\,|I|=n-i}
\Prob\big\{c\, s_{\min}(B_{I^c})
\leq t s_{\max}(B)\,\sqrt{n-i}\big\}\nonumber\\
&= 1-{n\choose n-i}
\Prob\big\{c\, s_{\min}(B_{[i]})
\leq t s_{\max}(B)\,\sqrt{n-i}\big\},\label{eq: aux parhgap;f}
\end{align}
where $[i]=\{1,2,\dots,i\}$.
The relation \eqref{eq: aux parhgap;f}
is the starting point of the argument of \cite{RV no-gap} (and also \cite{LuhRourke}) and is used to derive the main results of those papers
(see \cite[Proposition~4.1]{RV no-gap}).
But, while $s_{\max}(B)=O(\sqrt{n})$ with probability close to one,
the small ball probabilities for $s_{\min}(B_{[i]})$ are too weak to derive the optimal bound for $u^*_{i+1}$
from the last relation. Indeed, it can be verified that when $B$ is Gaussian, we have for all $i$ in $[n/2,n-1]$
and $\tau>0$:
$\Prob\big\{s_{\min}(B_{[i]})\leq \tau(n-i)/\sqrt{n}\big\}\geq (c\tau)^{n-i}$.
But this implies that the probability bound in \eqref{eq: aux parhgap;f}
is non-trivial {\it only when} $t\leq {c'(n-i)^{3/2}}/{n^2}$,
i.e.\ \eqref{eq: aux parhgap;f} can only be used to show that $u^*_{i+1}\geq {c'(n-i)^{3/2}}/{n^2}$
with high probability,
which falls short of the optimal lower bound ${c'(n-i)}/{n^{3/2}}$.
In a similar way, it can be argued that lower bounds for order statistics of eigenvectors of square random matrices
produced using this approach, are suboptimal.

\medskip

An alternative approach to delocalization is considered in \cite{NV}.
Let, as before, ${\bf u}$ be a null unit vector of an $(n-1)\times n$
random matrix $B$ with i.i.d entries of zero mean and unit absolute second moment.
In \cite{NV}, it is shown that
there is a coupling of ${\bf u}$ and a random vector ${\bf w}=(w_1,\dots,w_n)$ uniformly distributed on the unit sphere, such that
with probability at least $1-n^{-c}$, $|u_i^*-w_i^*|\leq n^{-1/2-c'}$,
where $i$ is in the range $[n^{1-c},n-n^{1-c}]$ for a small constant $c>0$  (see formula (36) in \cite{NV}).
The proof of the result is based on the Berry--Esseen theorem for frames used earlier in \cite{TV10}.
This result of \cite{NV}
provides much sharper estimates for $u_i^*$ compared to \cite{RV no-gap}, though
with much weaker probability bounds.
However, it is not clear how the approach can be adapted to studying delocalization of eigenvectors
rather than null vectors, and how to extend the range $[n^{1-c},n-n^{1-c}]$ for the components
of the non-increasing rearrangement for which the estimates are available.
We also note that a development of this approach to delocalization is considered in \cite{LuhRourke},
where sharp asymptotics for $\min\limits_{I\subset[n],\,|I|= \delta n}\|\n_I\|$
and $\max\limits_{I\subset[n],\,|I|= \delta n}\|\n_I\|$ is derived (for any constant $\delta$)
for the null unit vectors.

\medskip

Another argument in \cite{NV}, dealing with the sup-norm delocalization
and estimates for the smallest component of the null vector,
is based on using {\it test projections} ---
projections onto orthogonal complements to spans of some matrix columns.
The basic idea is that if a realization of the matrix $B$ admits a null unit vector with a ``large'' $\ell_\infty$--norm
then necessarily the test projections of some columns of $B$ are much smaller than average.
Estimating probability of the latter event amounts to application of known concentration or small ball probability
inequalities for norms of subgaussian random vectors.
Test projections were employed earlier in \cite{TV10} and \cite{RV sup norm} in a related context,
and are an important element of the present paper.

\subsection{
Main results}
\label{ss:intro3}

In this paper, we follow a geometric approach to random matrices
in our study of the eigenvectors.
Our main motivation is to establish the
optimal no-gaps delocalization estimates for real and complex non-Hermitian random matrices,
as well as for null vectors of real rectangular matrices.
The treatment of null vectors is significantly simpler, and we start by discussing the corresponding result first.
The following theorem is a simplified version
of Theorem~\ref{th: normal lower} in the text.
\begin{theorA}
For any $K\geq 1$ there are $C,c>0$
depending only on $K$ with the following property.
Let $n\geq C$, let
$B$ be an $(n-1)\times n$ random matrix with i.i.d $K$--subgaussian real entries
with zero mean and unit variance.
Further, let ${\bf u}=(u_1,\dots,u_n)$ be a random unit vector in $\ker(B)$.
Then for any
$i\in [n- cn, n-C\log n]$ and
$t\geq e^{-c(n-i)}$
we have
$$\Prob\Big\{u_{i+1}^*\leq \frac{(n-i)t}{n^{3/2}}\Big\}
\leq (Ct)^{n-i}+e^{-cn}.$$
\end{theorA}
Note that a unit random vector ${\bf w}$ uniformly distributed on the sphere $S^{n-1}(\R)$
(i.e.\ the null vector of a standard $(n-1)\times n$ Gaussian real matrix), satisfies
two-sided estimates $(ct)^{n-i}\leq\Prob\big\{w_{i+1}^*\leq \frac{(n-i)t}{n^{3/2}}\big\}\leq (Ct)^{n-i}$
for all $t\in(0,1]$ and $i\geq n/2$. Thus, our result recovers both the correct order of magnitude for the $(i+1)$--st largest
component of the null vector of $B$ and the optimal probability estimates (in a restricted range for $t$).
Moreover, considering discrete distributions with atoms, one can easily see that the additive term $e^{-cn}$
in our estimate is in general not removable; for example, the $(n-1)\times n$ Bernoulli
matrix admits null vector $(1,1,0,0,\dots,0)$ with probability $2^{1-n}$.

The generalized version of the above theorem --- Theorem~\ref{th: normal lower} ---
deals with {\it almost} null vectors.
As simple corollaries of Theorem~\ref{th: normal lower}, we obtain no-gaps delocalization bounds
on eigenvectors corresponding to small real eigenvalues and on singular vectors corresponding to small
singular values of square random matrices; see Corollaries~\ref{cor: eigenv} and~\ref{cor: sing}.


\bigskip

In the second part of the paper we consider eigenvectors of non-Hermitian square random matrices. The first main result here concerns
eigenvectors of  real matrices corresponding to real eigenvalues. Note that existence of real eigenvalues, hence real eigenvectors, of random real non-symmetric
matrices with i.i.d entries was established in \cite{TV real eig} under the assumption that the first four
moments of the matrix entries match those of the standard Gaussian variable;
statistics of real eigenvalues of Gaussian matrices were studied earlier in \cite{EKS}. We have:
\begin{theorB}
Let $n>2$, let $A$ be an $n\times n$ random matrix with i.i.d $K$--subgaussian
real entries with zero mean and unit variance.
Then with probability at least $1-e^{-\log^{2} n}$ we have
\begin{equation}\label{eq: B theor}
v_{i+1}^*\geq \frac{n-i}{n^{3/2}\log^C n}\;\;\mbox{ for every {\bf real} unit eigenvector $\n=(v_1,\dots,v_n)$}
\end{equation}
for every $i\in [n/2,n-\log^C{n}]$,
where $C>0$ may only depend on $K$.
\end{theorB}
The second main result of the second part of the paper concerns eigenvectors of non-Hermitian matrices with complex entries:
\begin{theorC}
Let $n>2$, let $A$ be an $n\times n$ random matrix with i.i.d $K$--subgaussian
complex entries with zero mean, unit second absolute moment and i.i.d real and imaginary parts.
Then with probability at least $1-e^{-\log^{2} n}$ we have
\begin{equation}\label{eq: C theor}
v_{i+1}^*\geq \frac{\sqrt{n-i}}{n\log^C n}\;\;\mbox{ for every unit eigenvector $\n=(v_1,\dots,v_n)$}
\end{equation}
for every $i\in [n/2,n-\log^C n]$,
where $C>0$ may only depend on $K$.
\end{theorC}

In view of the above discussion, the estimates are optimal up to the polylogarithmic multiple $\log^Cn$.
The probability estimate can be strengthened to $1-e^{-\log^{C'}n}$
for any constant $C'>2$, at expense of increasing the constant $C$ in the theorems
(we would like to note here that the weaker lower bounds on $v_{i+1}^*$ obtained in \cite{RV no-gap,LuhRourke} hold
with higher probability than in our paper).

Theorem~B can be restated as follows:
with probability at least $1-e^{-\log^{2} n}$ any real eigenvector satisfies
$$\min\limits_{I\subset[n],\;|I|=m}\|\n_I\|_2\ge \frac{m^{3/2}}{n^{3/2}\log^{\widetilde C}n}\|\n\|_2\quad\mbox{ for every
$m\geq \log^{\widetilde C}n$},$$
where $\n_I=(v_i)_{i\in I}$ denotes the restriction of $\n$ to $I$.
Similarly, Theorem~C implies that with probability $1-e^{-\log^{2} n}$
we have
$$\min\limits_{I\subset[n],\;|I|=m}\|\n_I\|_2\ge \frac{m}{n\log^{\widetilde C}n}\|\n\|_2\quad\mbox{ for every
$m\geq \log^{\widetilde C}n$}.$$

\begin{rem}
Theorem~B provides no-gaps delocalization estimates only for real eigenvectors.
For complex eigenvectors of real matrices with i.i.d entries, the situation is not clear to us.
Analysis of simple cases, in particular, eigenvectors of real Gaussian matrices
corresponding to eigenvalues with big imaginary and small real parts, shows that
such an eigenvector $\n$ is
{\it more delocalized} in our sense, so that a much stronger lower bound for $v_i^*$
than ${c(n-i)}/{n^{3/2}}$ can be verified.
It is reasonable to expect similar phenomenon for any random non-symmetric matrix with normalized i.i.d
real
entries, that is, sharp no-gaps delocalization estimates should depend on the magnitudes of the
real and imaginary parts of the eigenvalues.
It is not clear if the method of this paper can be adapted to catch this property.
\end{rem}

\begin{rem}
It is natural to expect that the method of this paper can be developed to treat the case of matrices with
non-identically distributed entries and without a bounded subgaussian moment.
We have not explored that direction.
\end{rem}


\subsection{Overview of the proof of Theorems~A,~B and~C}
\label{ss:intro5}

Our proof of Theorem~A uses test projections (mentioned when discussing \cite{NV}) as a basic tool.
Additionally, we employ an averaging procedure which, in a different form, was used earlier in
\cite{NV}, as well as in
\cite{T square}.
The third ingredient of the proof of Theorem~A is a small ball probability estimate
for the smallest singular value of certain auxiliary rectangular matrices with a special structure.
Those estimates are given in Section~\ref{s: small b}.

The study of eigenvectors of random matrices with help of test projections, compared to the null vectors,
is significantly more involved. The main difficulty consists in estimating the magnitude of
projections of the coordinate vectors $-z {\bf e}_i$ which arise when considering matrices $\Az$.
A strategy for estimating the projections was proposed earlier in \cite{RV sup norm},
and was based on studying certain biorthogonal systems of random vectors
and a special choice of the kernel of the projection operator.
While some elements of that argument turn out extremely useful in our context,
certain aspects of the proof, in particular, selection of the {\it spectral window} \cite[Section~5]{RV sup norm},
do not seem applicable. In this paper, we develop a geometric approach based on studying
dimensions of the ellipsoids generated by projected columns of the matrix $\Az $, as well as
the dual ellipsoid. We conclude the introduction with a brief description of the method.

Let $\n=(v_1,\dots,v_n)$ be a unit eigenvector of $A$ corresponding to an
eigenvalue $z$. It can be shown that with a high probability
$\n$ is {\it incompressible} and, in particular, for some $\beta=\Theta(n^{-1/2})$
and $\ell$ a small constant proportion of $n$,
we have $\n^*_\ell\ge \beta$ (see \cite{RV}).
Our goal is to show that for $i\in[n/2,n-{\rm polylog}(n)]$ with a high probability $\n^*_{i+1}\geq \theta$ for $\theta=\theta(i)$
given by the theorems.
Thus, we need to estimate from above the probability of ``bad'' events of the form
$$
|v_j|\geq \beta\quad\text{and}\quad|v_{j_1}|,|v_{j_2}|,\dots,|v_{j_k}|\leq\theta
$$
for a fixed collection of distinct indices $j;j_1,j_2,\dots,j_k$.
The last condition deterministically implies that the orthogonal projection  $\Proj_F(\col_j(\Az ))$ of $\col_j(\Az )$ onto $F=\spn\{\col_q(\Az ):\,q\in[n]\setminus\{j;j_1,j_2,\dots,j_k\}\}^\perp$
is contained in the ellipsoid
$$\Big\{\sum\limits_{\ell=1}^k a_\ell\Proj_F(\col_{j_\ell}(\Az )):\;\|(a_1,\dots,a_k)\|_2\leq \frac{\theta\sqrt{k}}{\beta}\Big\}.$$
This elementary observation reduces the no-gaps delocalization estimates to bounding probabilities of events of the form
$\big\{X\in\Ellips\big\}$,
where $X$ is an appropriate random vector and $\Ellips$ is an appropriate random ellipsoid.
The actual reduction procedure that we use replaces no-gaps delocalization with {\it two} conditions
of that form:
\begin{itemize}
\item[(a)]
$\Proj_F(-z{\bf e}_j)\in
\Big\{\sum\limits_{\ell=1}^k a_\ell\Proj_F(-z{\bf e}_{j_\ell}):\;\|a\|_2\leq \frac{\theta\sqrt{k}}{\beta}\Big\}+B
\quad\mbox{and};$
\item[(b)]
$\Proj_F(\col_j(\Az ))\in\Big\{\sum\limits_{\ell=1}^k a_\ell\Proj_F(-z{\bf e}_{j_\ell})
:\;\|a\|_2\leq \frac{\theta\sqrt{k}}{\beta}\Big\}+B',$
\end{itemize}
where $B,B'$ are appropriate dilations of the Euclidean ball in $F$
(see Lemma~\ref{l: two conditions}).
The two conditions are treated using different techniques.
The first one can be restated in terms of the {\it dual basis} for $\Proj_F(-z{\bf e}_j),\Proj_F(-z{\bf e}_{j_\ell})$,
$\ell=1,2,\dots,k$ in $F$ as a statement about magnitudes of inner products
of the dual basis vectors with a specially chosen vector in $F$
(see Lemma~\ref{l: y existence} giving a ``raw'' deterministic statement).
That relation is one of basic elements of the proof and essentially taken from \cite{RV sup norm},
although not stated there explicitly. Bounding probability of (a) is then reduced to estimating probability of the
event of the form
$$\big\{\exists\, Y\in F:\;\|Y\|_2\leq T\mbox{,
$\langle Y_{j},Y\rangle=1$ and $|\langle Y_{j_\ell},Y\rangle|\leq \delta$}\big\},$$
where $Y_j,Y_{j_1},\dots,Y_{j_k}$ is the dual basis for $\Proj_F(-z{\bf e}_j),\Proj_F(-z{\bf e}_{j_1}),\dots,\Proj_F(-z{\bf e}_{j_k})$,
and $T,\delta>0$ are appropriate parameters.
In fact, for technical reasons, we work with perturbations of the dual basis, but
here we consider a simplified scheme.
Estimating probability of the last event is challenging, first, because the random
vectors $Y_j,Y_{j_1},\dots,Y_{j_k}$ do not have to be isotropic (i.e.\ the covariance structure may be not a multiple of identity)
and, second,
since existence of a vector $Y$ satisfying the conditions is not easy to restate in terms
of anti-concentration properties of $Y_j,Y_{j_\ell}$'s.
The first difficulty --- treatment of {\it anisotropic} random vectors --- was addressed in \cite{RV sup norm},
and we reuse some of the estimates from \cite{RV sup norm}, while adding new relations
(see Subsection~\ref{subs: two types}).
The second difficulty --- a reduction of the condition for $Y_j,Y_{j_\ell}$ and $Y$
to anti-concentration estimates for $Y_j,Y_{j_\ell}$'s ---
is resolved by constructing a special discretization of the set of admissible vectors $Y$,
and taking the union bound over the set. The principal issue in this part of the proof
is to find a decoupling that would allow to resolve probabilistic dependencies
between the admissible vectors $Y$ and vectors $Y_j,Y_{j_\ell}$'s
(otherwise, applying small ball probability estimates for $|\langle Y_{j_\ell},Y\rangle|$
would be impossible in a direct way). This is the central part of our argument
(see Lemmas~\ref{l: y reduction} and~\ref{l: aux decoupling}).

To bound the (conditional) probability of the event (b), given a realization
of the ellipsoid $\Ellips=\big\{\sum_{\ell=1}^k a_\ell\Proj_F(-z{\bf e}_{j_\ell})
:\;\|a\|_2\leq \frac{\theta\sqrt{k}}{\beta}\big\}$,
we utilize independence of $\col_j(A)$ from $F$ and $\Ellips$.
A straightforward argument estimating the probability in terms of the volume of $\Ellips$
turns out too rough for our purposes. The volume of $\Ellips$ may be large because of few long
semi-axes, while the probability of $\Proj_F(\col_j(\Az ))\in\Ellips+B'$ is still small.
In our estimate, we use that $\Proj_F(\col_j(A))$ has relatively small Euclidean norm
(of order $\sqrt{k}$ or slightly greater) with large probability,
and replace the ellipsoid $\Ellips$ with its ``truncation'' obtained by intersecting with a Euclidean ball
centered at $\Proj_F(-z{\bf e}_j)$ (see Lemma~\ref{l: aux aponfa;sa32r}).

Our proof uses several discretizations ---
of the family of ellipsoids; of basic sequences; of vectors in a given linear subspace (see Definitions \ref{page: elclass}, \ref{page: disclass}, \ref{page: disccollec} of classes $\elclass(R,{\bf b})$, $\disclass(r ,{\bf p})$, and $\disccollec(\W_\q,\delta',T)$).
Some other aspects of the proof not mentioned here are discussed at the beginning of Section~\ref{s: eigenvectors}.

\section{Preliminaries}
\label{s:prelim}

Let us start by introducing notation.
Given a finite set $I$, we denote by $|I|$ its cardinality. For a positive integer $k$, we write $[k]$ for the set $\{1,2,\dots,k\}$.
By ${\bf 1}_A$ we denote the indicator of an event or a subset $A$.
Everywhere in the text, $\CorR$ denotes either the field of real numbers $\R$, or complex numbers $\C$.
For a complex number $z$, we denote by $\bar z$ the conjugate of $z$;
$\cxi$ stands for imaginary unit.
The canonical inner product in $\CorR^n$ is denoted by $\langle\cdot,\cdot\rangle$,
and the standard Euclidean norm in $\CorR^n$ --- by $\|\cdot\|_2$.
We denote the sup-norm (or $\ell_\infty$--norm) by $\|\cdot\|_\infty$.
For an $m\times n$ matrix $B$, its columns are denoted by $\col_i(B)\in\CorR^{m}$, $i\in[n]$.
When the matrix is clear from the context, we simply write $\col_i$ instead of $\col_i(B)$.
The spectral norm of $B$ is denoted by $\|B\|$,
and its Hilbert--Schmidt norm --- by $\|B\|_{HS}$.
The conjugate transpose of a matrix $B$ will be denoted by $B^*$.
For a linear subspace $E$ of $\CorR^k$, we write $E^\perp$
for its orthogonal complement. We will write $\Proj_E:\CorR^k\to\CorR^k$
for the orthogonal projection operator onto $E$.
Given two subsets $S_1,S_2$ of $\CorR^n$, we write $\d(S_1,S_2)$ for the Euclidean distance between $S_1$ and $S_2$.
Further, we write $S_1+S_2$ for the {\it Minkowski sum} of $S_1$ and $S_2$ defined as
$S_1+S_2:=\{x+y:\;x\in S_1,\;y\in S_2\}$.
Let $B_2^k(\CorR)$ (resp., $S^{k-1}(\CorR)$) be the unit Euclidean ball
(resp., Euclidean sphere) in $\CorR^k$. We also write $B_2^E:=B_2\cap E$ for the unit Euclidean ball in a linear subspace $E$.

Given a real random variable $\xi$, let $\cf(\xi,\cdot)$ be its {\it L\'evy concentration function} defined by
$$\cf(\xi,t):=\sup\limits_{a\in\CorR}\Prob\big\{|\xi-a|\leq t\big\},\quad t>0.$$
More generally, for a random vector $X$ in $\CorR^m$ let
$$\cf(X,t):=\sup\limits_{W\in\CorR^m}\Prob\big\{\|X-W\|_2\leq t\big\},\quad t>0.$$

Following \cite{RV no-gap}
for any complex vector $W\in\C^m$ we define $\real(W)$ as a vector in $\R^{2m}$
of the form $\Re(W)\oplus \Im(W)$, where $\Re(W)$ and $\Im(W)$ is the real and imaginary part of $W$,
respectively, and $\oplus$ denotes vector concatenation.
Further, for any complex subspace $E\subset \C^m$ let
$\real(E)\subset\R^{2m}$ be the subspace defined as $\real(E):=\{\real(W):\;W\in E\}$.
The following relation, taken from \cite{RV no-gap}, holds:
$$
\Proj_{\real(E)}(\real(Y))=\real(\Proj_E(Y))\;\;\mbox{for any complex subspace }E\subset\C^m,\; Y\in\C^m.
$$
The next lemma will be useful
\begin{lemma}\label{l: real ell}
Let $\Ellips$ be a $k$--dimensional ellipsoid in $\C^n$ defined as
$$\Ellips:=\Big\{\sum\limits_{i=1}^k a_i X_i:\;\|(a_1,\dots,a_k)\|_2\leq 1\Big\},$$
where $\{X_1,\dots,X_k\}$ is a collection of pairwise orthogonal vectors in $\C^n$.
Then $\real(\Ellips)$ is a $2k$--dimensional ellipsoid in $\R^{2n}$ with semi-axes
$\real(X_1)$, $\real(\cxi X_1),\dots,$ $\real(X_k)$, $\real(\cxi X_k)$.
\end{lemma}

Given a random vector $X$ in $\CorR^m$, we say that $X$ is {\it isotropic} if
any one-dimensional projection $\langle X,{\bf w}\rangle$ (${\bf w}\in S^{m-1}(\CorR)$) has zero mean
and $\Exp|\langle X,{\bf w}\rangle|^2=1$.
In particular, a complex vector with independent coordinates
of zero mean and unit absolute second moment (with i.i.d real and imaginary parts) is isotropic.
A variable $\xi$ in $\CorR$ is {\it $K$--subgaussian} if $(\Exp|\xi|^p)^{1/p}\leq K\sqrt{p}$ for all $p\geq 1$.
The definition of a $K$--subgaussian variable admits several equivalent formulations;
see, for example, \cite[Section~5.2.3]{V 12}.

Throughout the text, we deal with two models of random vectors $X$:
\begin{equation}\label{Asmp 1}\tag{{\bf *}}
\mbox{$X$ is real isotropic with i.i.d.\ $K$--subgaussian coordinates, if $\CorR=\R$},
\end{equation}
and
\begin{equation}\label{Asmp 2}\tag{{\bf **}}
\begin{split}\mbox{$X$ is complex isotropic with i.i.d.\ $K$--subgaussian coordinates,}\\
\mbox{each coordinate with i.i.d.\ real and imaginary parts, if $\CorR=\C$.}
\end{split}
\end{equation}

\medskip

Following \cite{RV, RV smallest}, for any $\delta,\rho>0$, we define the set of {\it incompressible} unit vectors
$$
\Incomp_n(\delta,\rho):=\Big\{X\in S^{n-1}(\CorR):\, \d\big(X,\{Y\in \CorR^n:\, |\supp Y|\le \delta n\}\big)>\rho\Big\}.
$$
The incompressible vectors naturally appear in kernels of ``almost square'' random matrices with independent entries.
The next standard lemma can be verified with help of an $\varepsilon$--net argument
(see \cite{LPRT,RV,RV smallest}) and elementary properties of incompressible vectors
\cite[Lemma~3.4]{RV}:
\begin{lemma}\label{l: incomp in ker}
For any $\mathcal M\geq 1$ and $K\geq 1$ there are $\zeta>0$, $c>0$ depending only on $\mathcal M$ and $K$
with the following property.
Let $m\leq n\leq 2m$ and let $W$ be an $m\times n$ random matrix with i.i.d.\ rows
satisfying \eqref{Asmp 1} or \eqref{Asmp 2}.
Further, let $M\in\CorR^{m\times n}$ be any fixed matrix of spectral norm at most $\mathcal M\,\sqrt{n}$.
Then with probability at least $1-2e^{-cn}$
$$\|(W+M)Y\|_2\geq c\sqrt{n}\quad\mbox{ for any }\;\;Y\in S^{n-1}(\CorR)\setminus\Incomp_n(\zeta,\zeta),$$
and, in particular,
$$\|(W+M)Y\|_2\geq c\sqrt{n}\quad\mbox{ for any }\;\;Y\in S^{n-1}(\CorR)\mbox{ with }
\big|\big\{i\leq n:\;|Y_i|\geq\zeta/\sqrt{n}\big\}\big|< \zeta n.$$
\end{lemma}

The next notion and two theorems are taken from \cite{RV,RV smallest}.
Given a {\it real} vector $X\in\R^n$, and parameters $\alpha>0$ and $\gamma\in(0,1)$,
{\it the least common denominator (LCD)} of $X$ is given by
$$
\LCD_{\alpha,\gamma}(X):=\inf\big\{\theta>0:\;\d(\theta X,\, \Z^n)<\min(\gamma \|\theta X\|_2,\alpha)\big\}.
$$
Moreover, if $E$ is a linear subspace of $\R^n$, then the LCD of $E$ is defined as follows:
\begin{align*}
\LCD_{\alpha,\gamma}(E)&:=
\inf\big\{\LCD_{\alpha,\gamma}(X):\;X\in E\cap S^{n-1}(\R)\big\}\\
&=\inf\big\{\|Y\|_2:\, Y\in E,\, \d(Y,\, \Z^n)<\min(\gamma \|Y\|_2,\alpha)\big\}.
\end{align*}

\begin{theor}[{Distance to a subspace, \cite[Theorem 4.2]{RV smallest}}]\label{th: 4.2}
Let $X$ be a random vector in $\R^n$ satisfying \eqref{Asmp 1} with a parameter $K$.
Let $E$ be a subset of $\R^n$ with $m:=\codim E<n$.
Then for every $Y\in \R^n$, any
$\alpha>0$, $\gamma\in (0,1)$ and for all $t\ge {m}/\LCD_{\alpha,\gamma}(E^\perp)$, we have
$$
\Prob\{\d(X,\, E+Y)\le t\}\le \bigg(\frac{C t}{\gamma\sqrt{m}}\bigg)^{m}+C^{m} e^{-c\alpha^2},
$$
where $c,C>0$ depend only on $K$. In particular, 
for any unit vector $Z\in\R^n$  and any $t>0$ we have
$$
\Prob\{ |\langle X,\,Z\rangle|\le t\}\le \frac{C t}{\gamma}+C e^{-c\alpha^2}+\frac{C}{\gamma\,\LCD_{\alpha,\gamma}(Z)}.
$$
\end{theor}

\begin{theor}[{Structure of a random subspace, \cite[Theorem~4.3]{RV smallest}}] \label{th: 4.3}
For any $K,\mathcal M\geq 1$ there exist $c_{\text{\tiny\ref{th: 4.3}}},C_{\text{\tiny\ref{th: 4.3}}}>0$ depending only on $K$
and $\mathcal M$ with the following property.
Let $n\geq C_{\text{\tiny\ref{th: 4.3}}}$,
$n> m\geq n-c_{\text{\tiny\ref{th: 4.3}}} n$.
Let $X_1,\dots, X_{m}$ be mutually independent random vectors in $\R^n$
satisfying \eqref{Asmp 1} with parameter $K$.
Further, let $Y_1,\dots,Y_m$ be fixed vectors in $\R^n$
such that the spectral norm of the $n\times m$
matrix with columns $Y_1,\dots,Y_m$ is bounded above by $\mathcal M\sqrt{n}$.
Then for $E:=\spn\{X_1+Y_1,X_2+Y_2,\dots,X_{m}+Y_m\}$ we have
$$
\Prob\big\{\LCD_{c_{\text{\tiny\ref{th: 4.3}}} \sqrt{n},c_{\text{\tiny\ref{th: 4.3}}}}(E^\perp)\le
\sqrt{n}e^{c_{\text{\tiny\ref{th: 4.3}}} n/(n-m)}\big\}\le e^{-c_{\text{\tiny\ref{th: 4.3}}} n}.
$$
\end{theor}

The next result is essentially proved for centered matrices in \cite{RV}; see \cite[Lemma~5.8]{RV}
and \cite[Lemma~4.8]{RV smallest}.
For non-centered matrices, the statement follows by a straightforward modification of the argument.
\begin{theor}\label{th: LCD of small}
For any $K,\mathcal M>0$ there are $C_{\text{\tiny\ref{th: LCD of small}}},c_{\text{\tiny\ref{th: LCD of small}}}>0$
depending only on $K,\mathcal M$ with the following property.
Let $m\geq C_{\text{\tiny\ref{th: LCD of small}}}$
and let $B$ be an $m\times m$ random matrix with i.i.d $K$--subgaussian real entries
of zero mean and unit variance. Further, let $M$ be any fixed $m\times m$ real matrix with $\|M\|\leq \mathcal M\sqrt{m}$.
Then with probability at least $1-e^{-c_{\text{\tiny\ref{th: LCD of small}}}m}$ we have
$$\|(B+M)Z\|_2\geq \frac{c_{\text{\tiny\ref{th: LCD of small}}}m}
{\max(\sqrt{m},
\LCD_{c_{\text{\tiny\ref{th: LCD of small}}} \sqrt{m},c_{\text{\tiny\ref{th: LCD of small}}}}(Z))}\quad\mbox{for any }Z\in S^{m-1}(\R).$$
\end{theor}

The next theorem
provides small ball probability estimates for projections of complex random vectors
in the absence of strong structural assumptions on the range of the projection.
In one-dimensional setting, those are
reformulations of Erd\H os--Littlewood--Offord and L\'evy--Kolmogorov--Rogozin inequalities
\cite{LO,Erdos,Kolmogorov,Rogozin,Esseen,Kesten};
high-dimensional versions are studied, in particular, in \cite{RV images,LPP,RV no-gap}.
The statement below can be proved by a reduction to the real setting via the relation
$\cf(\Proj_F(X),\sqrt{k}t)=\cf(\Proj_{\real(F)}(\real(X)),\sqrt{k}t)$ (see \cite{RV no-gap})
and applying the above anti-concentration estimates from \cite{RV smallest}.

\begin{theor}\label{l: sb compl incomp}
For any $\zeta>0$ there is $C>0$ depending only on $\zeta$ with the following property.
Let $X$ be a random vector in $\C^m$ satisfying assumption
\eqref{Asmp 2}.
Further, let $Y$ be a fixed vector in $\Incomp_m(\zeta,\zeta)$.
Then
$$\cf(\langle X,Y\rangle,t)\leq (Ct+Cm^{-1/2})^2\;\mbox{ for all $t>0$.}$$

More generally, if $F$ is a fixed subspace in $\C^m$ of dimension $k$ such that
$F\cap S^{m-1}(\C)\subset \Incomp_m(\zeta,\zeta)$ then
$$\cf(\Proj_F(X),\sqrt{k}t)\leq (Ct+Cm^{-1/2})^{2k}\;\mbox{ for all $t>0$.}$$
\end{theor}

\section{An estimate for the smallest singular value of real rectangular matrices}\label{s: small b}

Let $M$ be a $d\times r$ real random matrix with independent isotropic columns.
In this section, we are concerned with estimating the smallest singular value of $M$
(assuming certain small ball probability estimates and concentration for individual matrix columns).
The problem of estimating $s_{\min}(M)$ is a standard question within the random matrix theory,
due to its relevance to questions in statistics,
convex geometric analysis, computer science. We refer, among others, to surveys and books
\cite{BS 2010,RV ICM 2010,V 12,RVW,V 16} for more information and further references.
For a tall matrix (for example, with $d\geq 2r$) with i.i.d subgaussian entries, a basic $\varepsilon$--net argument gives
$s_{\min}(M)\geq c\sqrt{d}$ with probability at least $1-e^{-cd}$, with $c>0$ depending
only on the subgaussian moment; with more elaborate arguments, similar estimates
are available for ``almost square'' matrices and
under more general assumptions on the matrix columns (see, in particular, \cite{LPRT,RV smallest,KM,T IMJ,Yaskov}).
This estimate is close to optimal in general: for example, the random Bernoulli matrix is of deficient rank
with probability $e^{-c'd}$. However, with some additional assumptions on anti-concentration
of the matrix columns, the singularity probability is much smaller.

To obtain satisfactory quantitative
estimates for $s_{\min}$, we modify
the basic $\varepsilon$--net argument by replacing
the spectral norm of the matrix with the Hilbert--Schmidt norm:
due to much better concentration properties of the latter, we get strong small ball probability estimates for $s_{\min}$,
provided that such estimates exist for arbitrary linear combinations of the matrix columns.
We start with a simple lemma obtained as a consequence of the Hanson--Wright inequality
\cite{HW, RV HS}.

\begin{lemma}\label{l: HS conc}
Let $m,r,d\in\N$, $K>0$, and let
$X_1,\dots,X_r$ be i.i.d vectors in $\R^m$ satisfying \eqref{Asmp 1}.
Further, let $M$ be an $m\times r$ random matrix such that
$\col_\ell(M)=\Proj_F(X_\ell)$, $\ell=1,\dots,r$,
with $F$ being a $d$--dimensional fixed subspace of $\R^m$.
Then
$$
\Prob\big\{\|M\|_{HS}\geq C_{\text{\tiny\ref{l: HS conc}}}\sqrt{rd}\big\}\leq \exp(-c_{\text{\tiny\ref{l: HS conc}}}rd),
$$
for some $C_{\text{\tiny\ref{l: HS conc}}}>1$, $c_{\text{\tiny\ref{l: HS conc}}}>0$ depending only on the subgaussian moment $K$.
\end{lemma}
\begin{rem}
The above lemma can be viewed as a special case of \cite[Theorem~2.1]{RV HS}, in which we take
$A$ as the $mr\times mr$ block-diagonal matrix with $r$ identical blocks representing orthogonal projection
onto $F$, and consider a concatenated $mr$--dimensional random vector composed of $X_\ell$'s.
\end{rem}

In the following statement we construct the required net inside the Euclidean ball.
The construction is a direct application of the probabilistic method, and was previously used in other contexts
in high-dimensional convex and discrete geometry. As a classical illustration of this method we refer to Rogers' paper
\cite{Rogers}. Among recent applications, see \cite{GLPT} for randomized coverings of convex sets
and paper \cite{KlLi} where a statement similar to the one below is proved.

\begin{lemma}\label{l: HS net}
Let $r\geq C_{\text{\tiny\ref{l: HS net}}}$, and let $m$ satisfy $m\leq 2^r$.
For every $t\in (2^{-2^r},1]$, there exists a non-random subset
$\Net\subset \frac{3}{2}B_2^r(\R)\setminus \frac{1}{2}B_2^r(\R)$ of cardinality at most $(C_{\text{\tiny\ref{l: HS net}}}/t)^r$ with the following property:
For every real $m\times r$
matrix $A$ with the Hilbert--Schmidt norm at most $\sqrt{r}$, and every $X\in S^{r-1}(\R)$
there is $X'\in \Net$ satisfying $\|A(X-X')\|_2\le t$.
Here, $C_{\text{\tiny\ref{l: HS net}}}>0$ is a universal constant.
\end{lemma}
\begin{proof}
For brevity, we write $S^{r-1}$, $B_2^r$ instead of $S^{r-1}(\R),B_2^r(\R)$.
The proof involves a covering
of the set of matrices with uniformly bounded Hilbert--Schmidt norms, and a standard covering of $S^{r-1}$.

We start with constructing a net of matrices.
Let $\HS$ be the set of all real-valued $m\times r$ matrices having the Hilbert--Schmidt norm at most $\sqrt{r}$.
Since the absolute values of coordinates of matrices from $\HS$ are bounded by
$\sqrt{r}$, any $\frac{1}{\sqrt{mr}}$--net in the parallelotope $[-\sqrt{r},\sqrt{r}]^{m\times r}$
(with respect to the $\ell_\infty$--metric) is also a $1$--net for $\HS$ (with respect to the Hilbert--Schmidt norm).
Thus, there is a $1$--net $\Net_{HS}\subset\HS$ of cardinality at most $(2\sqrt{m}r)^{mr}$.

Further,
let $\widetilde\Net$ be a Euclidean $\frac{t}{2\sqrt{r}}$--net
on $S^{r-1}$ of cardinality at most $(6\sqrt{r}/t)^r$.

Given positive real numbers $t_1,t_2$ and a random $r$--dimensional
vector $X=(x_1,\dots,x_r)$ uniformly distributed in the Euclidean ball $t_1 B_2^r$, for any fixed matrix $A\in \HS$ we have
\begin{equation}\label{eq: ax pohas;fa}
\Prob\big\{\|AX\|_2>t_2\big\}\le (t_1/t_2)^2.
\end{equation}
Indeed, we have $\Exp x_\ell^2=t_1^2/(r+2)$ for all $\ell\leq r$,
and, applying the singular value decomposition to $A$, obtain
$\Prob\big\{\|AX\|_2>t_2\big\}
=\Prob\{\sum_{\ell=1}^r s_\ell^2\,x_\ell^2>t_2^2\}$, where $s_1,\dots, s_r$ are the singular values of $A$.
Since $\sum_{\ell=1}^r s_\ell^2\le r$, Markov's inequality implies
\begin{equation*}
\Prob\big\{\|AX\|_2>t_2\big\}\le t_2^{-2}\sum_{\ell=1}^r s_\ell^2\, \Exp x_\ell^2\le (t_1/t_2)^2.
\end{equation*}

Now, define $S$ as a collection of mutually independent random vectors
uniformly distributed in the shell $\frac{3}{2}B_2^r\setminus \frac{1}{2}B_2^r$,
of cardinality $\lceil\big(\frac{C}{t}\big)^r\rceil$,
where the constant $C>1$ is to be chosen later.
Note that in order to prove the lemma, it is enough to show that
$$
\Prob\big\{\forall A\in\HS\;\;\forall X\in S^{r-1}\;\; \exists  X'\in S:\,\,\, \|A(X-X')\|_2\le t\big\}>0,
$$
and then take $\Net$ as an appropriate realization of $S$.
Note that for any $X\in S^{r-1}$ there is $X''\in \widetilde \Net$ with $\|X-X''\|_2\leq\frac{t}{2\sqrt{r}}$,
so that for any matrix $A\in\HS$ we have $\|A(X-X'')\|_2\leq \frac{t}{2}$.
Thus, the above estimate holds if
$$1-p_0:=\Prob\big\{\forall A\in\HS\;\;\forall X\in \widetilde\Net\;\; \exists  X'\in S:\,\,\, \|A(X-X')\|_2\le t/2\big\}>0.$$
By the union bound,
\begin{align*}
  p_0&\leq \sum_{X\in\widetilde \Net}\Prob\big\{ \exists A\in\HS\;\;\forall  X'\in S:\,\,\, \|A(X-X')\|_2> t/2\big\}
  \\
&\leq \sum_{X\in\widetilde \Net}\sum\limits_{A\in \Net_{HS}}
\Prob\Big\{\forall  X'\in S\cap \Big({\frac{t}{8}}B_2^r+X\Big):\,\,\, \|A(X-X')\|_2> t/4\Big\}.
\end{align*}
Fix for a moment any $A\in\Net_{HS}$ and $X\in\widetilde\Net$, and
denote by $\Event$ the event that the Euclidean ball $\frac{t}{8}B_2^r+X$ contains
at least $(C/48)^r$ points from $S$.
Observe that $\frac{t}{8}B_2^r+X$ is entirely contained in $\frac{3}{2}B_2^r\setminus \frac{1}{2}B_2^r$;
thus, the probability that a random vector uniformly distributed on $\frac{3}{2}B_2^r\setminus \frac{1}{2}B_2^r$
falls into $\frac{t}{8}B_2^r+X$ is greater than $({t}/{12})^r$.
Hence, using the definition of $S$, we get
$$\Prob(\Event^c)\leq {|S|\choose \lfloor (C/48)^r\rfloor}\,\bigg(1-\Big(\frac{t}{12}\Big)^r\bigg)^{|S|-\lfloor (C/48)^r\rfloor}
\leq \big(eC/t\big)^{r\lfloor (C/48)^r\rfloor}e^{-(C/12)^r/2}
\leq\exp(-e^{r}),$$
as long as the constant $C$ is chosen sufficiently large.
Further, conditioned on $\Event$,
we can estimate the probability
$\Prob\{\forall  X'\in S\cap ({\frac{t}{8}}B_2^r+X):\,\,\, \|A(X-X')\|_2> t/4\}$ using relation \eqref{eq: ax pohas;fa} and conditional independence;
specifically,
$$\Prob\Big\{\forall  X'\in S\cap \Big({\frac{t}{8}}B_2^r+X\Big):\,\,\, \|A(X-X')\|_2> t/4\;\big|\;\Event\Big\}
\leq 2^{-2(C/48)^r}.$$
Finally, combining the estimates and taking the union bound, we obtain
$$p_0
\leq |\widetilde \Net|\,|\Net_{HS}|\,\big(2^{-2(C/48)^r}+\exp(-e^{r})\big)<1,$$
and the result follows.
\end{proof}

\medskip

As a consequence of the above lemmas, we obtain the main statement of the section.
\begin{prop}\label{prop: smin}
For any $\mathcal M,K\geq 1$ there is $C_{\text{\tiny\ref{prop: smin}}}>1$  depending only on $K,\mathcal M$,
with the following property.
Let $r\geq C_{\text{\tiny\ref{prop: smin}}}$, $m\leq 2^r$, $d\ge r$. Further, let
$M$ be an $m\times r$ random matrix, where each column $\col_\ell(M)$
is the orthogonal projection of a random isotropic vector $X_\ell$ in $\R^m$ with i.i.d $K$--subgaussian coordinates,
onto a $d$--dimensional fixed subspace $F$; $X_1,\dots,X_r$ are mutually independent.
Assume that for any unit vector $a=(a_1,\dots,a_r)\in\R^r$, the linear combination
$\sum_{\ell=1}^r a_\ell\col_\ell(M)$ satisfies
$$\cf\Big(\sum_{\ell=1}^r a_\ell\col_\ell(M),\sqrt{d}t\Big)\leq (\delta t)^d+\nu,\quad\quad t>0,$$
for some numbers $\delta\geq 1$ and $\nu>0$. Then for all $t\geq (e^{-c_{\text{\tiny\ref{l: HS conc}}}r}+{\nu^{1/d}})/{\delta}$
and any fixed $m\times r$ matrix $W$ with $\|W\|_{HS}\leq\mathcal M\sqrt{rd}$
we have
$$\Prob\big\{s_{\min}(M+W)\leq C_{\text{\tiny\ref{l: HS conc}}}\sqrt{d}t\big\}
\leq  {C^{r+d}_{\text{\tiny\ref{prop: smin}}}}\,\delta^d t^{d-r}.$$
\end{prop}

\begin{proof}
%

Fix any $t\geq (e^{-c_{\text{\tiny\ref{l: HS conc}}}r}+{\nu^{1/d}})/{\delta}$. By Lemma~\ref{l: HS net}, there exists a discrete set $\Net\subset\frac{3}{2}B_2^r\setminus \frac{1}{2}B_2^r$, $|\Net|\le(C_{\text{\tiny\ref{l: HS net}}}/t)^r$,  such that for every $X\in S^{r-1}$ there exists $X'=(x_1',\dots,x_r')\in\Net$ satisfying
 $\|(M+W)(X-X')\|_2\leq t\|M+W\|_{HS}/\sqrt{r}$. Hence,
\begin{align*}
 \Prob\big\{s_{\min}(M+W)\leq C_{\text{\tiny\ref{l: HS conc}}}\sqrt{d}t\big\}
 &=\Prob\big\{\min\limits_{X\in S^{r-1}}\|(M+W)X\|_2\leq C_{\text{\tiny\ref{l: HS conc}}}\sqrt{d}t\big\}
 \\
&\leq  \Prob\big\{\min\limits_{X'\in\Net}\|(M+W)X'\|_2\leq  C_{\text{\tiny\ref{l: HS conc}}}\sqrt{d}t+t \|M+W\|_{HS}/\sqrt{r}\big\}.
\end{align*}
Also,   since $\|X'\|_2\geq 1/2$, by the assumption of the Proposition we have
$$
\cf\Big(\sum_{\ell=1}^r x_\ell'\col_\ell(M),(2C_{\text{\tiny\ref{l: HS conc}}}+\mathcal M)\sqrt{d}t\Big)\leq
\big((4C_{\text{\tiny\ref{l: HS conc}}}+2\mathcal M)\delta t\big)^d+\nu.
$$
Taking the union bound over all $X'\in\Net$ and applying Lemma~\ref{l: HS conc},
we get

\begin{align*}
 \Prob\big\{s_{\min}(M+W)\leq C_{\text{\tiny\ref{l: HS conc}}}&\sqrt{d}t\big\}\leq  \Prob\big\{\|M\|_{HS}\ge C_{\text{\tiny\ref{l: HS conc}}}\sqrt{rd} \big\}
 \\
&+ \Prob\big\{\min\limits_{X'\in\Net}\|(M+W)X'\|_2\leq (2C_{\text{\tiny\ref{l: HS conc}}}+\mathcal M)\sqrt{d}t\,\,\vert\,\,
\|M\|_{HS}\le C_{\text{\tiny\ref{l: HS conc}}}\sqrt{rd} \big\}
\\
&\leq (C_{\text{\tiny\ref{l: HS net}}}/t)^r\big[ \big((4C_{\text{\tiny\ref{l: HS conc}}}+2\mathcal M)\delta t\big)^d+\nu\big]+e^{-c_{\text{\tiny\ref{l: HS conc}}}rd}.
\end{align*}
The result follows.
\end{proof}

\begin{rem}
Let us compare the above estimate with the standard $\varepsilon$--net argument,
involving the spectral norm of the matrix.
Assume for concreteness that $d=2r$ and $\nu=0$.
The standard $\varepsilon$--net argument then gives
$$\Prob\big\{s_{\min}(M)\leq C't\sqrt{d}\big\}
\leq \bigg(\frac{3}{t'}\bigg)^r\sup\limits_{X\in S^{r-1}}\Prob\big\{\|MX\|_2\leq 2C't\sqrt{d}\big\}
+\Prob\big\{\|M\|\geq C't\sqrt{d}/t'\big\},$$
where we optimize over $t'\in(0,t]$ --- the parameter of the net on $S^{r-1}$.
Under the assumptions on $M$,
the best possible estimate for $\|M\|$ is $\Prob\{\|M\|\geq \tau\}\leq e^{-c\tau^2}$ for all $\tau\geq C\sqrt{d}$
(it is not difficult to see that the reverse estimate holds with a different constant, say, when $X_1,\dots,X_r$ are Gaussians).
For all $t\leq \frac{1}{\widetilde C\delta^2}$ the optimization over $t'$ then gives
$$\Prob\big\{s_{\min}(M)\leq C't\sqrt{d}\big\}\leq \bar C^d \delta^d t^{d-r}\big(\sqrt{\log(t^{-1}\delta^{-2})}\big)^{r},$$
producing the extra logarithmic factor.
\end{rem}

\section{No-gaps delocalization of null vectors in the real setting. Proof of Theorem A}

In this section,
we consider delocalization of almost null vectors for rectangular matrices with i.i.d subgaussian entries.
As we have discussed in the introduction, our goal is not only to get optimal bounds for
the smallest coordinates of the vectors but also to derive optimal deviation estimates.
We combine the well known technique of test projections used in \cite{NV} with additional ingredients:
an efficient averaging procedure and small ball probability estimates for the smallest singular
values of matrices of projections, which were considered in the previous section.

\begin{lemma}
\label{l:determ}
Given $\tau\geq 0$ and $m\leq n$, let $B$ be an $m\times n$ matrix with entries in $\CorR$,
and let ${\bf u}$ be a unit vector
in $\CorR^n$ such that $\|B{\bf u}\|_2\leq \tau$.
Given $\beta> \theta>0$, take any two sets
\begin{align*}
 I_\theta\subset \big\{i\leq n:\;|u_i|\leq\theta\big\}
 \quad\text{and}\quad
J_\beta\subset\big\{i\leq n:\;|u_i|\geq\beta\big\}
\end{align*}
and assume that $k:=|I_\theta|\neq 0$ and $r:=|J_\beta|\neq 0$.
 Let
 $$
 F:=\spn\{\col_i(B):\, i\in (I_\theta\cup J_\beta)^c\}^\perp.
 $$
Define an $m\times k$ matrix $M$ and an $m\times r$ matrix $M'$ as
 $$
M:=\big(\Proj_F(\col_{i}(B))\big)_{i\in I_\theta} \quad\text{and}\quad M':=\big(\Proj_F(\col_{j}(B))\big)_{j\in J_\beta}.
 $$
Then we have
$$\beta\sqrt{r}\,s_{\min}(M')\leq \theta\sqrt{k}\,s_{\max}(M)+\tau.$$
\end{lemma}
\begin{proof}
Denote by $E_1$ the ellipsoid
$$E_1:=\Big\{\sum\limits_{i\in I_\theta} a_i\Proj_F(\col_{i}(B)):\;
\|(a_i)_{i\in I_\theta}\|_2\leq  \theta\sqrt{k}\Big\},$$
and by $E_2$ --- the ellipsoid (hypersurface)
$$E_2:=\Big\{\sum\limits_{i\in J_\beta} a_i'\Proj_F(\col_{i}(B)):\;
\|(a'_i)_{i\in J_\beta}\|_2= \beta\sqrt{r}\Big\}.$$
Then $ E_2\cap(E_1+\tau B_2^m(\CorR)\cap F)$ is non-empty. Indeed, setting
$$
{\bf v}:=\frac{\beta\sqrt{r}}{\|(u_j)_{j\in J_\beta}\|_2}
\sum_{j\in J_\beta} u_j \Proj_F(\col_{j}(B)),
$$
by construction we get ${\bf v}\in E_2$. On the other hand,
since $\beta\sqrt{r}/\|(u_j)_{j\in J_\beta}\|_2\le 1$ and
$$
\sum_{j\in J_\beta} u_j \Proj_F(\col_{j}(B))+\sum\limits_{i\in I_\theta} u_{i} \Proj_F(\col_{i}(B))=
\Proj_FB{\bf u},$$
where $\Proj_FB{\bf u}\in \tau B_2^m(\CorR)\cap F$ and
$\sum\limits_{i\in I_\theta} u_{i} \Proj_F(\col_{i}(B))\in E_1$, we have ${\bf v}\in (E_1+\tau B_2^m(\CorR)\cap F)$.

Now, since $ E_2\cap(E_1+\tau B_2^m(\CorR)\cap F)$ is non-empty, there exist vectors
$${\bf a}\in \theta\sqrt{k} B_2^k(\CorR) \mbox{ and }{\bf a}'\in \beta\sqrt{r} S^{r-1}(\CorR)\,
\mbox{ such that }\, \|M'{\bf a}'-M{\bf a}\|_2\le \tau,$$
and it remains to note that the l.h.s. is at least $\beta\sqrt{r}\,s_{\min}(M')- \theta\sqrt{k}\,s_{\max}(M)$.
\end{proof}

 In what follows, this lemma allows to reduce the problem of estimating coordinates of ${\bf u}$ to
comparing largest and smallest singular values of auxiliary random matrices.
The probabilistic estimate on the singular values
is obtained as a combination of Proposition~\ref{prop: smin}
and structural results of \cite{RV, RV smallest} stated in Section~\ref{s:prelim}.

\begin{prop}\label{cor: smin smax}
Let $m,d\in\N$ and let $F$ be a subspace of $\R^m$ of dimension $d$.
Further, let
$k,r<d$, and let $X_1,\dots,X_r,Y_1,\dots,Y_k$ be i.i.d random vectors in $\R^m$
satisfying \eqref{Asmp 1} with a parameter $K$. Let $M$ and $M'$ be $m\times r$ and $m\times k$ random matrices with
columns $\Proj_F(X_\ell)$, $\ell\leq r$ and $\Proj_F(Y_\ell)$, $\ell\leq k$, respectively.
Finally, assume that $W$ and $W'$ are fixed $m\times r$ and $m\times k$ matrices
with spectral norms at most $\mathcal M\sqrt{d}$.
Then for any $\varepsilon_1,\varepsilon_2>0$, $\tau\in[0,\varepsilon_2\sqrt{d}]$ and $\alpha>0$, $\gamma\in(0,1)$ we have
\begin{align*}
\Prob&\big\{\varepsilon_1 s_{\min}(M+W)\leq \varepsilon_2 s_{\max}(M'+W')+\tau\big\}\\
&\leq (C'_{\text{\tiny\ref{cor: smin smax}}}/\gamma)^{d}
\bigg(\frac{\sqrt{d}}{\LCD_{\alpha,\gamma}(F)}+
\gamma e^{-c'_{\text{\tiny\ref{cor: smin smax}}}\alpha^2/d}+e^{-c_{\text{\tiny\ref{l: HS conc}}}r}
+\frac{\varepsilon_2}{\varepsilon_1}\bigg)^{d-r},
\end{align*}
where $C'_{\text{\tiny\ref{cor: smin smax}}},c'_{\text{\tiny\ref{cor: smin smax}}}>0$ may  depend only on $K,\mathcal M$. In particular, if
$\LCD_{\gamma\sqrt{m},\gamma}(F)\ge \sqrt{m}e^{\gamma m/d}$, then the r.h.s. is less than
$
C_{\text{\tiny\ref{cor: smin smax}}}^{d}
\big( e^{-c_{\text{\tiny\ref{cor: smin smax}}}m/d}+e^{-c_{\text{\tiny\ref{l: HS conc}}}r}
+{\varepsilon_2}/{\varepsilon_1}\big)^{d-r},$
where $C_{\text{\tiny\ref{cor: smin smax}}},c_{\text{\tiny\ref{cor: smin smax}}}>0$ depend only on $K,\mathcal M$ and $\gamma$.
\end{prop}
\begin{proof}
Fix for a moment any $(a_1,\dots,a_r)\in S^{r-1}(\R)$,
and let
$$Z:=\sum\limits_{\ell=1}^r a_\ell \Proj_F(X_\ell)=\Proj_F\Big(\sum\limits_{\ell=1}^r a_\ell X_\ell\Big),\quad
\mbox{so that}\quad \|Z\|_2=\d\Big(\sum\limits_{\ell=1}^r a_\ell X_\ell,\,F^\perp\Big).$$
Clearly, $\sum_{\ell=1}^r a_\ell X_\ell$ is an isotropic random vector in $\R^m$ with i.i.d coordinates.
Moreover, each coordinate is $CK$--subgaussian, with $C>0$ being a universal constant
(see, for example, \cite[Lemma~5.9]{V 12}).
Now, by Theorem~\ref{th: 4.2}, we have for any $\alpha>0$, $\gamma\in (0,1)$
and for all $t\ge \sqrt{d}/\LCD_{\alpha,\gamma}(F)$:
$$
\cf\big(Z, \sqrt{d}t\big)\le \big({\widetilde C t}/{\gamma}\big)^{d}+\widetilde C^{d} e^{-c\alpha^2},
$$
where $\widetilde C,c>0$ depend only on the subgaussian moment $K$. Setting
$$
t_0:=\sqrt{d}/\LCD_{\alpha,\gamma}(F)+(e^{-c_{\text{\tiny\ref{l: HS conc}}}r}+\widetilde C e^{-c\alpha^2/d})\gamma/\widetilde C
$$
and applying Proposition~\ref{prop: smin} with $\delta =\widetilde C/\gamma$ and $\nu=\widetilde C^{d} e^{-c\alpha^2}$, we get
$$
\Prob\big\{s_{\min}(M+W)\leq C_{\text{\tiny\ref{l: HS conc}}}\sqrt{d}t\big\}
\leq  C^{r+d}_{\text{\tiny\ref{prop: smin}}}\,\delta^d t^{d-r},
\quad t\geq t_0.
$$
Further, it follows from the Hanson--Wright inequality (see Theorem~3.2 of \cite{RV HS}) that
\begin{equation}\label{HR}
\Prob\big\{s_{\max}(M'+W')\geq C'' h\sqrt{d+k}\big\}\leq e^{-h^2(d+k)},\quad h\geq 1,
\end{equation}
where $C''\geq 1$ may only depend on $K,\mathcal M$.

Note that for any two independent real-valued random variables $\xi_1$, $\xi_2$,
\begin{align*}
\Prob\{\xi_1\le \xi_2\}&\le 
\Prob\{\xi_1\le a\}
+\sum_{\ell=1}^\infty \Prob\{\xi_1\le 2^\ell a\}\,\Prob\{\xi_2\ge 2^{\ell-1} a\},\quad \forall a\in\R.
\end{align*}
Using this relation together with the inequality
$\tau \le \varepsilon_2\sqrt{d}$, we get 
\begin{align*}
&\Prob\big\{\varepsilon_1 s_{\min}(M+W)\leq \varepsilon_2 s_{\max}(M'+W')+\tau\big\}
\leq \Prob\big\{s_{\min}(M+W)\leq C_{\text{\tiny\ref{l: HS conc}}}\sqrt{d}t_0\big\}\\
&+\sum\limits_{\ell=1}^{\infty}\Prob\{s_{\min}(M+W)\leq C_{\text{\tiny\ref{l: HS conc}}}\sqrt{d}\,2^\ell t_0\}\,
\Prob\{s_{\max}(M'+W')\geq (C_{\text{\tiny\ref{l: HS conc}}}(\varepsilon_1/\varepsilon_2)2^{\ell-1}t_0-1)\sqrt{d}\}.
\end{align*}
Let $\ell_0>0$ be the smallest positive integer satisfying 
$
C_{\text{\tiny\ref{l: HS conc}}}(\varepsilon_1/\varepsilon_2)2^{\ell_0}t_0\sqrt{d}
\geq 2C'' \sqrt{d+k},
$
so that 
$$
2^{\ell_0-1}t_0< C(t_0+\varepsilon_2/\varepsilon_1),
$$
where $C>0$ is an absolute constant. Applying the above estimates for $s_{\min}$ and $s_{\max}$ and using a trivial probability bound for $s_{\max}$ for all $\ell\leq \ell_0$ (this corresponds to $h<1$ in (\ref{HR})), we get
\begin{align*}
&\Prob\big\{\varepsilon_1 s_{\min}(M+W)\leq \varepsilon_2 s_{\max}(M'+W')+\tau\big\}
\\
&\le\sum\limits_{\ell=0}^{\ell_0}
C^{r+d}_{\text{\tiny\ref{prop: smin}}}\,\delta^d (2^\ell t_0)^{d-r}+\sum\limits_{\ell=\ell_0+1}^{\infty}
C^{r+d}_{\text{\tiny\ref{prop: smin}}}\,\delta^d (2^\ell t_0)^{d-r}
\exp\big(-(C_{\text{\tiny\ref{l: HS conc}}}(\varepsilon_1/\varepsilon_2)2^{\ell-1}t_0-1)^2 d(C'')^{-2}\big)
\\
&\leq 2C^{r+d}_{\text{\tiny\ref{prop: smin}}}\,\delta^d (2^{\ell_0} t_0)^{d-r}
+2C^{r+d}_{\text{\tiny\ref{prop: smin}}}\,\delta^d (2^{\ell_0+1} t_0)^{d-r}
\exp\big(-(C_{\text{\tiny\ref{l: HS conc}}}(\varepsilon_1/\varepsilon_2)2^{\ell_0}t_0-1)^2 d(C'')^{-2}\big)\\
&\leq 4C^{r+d}_{\text{\tiny\ref{prop: smin}}}\,\delta^d (2^{\ell_0+1} t_0)^{d-r}
\\
&\leq 4C^{r+d}_{\text{\tiny\ref{prop: smin}}}\,\delta^d (4C(t_0+\varepsilon_2/\varepsilon_1))^{d-r},
\end{align*}
where to estimate the value of the series we used that, by our assumption on $\ell_0$,
$$\frac{2^{d-r}
\exp(-(C_{\text{\tiny\ref{l: HS conc}}}(\varepsilon_1/\varepsilon_2)2^{\ell+1}t_0-1)^2 d(C'')^{-2})}
{\exp(-(C_{\text{\tiny\ref{l: HS conc}}}(\varepsilon_1/\varepsilon_2)2^{\ell}t_0-1)^2 d(C'')^{-2})}\leq \frac{1}{2},\quad \ell\geq \ell_0+1.
$$
The result follows.

\end{proof}

Now, we can prove the main result of this section, which gives Theorem~A
from the introduction for an appropriate choice of parameters.

\begin{theor}\label{th: normal lower}
For any $K\geq 1$
there are $C_{\text{\tiny\ref{th: normal lower}}},c_{\text{\tiny\ref{th: normal lower}}}>0$
depending only on $K$ with the following property.
Let $n\geq C_{\text{\tiny\ref{th: normal lower}}}$, let $1\le n-m\leq {n}/{\log n}$, and let
$B$ be an $m\times n$ random matrix with independent columns satisfying \eqref{Asmp 1}.
Then for any
$$
C_{\text{\tiny\ref{th: normal lower}}}(n-m)\log n\le k\le c_{\text{\tiny\ref{th: normal lower}}}n,\quad
e^{-c_{\text{\tiny\ref{th: normal lower}}} k/(n-m)}\le t\le 1,\quad 0\le\tau\le k^2t/n^{3/2}
$$
and any $m\times n$ fixed matrix $W$ with $\|W\|\leq \sqrt{k}$ we have
$$\Prob\Big\{\exists {\bf u}\in S^{n-1}(\R):\;
\|(B+W){\bf u}\|_2\leq \tau\mbox{ and }u_{n-k+1}^*\leq \frac{kt}{n^{3/2}}\Big\}
\leq (C_{\text{\tiny\ref{th: normal lower}}}t)^{k}+e^{-c_{\text{\tiny\ref{th: normal lower}}}n}.$$
\end{theor}
\begin{proof}
The constants $C_{\text{\tiny\ref{th: normal lower}}},c_{\text{\tiny\ref{th: normal lower}}}$ can be recovered from the proof below.
Given $n\geq C_{\text{\tiny\ref{th: normal lower}}}$ and $1\le n-m\leq {n}/{\log n}$, take any $k,t,\tau$ satisfying
$$
k\in[ C_{\text{\tiny\ref{th: normal lower}}}(n-m)\log n,c_{\text{\tiny\ref{th: normal lower}}}n],\quad t\le 1,\quad\text{and}\quad \tau\le k^2t/n^{3/2}\le c_{\text{\tiny\ref{th: normal lower}}}^2\sqrt{n},
$$
and a fixed matrix $W$ with $\|W\|\leq \sqrt{k}$.
Let ${\bf u}=(u_1,\dots,u_n)$ be a unit random vector such that $\|(B+W){\bf u}\|_2\leq\tau$
everywhere on the probability space. It follows from \cite[Lemma~2.6]{RV smallest}
that there is $\widetilde c=\widetilde c(K)>0$  such that if
$c_{\text{\tiny\ref{th: normal lower}}}\leq \widetilde c$,
$$
\beta:=\widetilde c/\sqrt{n}\quad\text{and}\quad J_\beta:=\big\{i\leq n:\;|u_i|\geq\beta\big\}
$$
then $\Prob\{|J_\beta|<\widetilde c n\}<e^{-\widetilde c n}$.
Set
$$
\theta:={kt}/{n^{3/2}} \quad\text{and}\quad I_\theta:=\{i\leq n:\;|u_i|\leq \theta\}
$$
 and note that
$$
u_{n-k+1}^*\le \theta\quad\text{if and only if}\quad |I_{\theta}|\ge k.
$$
Thus we need to estimate probability of the event $\big\{|I_\theta|\geq k\big\}$. Using  Markov's inequality, we get
\begin{align}
 \Prob\big\{|I_\theta|\geq k\big\}&\le k^{-k}\Exp(|I_\theta|^{k}{\bf 1}_{\{|I_\theta|\ge k\}})\notag
 \\
 &\le  (\widetilde{c}nk)^{-k}\Exp\big(|I_\theta|^{k}|J_\beta|^{k}{\bf 1}_{\{|I_\theta|\ge k\}}
{\bf 1}_{\{|J_\beta|\geq \widetilde c n\}}\big)+(n/k)^{k}e^{-\widetilde{c}n}, \label{Markov}
\end{align}
and we need to get a bound for $\Exp\big(|I_\theta|^{k}|J_\beta|^{k}{\bf 1}_{\{|I_\theta|\ge k\}}{\bf 1}_{\{|J_\beta|\geq \widetilde c n\}}\big)$.

Let $\Event_{rk}$ be the event that all $m\times m$ submatrices of $B+W$ are of full rank.
The main result of \cite{RV} implies that, as long as $n-m\leq c'n$
for a sufficiently small $c'>0$ (depending on $K$), we have
$$
\Prob(\Event_{rk})\geq 1-e^{-c' n}
$$
(as we already mentioned, \cite{RV} deals with centered random matrices, however, adding
a non-random shift with an appropriately bounded spectral norm does not in any way change the argument).
Given any two distinct subsets of indices  $I:=\{i_1,\dots, i_{k}\}$, $J:=\{j_1,\dots,j_k\}\subset [n]$,  denote
$$
F=F(I,J):=\spn\{\col_{i}(B+W),\, i\in(I\cup J)^c\}^\perp.
$$
Conditioned on $\Event_{rk}$, we have
\begin{equation}\label{dimF}
\dim F=m-(n-2k)=:d\in (k,2k).
\end{equation}
Further, assuming that $c_{\text{\tiny\ref{th: normal lower}}}\le c_{\text{\tiny\ref{th: 4.3}}}/2$, by Theorem~\ref{th: 4.3} we get
$$
\Prob\big\{\LCD_{c_{\text{\tiny\ref{th: 4.3}}} \sqrt{m},c_{\text{\tiny\ref{th: 4.3}}}}(F)\le \sqrt{m}\,e^{c_{\text{\tiny\ref{th: 4.3}}} m/d}\big\}\le e^{-c_{\text{\tiny\ref{th: 4.3}}} m}.
$$
Let $M$, $M'$ be the $m\times k$ matrices with columns $\Proj_F(\col_{j}(B+W))$, $j\in J$ and
 $\Proj_F(\col_{i}(B+W))$, $i\in I$, respectively.
  Set
$$
\varepsilon_1:=\beta\sqrt{k}=\widetilde c\sqrt{k/n},\quad \varepsilon_2:=\theta\sqrt{k}=t(k/n)^{3/2}.
$$
Note that with this choice,  $\varepsilon_1>\varepsilon_2\geq \tau/\sqrt{d}$. By Proposition~\ref{cor: smin smax} we have
\begin{align}
\Prob\Big\{&\varepsilon_1 s_{\min}(M)\leq \varepsilon_2 s_{\max}(M')+\tau\;\big|\;\{\dim F=d\}\cap\big\{\LCD_{c_{\text{\tiny\ref{th: 4.3}}} \sqrt{m},c_{\text{\tiny\ref{th: 4.3}}}}(F)\geq \sqrt{m}\,e^{c_{\text{\tiny\ref{th: 4.3}}} m/d}\big\}\Big\}\notag
\\
&\leq C_{\text{\tiny\ref{cor: smin smax}}}^{d}\big(e^{-cn/k}+e^{-ck}+{\varepsilon_2}/{\varepsilon_1}\big)^{d-k},\label{aux}
\end{align}
where $C_{\text{\tiny\ref{cor: smin smax}}}$ and $c=\min\{c_{\text{\tiny\ref{cor: smin smax}}}/2,c_{\text{\tiny\ref{l: HS conc}}},c'\}$
depend on $K$ and $c_{\text{\tiny\ref{th: 4.3}}}$, and we used that $m/d\ge n/(2k)$.
Let $\chi_{j_1,\dots,j_k,i_1,\dots,i_k}(\varepsilon_1,\varepsilon_2)$
be the indicator of the event
$$\varepsilon_1 s_{\min}(M)\leq \varepsilon_2 s_{\max}(M')+\tau.$$
It follows from (\ref{dimF}) -- (\ref{aux}) that for all $\varepsilon_2/\varepsilon_1\geq e^{-c n/k}+e^{-c k}$
we have
\begin{align*}
\Exp(\chi_{j_1,\dots,j_k,i_1,\dots,i_k}(\varepsilon_1,\varepsilon_2))&\leq
 C_{\text{\tiny\ref{cor: smin smax}}}^d\big(2{\varepsilon_2}/{\varepsilon_1}\big)^{d-k}+e^{-c'n}+e^{-c_{\text{\tiny\ref{th: 4.3}}} m}
\leq C^d (\varepsilon_2/\varepsilon_1)^{d-k}
\end{align*}
with $C>2C_{\text{\tiny\ref{cor: smin smax}}}$, so that
\begin{equation}\label{eq:sumchi}
\Exp\Big(\sum\chi_{j_1,\dots,j_k,i_1,\dots,i_k}
(\varepsilon_1,\varepsilon_2)\Big) \le C^d n^{2k} (\varepsilon_2/\varepsilon_1)^{d-k},
\end{equation}
where the sum is taken over all ordered $2k$--tuples $(j_1,\dots,j_k,i_1,\dots,i_k)$ with distinct components.
Now for any distinct $i_1,\dots,i_k\in I_\theta$ and $j_1,\dots,j_k\in J_\beta$
(whenever $|I_\theta|,|J_\beta|\geq k$)
by Lemma~\ref{l:determ} we have
$\chi_{j_1,\dots,j_k,i_1,\dots,i_k}(\varepsilon_1,\varepsilon_2)=1$, hence {\it deterministically}
\begin{align}
\label{determ}
&\sum
\chi_{j_1,\dots,j_k,i_1,\dots,i_k}(\varepsilon_1,\varepsilon_2)\\
&\hspace{1cm}\ge |I_\theta|(|I_\theta|-1)\dots(|I_\theta|-k+1)
|J_\beta|(|J_\beta|-1)\dots(|J_\beta|-k+1)
{\bf 1}_{\{|I_\theta|\ge k\}}{\bf 1}_{\{|J_\beta|\geq \widetilde c n\}}\notag
\\
&\hspace{1cm}\ge 9^{-k}|I_\theta|^{k}|J_\beta|^k{\bf 1}_{\{|I_\theta|\ge k\}}{\bf 1}_{\{|J_\beta|\geq \widetilde c n\}}.\notag
\end{align}
Let
$
t_0:=\widetilde cnk^{-1}(e^{-c n/k}+e^{-c k/(n-m)}).
$
Note that $t_0<1$ provided that $c^{-1}_{\text{\tiny\ref{th: normal lower}}},C_{\text{\tiny\ref{th: normal lower}}}> c^{-1}$ are big enough.
For $t\ge t_0$ we have $(\widetilde cn/(kt))^{n-m}\le e^{ck}$ and
$$
\frac{\varepsilon_2}{\varepsilon_1}=\frac{kt}{\widetilde cn}\ge e^{-c n/k}+e^{-c k/(n-m)}\geq e^{-c n/k}+e^{-c k},
$$
which together with (\ref{eq:sumchi}) -- (\ref{determ}) yield
$$
\Exp(|I_\theta|^{k}|J_\beta|^k{\bf 1}_{\{|I_\theta|\ge k\}}{\bf 1}_{\{|J_\beta|\geq \widetilde c n\}})
\leq 9^k n^{2k}C^d (\varepsilon_2/\varepsilon_1)^{d-k}
\le
(9C^2/\widetilde c)^k n^{3k-d}(kt)^{d-k}.
$$
Hence, plugging this in \eqref{Markov}, we get for $t\ge t_0$
\begin{align*}
 \Prob\big\{|I_\theta|\geq k\big\}&\leq (3C/\widetilde c)^{2k} (n/k)^{2k-d}t^{d-k}+(n/k)^{k}e^{-\widetilde{c}n}
 \\
 &= \Big( (3C/\widetilde c)^{2k}({n}/(kt))^{n-m}+(n/(kt))^{k}e^{-\widetilde{c}n}\Big)t^k\le (\widehat{C}t)^k
 \end{align*}
for some $\widehat{C}>1$.
Thus, for any $t\geq t_0$
and $\tau\leq {k^2 t}/{n^{3/2}}$, any unit random vector ${\bf u}$ with $\Prob\{\|(B+W){\bf u}\|_2\leq \tau\}=1$
satisfies
$$\Prob\big\{u^*_{n-k+1}\leq {kt}/{n^{3/2}}\big\}\leq (\widehat C t)^{k}.$$
 Also, $\Prob\big\{u^*_{n-k+1}\leq {kt}/{n^{3/2}}\big\}\leq (\widehat C t_0)^{k}$ for $t\le t_0$.
Note that there exists $\overline{c}>0$ such that
$$
t_0\le\max(e^{-\overline{c} n/k},e^{-\overline{c} k/(n-m)}).
$$
Hence, for all $t\geq e^{-\bar c k/(n-m)}$ we get from the above
$$
\Prob\big\{u^*_{n-k+1}\leq {kt}/{n^{3/2}}\big\}\leq (\widehat C \max(t,t_0))^k\le
(\widehat C \max(t,e^{-\bar c n/k}))^{k}\leq (\widehat C t)^{k}+\widehat C^k e^{-\bar c n}.
$$
The result follows.
\end{proof}

As an illustration of the above result,
we consider no-gaps delocalization of eigenvectors of non-Hermitian
random matrices corresponding to real eigenvalues of {\it small absolute value}.
\begin{cor}\label{cor: eigenv}
For any $K\geq 1$ there are $C_{\text{\tiny\ref{cor: eigenv}}},c_{\text{\tiny\ref{cor: eigenv}}}>0$
depending only on $K$ with the following property.
Let $n\geq C_{\text{\tiny\ref{cor: eigenv}}}$, let
$B$ be an $n\times n$ random matrix with independent columns satisfying \eqref{Asmp 1}.
Then for any integer $k\in [C_{\text{\tiny\ref{cor: eigenv}}}\log n,\,c_{\text{\tiny\ref{cor: eigenv}}}n]$ we have
\begin{align*}
\Prob\Big\{&\exists\, \n\in S^{n-1}(\R):\;
B\n=\lambda \n\mbox{ for some }\;\;-\sqrt{k}\leq \lambda\leq \sqrt{k}
\;\;\mbox{ and }v_{n-k+1}^*\leq \frac{c_{\text{\tiny\ref{cor: eigenv}}}
k}{n^{3/2}}\Big\}\leq e^{-k}.
\end{align*}
\end{cor}
\begin{proof}
  Applying a simple discretization procedure for the interval $[-\sqrt{k},\sqrt{k}]$, we get that the probability under consideration is less then
  \begin{align*}
c_{\text{\tiny\ref{cor: eigenv}}}^{-1}(n/k)^{3/2}\Prob\Big\{&\exists\, \n\in S^{n-1}(\R):\;
\|B\n-\lambda_0 \n\|_2\le \frac{c_{\text{\tiny\ref{cor: eigenv}}}k^2}{n^{3/2}}
\;\;\mbox{ and }v_{n-k+1}^*\leq \frac{c_{\text{\tiny\ref{cor: eigenv}}}k}{n^{3/2}}\Big\}
\end{align*}
for some $-\sqrt{k}\leq \lambda_0\leq \sqrt{k}$. Evidently, here $B$ can be replaced with any its submatrix of the size $(n-1)\times n$ without increasing the probability. Hence, choosing $c_{\text{\tiny\ref{cor: eigenv}}}$ being sufficiently small and applying
Theorem~\ref{th: normal lower} with $m=n-1$ and $W=\lambda_0 I$, we get the result.
\end{proof}

As another simple corollary, let us consider no-gaps delocalization for singular vectors corresponding to
small singular values.
We note that sup-norm delocalization of singular vectors is well studied in literature;
see, in particular, section~8 of survey \cite{RVW} and references therein, as well as papers
\cite{TV cov, Maltsev13, PillaiYin14, BEKYY, BKYY, Vu Wang, Wang}.

For a square $n\times n$ matrix $B$ with i.i.d $K$--subgaussian entries, Wei \cite{Wei} showed that the $\ell$--th smallest
singular value $s_{n-\ell+1}(B)$ satisfies $s_{n-\ell+1}(B)\leq \frac{Ct\ell}{\sqrt{n}}$ with probability $1-e^{-c\ell t}$
for all $t\geq 1$, where $C,c>0$ may only depend on $K$ (we refer to \cite{NV,RV upper,Kat18} for upper bounds for $s_{\min}(B)$).
Combining this result with the above theorem, we immediately get
\begin{cor}\label{cor: sing}
For any $K\geq 1$ there are $C_{\text{\tiny\ref{cor: sing}}},c_{\text{\tiny\ref{cor: sing}}}>0$
depending only on $K$ with the following property.
Let $n\geq C_{\text{\tiny\ref{cor: sing}}}$ and let
$B$ be an $n\times n$ random matrix with independent columns satisfying \eqref{Asmp 1}.
Then for any $\ell\in[n]$ and  any $C_{\text{\tiny\ref{cor: sing}}}\sqrt{n\max(\ell,\log n)}\le k\le c_{\text{\tiny\ref{cor: sing}}}n$ we have
\begin{align*}
\Prob\Big\{v_{n-k+1}^*\leq \frac{c_{\text{\tiny\ref{cor: sing}}}
k}{n^{3/2}},\;\;\mbox{ for a unit vector $\n$ satisfying $B^T B\n=s_{n-\ell+1}(B)^2\n$}
\Big\}\leq n^{-1}.
\end{align*}
\end{cor}
\begin{rem}
The estimates in the above corollary are non-trivial for $\ell=o(n)$.
\end{rem}

\section{Eigenvectors of non-Hermitian matrices}\label{s: eigenvectors}

In this section, we study no-gaps delocalization
of eigenvectors of non-Hermitian random matrices with i.i.d entries.
The basic tool, as in the case of almost null vectors, is a {\it test projection} onto the orthogonal complement $F$
of $n-\q$ columns of the random matrix (where $\q$ is chosen polylogarithmic in $n$).
The major difficulty in working with eigenvectors rather than null vectors
of the matrix is the necessity to control magnitudes of projections of rescaled coordinate vectors $-z{\bf e}_i$.
When $z$ is small by absolute value, a trivial upper bound $\|\Proj_F(-z{\bf e}_i)\|_2\leq |z|$ for the norm of the projection is already sufficient
(see Corollary~\ref{cor: eigenv}).
However, when $|z|$ is of order $n^{1/2}$, such trivial estimate becomes useless.

A similar problem was considered earlier in \cite{RV sup norm}
where the sup-norm delocalization was studied.
The authors of \cite{RV sup norm} have developed a strategy based on comparing the magnitudes of projections
of $-z{\bf e}_i$
with each other rather than estimating their ``absolute'' magnitudes.
The bound on the $\ell_\infty$--norm of an eigenvector was reduced to estimating probabilities of the form
\begin{equation}\label{eq: aux apoiajf;lksafa}
\Prob\big\{\|P(\col_{\q}(A)-z{\bf e}_{\q})\|_2\leq \varepsilon\max\limits_{i\leq \q-1}\|P(\col_{i}(A)-z{\bf e}_{i})\|_2\big\},
\end{equation}
where $P$ was a specially constructed test projection with $\ker(P)\supset F$.
It was shown that, for $\q$ polylogarithmic in $n$ and for $\varepsilon=\log^{-C}(n)$ for
a sufficiently large constant $C>0$, the above probability is very close to zero
implying the upper bound $\frac{{\rm polylog}(n)}{\sqrt{n}}$ for the
$\ell_\infty$--norms of unit eigenvectors of $A$ \cite[Section~5]{RV sup norm}.

The upper bound for \eqref{eq: aux apoiajf;lksafa} obtained in \cite{RV sup norm}
with the above choice of $\varepsilon$ is $\exp(-c \q/\log n)$ \cite[Theorem~5.1]{RV sup norm}, which is sufficient
for $\ell_\infty$--delocalization.
However, studying no-gaps delocalization with the same method requires much stronger estimates,
as they need to be able to survive the union bound over a very large number of combinations
of coordinates.
In particular, in the real case, for any $\beta>\theta>0$, the probability of the event
$$\big\{\exists\,\n\in S^{n-1}(\R):\;\|(\Az )\n\|_2\leq n^{-2}\mbox{ and }|v_{i}|\leq\theta,\;i<\q;
|v_{\q}|\geq\beta\big\}$$
has to be bounded from above by $\big(\q ^{C }{\theta}/{\beta}\big)^{\q}$,
so that the averaging argument, similar to that in the proof of Theorem~\ref{th: normal lower},
would imply the desired lower bound for the order statistics of eigenvectors.
The procedure of {\it selecting spectral window} developed in \cite{RV sup norm} and applied to
bound \eqref{eq: aux apoiajf;lksafa}, does not seem applicable to get the stronger bounds needed in our context.

\subsection{
No-gaps delocalization in terms of test projections and ellipsoids
}\label{subs: tets}

An alternative procedure developed in this paper
is based on a careful analysis of the ellipsoid
$$
\Ellips':=\Big\{\sum_{i=1}^{\q-1} a_i\Proj_F(-z{\bf e}_i):\;\|(a_1,\dots,a_{\q-1})\|_2\leq 1\Big\}
$$
as well as the {\it dual} ellipsoid.
Rather than comparing the lengths of $\Proj_F(-z{\bf e}_i)$, we consider the geometric problem of estimating probability of
events of the form
$$\Proj_F(-z{\bf e}_{\q})\in\varepsilon\, \Ellips'+\tau B_2^F.$$
A simple but important observation is the following deterministic lemma:

\begin{lemma}\label{l: two conditions}
Given $n\in \N$, let $A$ be an $n\times n$ matrix with entries in $\CorR$, and let $z\in\CorR$. For $\q<n$, let
$$
F:=\spn\{\col_i(\Az ):\;i\in[n]\setminus[\q ]\}^\perp.
$$
Let $\n=(v_1,\dots,v_n)\in S^{n-1}(\CorR)$ and let $\tau>0$, $\beta>\theta>0$ be parameters such that
$$
\|(\Az )\n\|_2\leq \tau,\quad |v_i|\leq \theta, \,\,i\leq \q-1,\quad\text{and}\quad |v_{\q }|\geq \beta.
$$
Assume also that for some $T>0$ we have
$\|\Proj_F\col_\ell(A)\|_2\leq T$, $\ell\in[\q ]$.
Then
\begin{align*}
&(a)\quad\Proj_F(-z {\bf e}_{\q })\in
\frac{\theta\sqrt{\q}}{\beta}\Ellips'
+\Big(\frac{\tau}{\beta}+\frac{T\q \theta}{\beta}+T\Big) B_2^F  \quad \mbox{and}
\\
&(b)\quad\Proj_F\col_{\q }(\Az )\in
\frac{\theta\sqrt{\q}}{\beta}\Ellips'
+\Big(\frac{\tau}{\beta}+\frac{T\q \theta}{\beta}\Big) B_2^F.
\end{align*}
\end{lemma}
\begin{proof}
Let $X:=\Proj_F\col_{\q }(\Az )$.
Applying the argument from the proof of
Lemma~\ref{l:determ} with $B=\Az $, $J_\beta=\{\q\}$, $I_\theta=[\q-1]$,
we obtain
$$
X\in \frac{\theta\sqrt{\q}}{\beta}
\Big\{\sum_{\ell=1}^{\q-1} a_\ell \Proj_F\col_{\ell}(\Az ):\;\|(a_\ell)\|_2\leq 1\Big\}
+\frac{\tau}{\beta}B_2^F.
$$
On the other hand, by the conditions on vectors $\col_\ell(A)$, we have
\begin{align*}
&\Big\{\sum\limits_{\ell=1}^{\q-1} a_\ell\Proj_F\col_\ell(\Az ):\;\|(a_\ell)\|_2\leq 1\Big\}
\subset \Big\{\sum\limits_{\ell=1}^{\q-1} a_\ell\Proj_F(-z {\bf e}_\ell):\;\|(a_\ell)\|_2\leq 1\Big\}
+T\sqrt{\q}\,B_2^F.
\end{align*}
Together with the previous inclusion, this gives
\begin{align*}
X\in \frac{\theta\sqrt{\q}}{\beta}\Ellips'
+\Big(\frac{\tau}{\beta}+\frac{T\q \theta}{\beta}\Big) B_2^F,
\end{align*}
proving assertion (b).
To get (a) it remains to note (again) that $\|\Proj_F\col_{\q }(A)\|_2\leq T$.
\end{proof}

The above lemma reduces the proof of Theorems B and C
to computing the probability of conditions (a) and (b) on the test projections.
Probabilistic analysis of (a) is relatively harder, and constitutes most part of the section.
The condition (b) is considered in Subsection~\ref{subs: deloc estimates} (see discussion at the beginning of the subsection and
Lemma~\ref{l: aux aponfa;sa32r}).

\smallskip

In what follows, rather than working with the sequence $\big(\Proj_F(-z{\bf e}_{1 }),\dots,\Proj_F(-z{\bf e}_{\q })\big)$
we will consider a dual basis in $F$. We will need some definitions.
For brevity, we use capital calligraphic letters to denote sequences
of vectors of a given length, for example, $\U_\q=(U_1,\dots,U_\q)$. Correspondingly,
$\spn(\U_\q)=\spn\{U_1,\dots,U_\q\}$ is the linear span of $U_1,\dots,U_\q$.
Further, given a sequence of vectors $\U_\q$ in a Euclidean space $E$
(over $\CorR$) we use notation
\begin{equation}\label{eq: aux el def}
\Ellips(\U_\q):=
\Big\{\sum\limits_{i=1}^\q  a_i U_i:\;a_1,\dots,a_{\q}\in\CorR,\;\|(a_1,\dots,a_\q )\|_2\leq 1\Big\}
\end{equation}
for the ellipsoid generated by $\U_\q$.

\begin{definition}[Biorthogonal system]
We say that a pair $\U_{\q}$, $\Y_{\q}\in E^\q$
forms {\it a biorthogonal system} in $E$ if
\begin{equation*}
  \spn(\U_\q)=\spn(\Y_\q)=E \quad\text{and}\quad\langle U_i,Y_j\rangle=\delta_{ij},\,\, 1\leq i,j\leq \q.
\end{equation*}
\end{definition}
The next observation is essentially taken from \cite{RV sup norm};
we provide its proof for Reader's convenience:
\begin{lemma}[{\cite{RV sup norm}}]\label{l: V and proj biorth}
Let $A=(a_{ij})_{i,j}$, $z\neq 0$, and subspace $F$ be as in Lemma~\ref{l: two conditions}.
Assume that the $(n-\q )\times (n-\q )$ principal submatrix $\widetilde A$ of $A$
obtained by crossing out first $\q $ rows and $\q$ columns, is invertible. Let ${\bf e}'_{1},\dots,{\bf e}'_{\q}$ be the coordinate vectors in $\R^\q$. Define $(n-\q )$--dimensional vectors
$$
{\bf q}_i={\bf q}_i(A,\q):=(\bar a_{i\ell})_{\ell\in [n]\setminus [\q ]},\quad i\le \q,
$$
and $(n-\q )\times (n-\q )$ matrix $D:=\bar z^{-1}(\widetilde A^*-\overline{z})^{-1}$. Finally, construct  $n$--dimensional vectors
\begin{equation}\label{eq: V act def}
V_i=V_i(A,z,\q):=(-\bar z^{-1}{\bf e}'_i)\oplus D{\bf q}_i,\quad i\in [\q ],
\end{equation}
where the direct sum ``$\oplus$'' should be understood as a concatenation of $\q$--dimensional vectors
$-\bar z^{-1}{\bf e}'_i$ with $(n-\q )$--dimensional vectors $D{\bf q}_i$'s.
Then $V_i$, $i\in [\q ]$ form a basis of the space $F$ and, moreover,
the sequences $\Proj_F(-z{\bf e}_1),\dots,\Proj_F(-z{\bf e}_\q)$ and $V_1,\dots,V_\q$ form a biorthogonal system in $F$.
\end{lemma}
\begin{proof}
  With the notation introduced above, we have
  \begin{equation*}
    (A-z\Id)=\Big[\,\,\cdot\cdot\cdot\,\,\Big|\,\col_{\q+1}(A-z\Id)\,\dots\,\col_{n}(A-z\Id)\Big]=\left[
            \,\,\cdot\cdot\cdot\,\,
\Bigg|
\begin{array}{c}
		\,\,Q\,\,
        \\
        \,\,(\widetilde{A}-z{\rm I}_{n-\q})\,\,
        	\end{array}\right],
  \end{equation*}
  where
  \begin{equation*}
   Q=\left[
\begin{array}{ccc}
		a_{1\,\q+1}&\,\dots\,&a_{1\,n}
        \\
        \vdots&&\vdots
        \\
        a_{\q\,\q+1}&\,\dots\,&a_{\q\,n}
        	\end{array}\right]=
        \left[
        \begin{array}{c}
		{\bf q}_1
        \\
        \vdots
        \\
        {\bf q}_\q
        	\end{array}\right].
  \end{equation*}
  Clearly,
  \begin{equation*}
    \Big[\,\,{\rm I}_{\q}\,\,\big|\,-Q(\widetilde{A}-{z}{\rm I}_{n-\q})^{-1}\Big]\cdot\left[\begin{array}{c}
		\,\,Q\,\,
        \\
        \,\,(\widetilde{A}-z{\rm I}_{n-\q})\,\,
        	\end{array}\right]=0,
  \end{equation*}
  Hence, the columns of the matrix
  \begin{equation*}
    \Big[\,\,{\rm I}_{\q}\,\,\big|\,-Q(\widetilde{A}-{z}{\rm I}_{n-\q})^{-1}\Big]^*=\left[\begin{array}{c}
		\,\,{\rm I}_{\q}\,\,
        \\
        \,\,-(\widetilde{A}^*-\overline{z}{\rm I}_{n-\q})^{-1}Q^*\,\,
        	\end{array}\right]
  \end{equation*}
  are orthogonal to $\col_{i}(A-z{\rm I}_n)$, $i=\q+1,\dots,n$, so that multiplying by $-\overline{z}^{-1}$ we get
  $\langle V_j,\col_{i}(A-z{\rm I}_n)\rangle=0$, $j=1,\dots,\q$, $i=\q+1,\dots,n$.
  Hence, $V_1,\dots,V_{\q}\in F$ and
  $$\langle \Proj_F(-z{\bf e}_i),V_j\rangle=\langle {\bf e}_i,\,{\bf e}'_j\oplus -(\widetilde{A}^*-\overline{z}{\rm I}_{n-\q})^{-1}{\bf q}_j\rangle=\delta_{ij},\quad i,j\in [\q ].$$
  This finishes the proof of the lemma.
\end{proof}

In \cite{RV sup norm} it was observed that working with the dual basis $V_1,\dots,V_{\q}$
has significant advantages compared to the ``original'' basic sequence
$\Proj_F(-z{\bf e}_1),\dots,\Proj_F(-z{\bf e}_\q)$. Indeed, in our random model, upon conditioning on a realization of $D$,
vectors $V_1,\dots,V_{\q}$ are mutually independent and, moreover, are linear images of the random vectors
with i.i.d.\ components. On the other hand,
expressing the geometric condition on the test projections from Lemma~\ref{l: two conditions}
in terms of the dual sequence, and a decoupling procedure to estimate probabilities,
are quite non-trivial problems.
It will be convenient for us to study probabilities of events corresponding to (a) and (b) in Lemma~\ref{l: two conditions} separately.
This requires us to introduce a splitting of ellipsoids,
as well as possible realizations of the dual basis $V_1,\dots,V_{\q}$, into a finite number of classes,
which will be employed in a decoupling argument.

\smallskip

The text below is split into four parts.
In Subsection~\ref{subs: biorth}, we study the geometric condition (a) from Lemma~\ref{l: two conditions}
in an abstract (deterministic) setting.
More precisely, having in mind the decoupling procedure briefly mentioned above, we are
concerned with studying a condition of the form
$$
\Ellips(\U_{\q})\in\elclass(R,{\bf b})\quad\mbox{and}\quad U_{\q}\in \delta\,\Ellips(\U_{\q-1})+T B_2^E,
$$
where $\U_{\q}=(U_1,\dots,U_\q)$, $\Y_{\q}=(Y_1,\dots,Y_\q)$ is an arbitrary biorthogonal system
in a Euclidean space $E$,
and $\elclass(R,{\bf b})$ is a {\it class} (a subset) of ellipsoids parametrized by a number $R$ and a vector ${\bf b}$
(see Definition \ref{page: elclass}, as well as the beginning of Subsection~\ref{subs: deloc estimates}
where the splitting into classes is discussed in detail).
Our goal is to express the condition in terms of the dual sequence $\Y_\q$.
In fact, for technical reasons instead of $\Y_{\q}$
we work with its {\it $\delta$-perturbation} $\W_\q$, i.e.\ a sequence of vectors such that
$\|Y_i-W_i\|_2\le\delta$, $i\in[\q]$.
As the first step, we show that the above condition on $\U_\q$ implies that
$$\exists Y:\;\|Y\|_2\leq T'\mbox{ and
$|\langle W_\ell,Y\rangle|\leq \delta'$, $\ell\leq \q-1$; $|\langle W_{\q},Y\rangle|=1$}
$$
for certain $\delta'=\delta'(T,\delta)$ and $T'=T'(T)$
(see Corollary \ref{cor: wy}). Probabilistically, the latter condition is very difficult to work with.
To be able to apply a union bound argument later, we need to discretize the domain for $Y$ and represent the last condition
as a statement about a {\it finite} number of vectors.
For this reason, we define the set $\disccollec(\W_\q,\delta',T')$
(see Definition~\ref{page: disccollec}) which can only contain vectors having a special structure (at most
$(C\q)^{C\q}$ vectors) and show that the last condition implies that
$\disccollec(\W_\q,\delta',T')\neq\emptyset$.
This implication is verified in the main technical element of the section ---
Lemma~\ref{l: y reduction}, while Proposition~\ref{p: perturbed vector existence}
summarizes the results.


In Subsection~\ref{subs: two types} we
consider the results of the previous subsection from the probabilistic viewpoint.
The main goal there is to estimate probability of the event $\{\disccollec(\V_\q,\delta',T')\neq\emptyset\}$,
where  $\V_\q$ is a $\delta$--perturbation of ${\mathcal V}_\q$ (see \eqref{eq: V act def} and formula~\eqref{V definition}).
We start by analyzing properties of anisotropic random vectors of the
form $DX$, where $D$ is a fixed matrix and $X$ is a random vector with i.i.d.\ components.
A systematic treatment of anti-concentration properties of anisotropic random vectors was given earlier in \cite{RV sup norm}.
We reuse some of the estimates from \cite{RV sup norm} while adding some new ones.
Further, we apply the obtained relations together with a special decoupling procedure to treat the event in question.
The central technical element
of the subsection is the decoupling argument in the proof of Lemma~\ref{l: aux decoupling},
while the results of the subsection are summarized in Proposition~\ref{p: decoupling}.

In Subsection~\ref{subs: deloc estimates}, using results from Sections~\ref{subs: biorth},~\ref{subs: two types},
we estimate probability of the event given (in a deterministic form)
by Lemma~\ref{l: two conditions}.
Thus, we are able to estimate probability of an event of the form
$$\big\{\exists\,\n\in S^{n-1}(\CorR):\;\|(\Az )\n\|_2\leq n^{-2}\mbox{ and }|v_{\ell}|\leq\theta,\;\ell\leq \q-1;
|v_{\q }|\geq\beta\big\}.$$
We use the estimate, together with an averaging argument similar to that in
the proof of Theorem~\ref{th: normal lower}, to
obtain the main statements of this paper
in Subsection~\ref{s: proofs}.

\subsection
{Biorthogonal systems, 
dual ellipsoids, and discretizations}\label{subs: biorth}

Throughout  this subsection we use notation  $\U_\q=(U_1,\dots,U_\q)$ for a  sequence of linearly independent vectors  in a Euclidean space $E$ over $\CorR$ and $\Y_\q=(Y_1,\dots,Y_\q)$ for the  dual sequence.

Let $\Ellips(\U_\q)$ be the ellipsoid generated by $\U_\q$ (see \eqref{eq: aux el def}).
Clearly, the ellipsoid is the linear image of the unit Euclidean ball in $\CorR^\q$
under the action of the linear operator
$L:\ell_2^\q(\CorR)\to E$ with $L({\bf e}_i)=U_i$, $i=1,2,\dots,\q$.
Denote by $s_1(\U_\q)\geq s_2(\U_\q)\geq \dots \geq s_\q(\U_\q)$ the singular values of $L$. Note that by the definition of a
biorthogonal system we have
$s_i(\U_\q)=s_{\q-i+1}^{-1}(\Y_\q),\quad i\le\q.$
\begin{definition}[{Class $\elclass(R,{\bf b})$}]\label{page: elclass}
 Given any  number $R>1$ and
a non-increasing sequence of integers
\begin{equation}\label{b condition}
 {\bf b}=(b_1,\dots,b_\q )\in \Z^\q,\, b_1\ge b_2\ge\dots\ge b_\q,\,\text{ such that }\, 1/2\leq 2^{b_i}\leq R, \,\,i\le \q,
\end{equation}
we say that {\it $\Ellips(\U_\q)$ belongs to the class $\elclass(R,{\bf b})$}
if
$$
\min\big(\max(s_i(\U_\q),1),R\big)\in [2^{b_i},2^{b_i+1}),\quad i\le\q,
$$
i.e. if $ \max ( s_i(\U_\q),1)\ge 2^{b_i}$
and $\min (s_i(\U_\q),R)< 2^{b_i+1}$.
\end{definition}

Given $R>1$, the classes $\elclass(R,{\bf b})$ for all ${\bf b}$ satisfying \eqref{b condition}
form a partition of the set of ellipsoids $\Ellips(\U_\q)$, $\U_\q \in E^\q$.
The next lemma is immediate:
\begin{lemma}\label{l: elclass count}
For any $R>1$, the total number of classes $\elclass(R,{\bf b})$
(for all admissible  ${\bf b}$) is bounded above by $\big(\log_2{R}+2\big)^\q $.
\end{lemma}

The classes $\elclass(R,{\bf b})$ provide a discretization of the set of ellipsoids of not-too-large complexity,
and allow us (in a probabilistic context) to condition on the event that a random ellipsoid generated by the
test projections of vectors $-ze_i$, belongs to a given class, without seriously affecting the probability estimates.
The definition of the classes ``truncates'' large semi-axes of the ellipsoid (those exceeding $R$)
and does not record information about the magnitude of small semi-axes (of length less than $1$):
the probabilistic argument proceeds in such a way that both large and very small semi-axes do not significantly
affect the estimates.

Let $T,\delta>0$ be some parameters.
In what follows, we are interested in describing the condition
\begin{equation}\label{eq: basic condition}
\Ellips(\U_\q)\in\elclass(R,{\bf b})
\quad\mbox{ and }\quad U_{\q}\in \delta\,\Ellips(\U_{\q-1})+T\,  B_2^E
\end{equation}
in terms of the sequence $\Y_\q$, where $\U_{\q-1}$ denotes the sequence $(U_1,\dots,U_{\q-1})$.

For technical reasons, 
we will work with perturbations of $\Y_\q$.
\begin{definition}[$\delta$--perturbation $\W_\q$]
Let $\delta>0$.
We say that a sequence of vectors $\W_\q$ is a {\it $\delta$--perturbation of $\Y_\q$}
if $\|Y_i-W_i\|_2\leq \delta$ 
for all $i\leq \q$.
\end{definition}
Note that we do not require $W_1,\dots,W_{\q}$ to be contained in $\spn(\Y_\q)$.
We have the following elementary consequence of perturbation inequalities for singular values
(see, for example, \cite[Theorem~1.3]{Djalil}):
\begin{lemma}\label{l: W Y sing}
Let $\W_\q$ be a $\delta$--perturbation of $\Y_\q$, and let $s_1(\Y_\q)\geq\dots\geq s_{\q}(\Y_\q)$ and
$s_1(\W_\q)\geq\dots\geq s_{\q}(\W_\q)$ be as above.
Then $s_i(\W_\q)\leq s_i(\Y_\q)+\delta\,\sqrt{\q}$, $i\leq \q.$
\end{lemma}
In the probabilistic setting considered later, we will not have a ``direct access'' to the singular values
$s_i(\W_\q)$
but instead will be able to measure the distances from $W_j$'s to spans of some other $W_i$'s.
The ordering of the vectors $W_j$'s will be quite important, which justifies the following definition:
\begin{definition}[Permutation $\sigma_{\W}$, distances $d_i$]\label{d:di}
  Given $\W_\q\in E^\q$, we define $\sigma_{\W}$ as a permutation of $[\q]$ such that
for all $1\leq i\le \q$,
\begin{align}
d_i:=\d\big(W_{\sigma_{\W}(i)},\spn\{W_{\sigma_{\W}(j)}:\;j<i\}\big)
=\max_{\ell\le\q} \,\d\big(W_{\ell},\spn\{W_{\sigma_{\W}(j)}:\;j<i\}
\big);\label{eq: sigma definition}
\end{align}
in particular $
d_1=\|W_{\sigma_{\W}(1)}\|_2=\max_{ i\le \q}\|W_i\|_2$.
\end{definition}

Note that the permutation may be not unique; in what follows for any $\W_\q$ we fix some $\sigma_{\W}$
satisfying \eqref{eq: sigma definition}.
The role of the permutation $\sigma_\W$ in the proof is twofold.
First, the rearrangement of vectors in $\W_\q$ provides a simple relation between the distances of
$W_{\sigma_{\W}(i)}$ to $\spn\{W_{\sigma_{\W}(j)}:\;j<i\}$ and the singular values $s_i(\W_\q)$
(see Lemma~\ref{l: W dist est} below). Second, the order of $\W_{\sigma_{\W}(i)}$'s
will be used in the crucial Lemma~\ref{l: y reduction}, which, in turn, ``prepares'' a decoupling argument
employed later in Subsection~\ref{subs: two types}.

It is easy to check that $d_1\ge s_1(\W_\q)/\sqrt{N}$ and $d_1\ge d_2\ge\dots\ge d_\q$. Moreover, we have
\begin{lemma}\label{l: W dist est}
Let $\W_\q\in E^\q$  and let
$s_1(\W_\q)\geq s_2(\W_\q)\geq\dots\geq s_{\q}(\W_\q)$ be as above.
Then, with $d_i$'s defined by formula \eqref{eq: sigma definition}, we have
$
d_i\geq  s_i(\W_\q)/\sqrt{\q-i+1}$, $i\leq \q.
$
\end{lemma}
\begin{proof}

It is convenient to regard $s_i(\W_\q)$'s as the singular values of the $m\times \q $ matrix $Q$ with columns $W_i$, $i\leq \q$.
Fix any $2\le i\leq \q$, and set $F:=\spn\{W_{\sigma_{\W}(j)}:\;j<i\}^\perp$.
The min-max formula for singular values of $Q$ (see, for example, \cite[Theorem~1.2]{Djalil}) implies that
$$
s_i(\W_\q)\leq \max_{QX\in F,\|X\|_2=1}\|QX\|_2 \le \|\Proj_{F} Q\|,
$$
where $\Proj_{F}$ denotes the orthogonal projection
onto $F$. On the other hand, by the definition of $\sigma_{\W}$ be have
$$
\|\Proj_{F} W_\ell\|_2=
\d\big(W_{\ell},\spn\{W_{\sigma_{\W}(j)}:\;j<i\}
\big)\leq d_i
$$
for all $\ell\le\q$. Hence,
$\|\Proj_{F} Q\|\leq d_i\sqrt{\q-i+1},$
and the result follows.
\end{proof}

\medskip

It follows from the definition of $\elclass(R,{\bf b})$ that for any $R>1$, ${\bf b}$ satisfying \eqref{b condition}, and any sequence $\U_\q$, if
$\Ellips(\U_\q)\in\elclass(R,{\bf b})$ then
\begin{equation}\label{eq: spvol to b}
2^{\sum_{i=1}^\q b_i}\le \prod\limits_{i=1}^{\q}\min\big(\max(s_i(\U_\q),1),R\big)< 2^{\q+\sum_{i=1}^\q b_i},
\end{equation}
where the quantity in the middle serves as a measure (``truncated volume'') of $\Ellips(\U_\q)$.
Lemma \ref{l: W dist est} allows to relate this measure to characteristics of a perturbation of a dual basis:

\begin{prop}\label{p: geom part 1}
Let $\delta>0$, $\q\ge1$, and $1<R\leq \delta^{-1}\q ^{-1/2}$. Let
$\U_\q,\Y_\q$ be a biorthogonal system in  $E$, where $\Ellips(\U_\q)\in\elclass(R,{\bf b})$ for some ${\bf b}\in\Z^N$.
 Let also $\W_\q$ be a $\delta$--perturbation of $\Y_\q$, and $d_i$, $i\le \q$, be defined by \eqref{eq: sigma definition}.
Then
$$\prod\limits_{i=1}^{\q} \min(d_i,1)
\leq (4\q )^{\q/2}\,2^{-\sum_{i=1}^{\q}b_i}.
$$
\end{prop}
\begin{proof}
Fix for a moment any $t>0$, and let $I\subset[\q]$ be the set of all indices $i\in[\q]$ with $s_i(\W_\q)\leq t$.
Then for every $i\in I^c:=[\q]\setminus I$ we have
$\min(d_i,t)\le t\le \min(s_i(\W_\q),t)$, so that
$$
\prod\limits_{i\le\q}\min(d_i,t)\leq \prod\limits_{i\in I} d_i\prod\limits_{i\in I^c} \min(s_i(\W_\q),t).
$$
Next, since $d_i\ge s_i(\W_\q)/\sqrt{\q}$ by Lemma~\ref{l: W dist est}, and in view of the standard identity
$$\prod_{i\le\q}d_i=\prod_{i\le\q}s_i(\W_\q)$$
(see, for example, formula~(3) in \cite{Djalil}), we get
$$
\prod\limits_{i\in I}d_i=\prod\limits_{i\le\q}s_i(\W_\q)\prod\limits_{i\in I^c}d_i^{-1}
\leq \q ^{\q /2}\prod\limits_{i\in I}s_i(\W_\q)=\q ^{\q /2}\prod\limits_{i\in I}\min(s_i(\W_\q),t).
$$
Hence
\begin{equation*}
\prod\limits_{i=1}^{\q}\min(d_i,t)\leq \q ^{\q /2}\prod\limits_{i=1}^{\q} \min(s_i(\W_\q),t).
\end{equation*}
Further, in view of Lemma~\ref{l: W Y sing} and the duality relation between 
$\U_\q$ and $\Y_\q$, we have
\begin{align*}
\min(s_i(\W_\q),t)&\leq \min\big(s_{\q-i+1}^{-1}(\U_\q)+\delta\,\sqrt{\q},t\big)\\
&\leq 2\max\big(\min(s_{\q-i+1}^{-1}(\U_\q),t),\min(\delta\,\sqrt{\q},t)\big),
\end{align*}
where, as before, $s_1(\U_\q)\geq\dots\geq s_{\q}(\U_\q)$ are the singular values
of the linear operator $L:\CorR^\q\to E$ with $L({\bf e}_i)=U_i$, $i=1,2,\dots,\q$.

Now, choosing $t=1$, we get from the above and our assumption on $R$
\begin{align*}
\min(s_i(\W_\q),1)&\leq 2\max\big(\min(s_{\q-i+1}^{-1}(\U_\q),1),\delta\,\sqrt{\q}\big)
\leq 2/\min\big(\max(s_{\q-i+1}(\U_\q),1),R\big),
\end{align*}
and by (\ref{eq: spvol to b}),
\begin{align*}
\prod\limits_{i=1}^{\q}\min(d_i,1)
&\leq \q ^{\q /2}\prod\limits_{i=1}^{\q} \min(s_i(\W_\q),1)
\\
&\leq (4\q) ^{\q /2}\prod\limits_{i=1}^{\q}\min\big(\max(s_i(\U_\q),1),R\big)^{-1}\le(4\q )^{\q/2}\,2^{-\sum_{i=1}^{\q}b_i}.
\end{align*}
This finishes the proof.
\end{proof}

Returning to formula \eqref{eq: basic condition},
we can now describe the condition $\Ellips(\U_\q)\in\elclass(R,{\bf b})$
 in terms of a perturbation of $\Y_\q$.
It will be convenient to introduce a classification of sequences of vectors based on statistics of distances.
\begin{definition}[{Class $\disclass(r ,{\bf p})$}]\label{page: disclass}
Given $0<r <1$,  let  ${\bf p}:=(p_1,\dots,p_\q )\in \Z^n$ satisfy
\begin{equation}\label{p condition}
 p_1\ge\dots\ge p_\q\,\, \mbox{ and }\,\, r /2\leq 2^{p_i}\leq 1, \,\,i\le \q.
\end{equation}
We say that a sequence of vectors $\W_\q$ in a Euclidean space $E$ {\it belongs to the class $\disclass(r ,{\bf p})$} if
$$
\min\big(\max(d_i,r ),1\big)\in [2^{p_i},2^{p_i+1})\quad \forall i\le\q,
$$
where the numbers $d_i$, $i\le N$, are defined by \eqref{eq: sigma definition}.
\end{definition}
The following observation is immediate:
\begin{lemma}\label{l: disclass count}
For any $0<r <1$, the total number of classes $\disclass(r ,{\bf p})$
for all ${\bf p}$ satisfying \eqref{p condition} does not exceed
$\big(2-\log_2{r }\big)^{\q}$.
\end{lemma}
Note that, for every $0<r <1$, we have $r \le\min\big(\max(d_i,r ),1\big)\le 1$, $i\le \q$. Thus,
for any $\W_\q$ there exists a sequence ${\bf p}$ satisfying \eqref{p condition}
such that $\W_\q$ belongs to $\disclass(r ,{\bf p})$,
i.e.\ the classes form a partition of the $\q$--sequences of vectors from $E$.
We have the following formal consequence of Proposition~\ref{p: geom part 1}, describing the condition $\Ellips(\U_\q)\in\elclass(R,{\bf b})$ in \eqref{eq: basic condition}:
\begin{cor}\label{cor: ellipsoid to distance type}
Given $\delta>0$, $\q\ge1$, let $\U_\q,\Y_\q$ be a biorthogonal system in $E$ and $\W_\q$ be a $\delta$--perturbation of $\Y_\q$.
Further, let $r<1$, $1<R\leq \delta^{-1}\q ^{-1/2}$, and let a vector ${\bf b}$ satisfy \eqref{b condition}.
Assume that
$$
\mbox{$\Ellips(\U_\q)\in\elclass(R,{\bf b})$  and $d_i=d_i(\W_\q)\ge r $, $i\le\q$}.
$$
Then
$$
\mbox{$\W_\q\in \disclass(r ,{\bf p})$ for some ${\bf p}$ satisfying $\,\,2^{\sum_{i=1}^{\q}p_i}\cdot 2^{\sum_{i=1}^{\q}b_i}\leq (4\q )^{\q/2}\,$.}
$$
\end{cor}
\begin{proof}
For every $i\leq \q$, we define $p_i$ as the (unique) integer satisfying $\min\big(d_i,1\big)\in[2^{p_i},2^{p_i+1})$.
Then Definition \ref{page: disclass} and Proposition \ref{p: geom part 1} imply the result.
\end{proof}
Crucially, the above corollary relates quantity $2^{\sum_{i=1}^{\q}b_i}$, which serves as a measure of the size of
$\Ellips(\U_\q)$ in our computations, with the quantity $2^{\sum_{i=1}^{\q}p_i}$
which can be thought of as a measure of ``spreadness'' of the sequence $\W_\q$.
The inequality $2^{\sum_{i=1}^{\q}p_i}\cdot 2^{\sum_{i=1}^{\q}b_i}\leq (4\q )^{\q/2}$ can be viewed as an ``approximation''
of the natural duality relation $\prod_{i=1}^\q s_i(\U_\q)\cdot \prod_{i=1}^\q s_i(\Y_\q)=1$ between the semi-axes of ellipsoids
generated by the biorthogonal system. We will return to discussing the relation between the parametric vectors ${\bf b}$ and ${\bf p}$
at the beginning of Subsection~\ref{subs: deloc estimates}.

\bigskip

In the second part of this subsection, we investigate the second condition in \eqref{eq: basic condition}.

\begin{lemma}[Inclusion into ellipsoid via biorthogonal system]\label{l: y existence}

\noindent Let $\U_\q$, $\Y_\q$ be a biorthogonal system in $E$.
Assume further that for some $T,\delta> 0$ we have
$$U_{\q}\in \delta\,\Ellips(\U_{\q-1})+T  B_2^E .$$
Then there is $Y\in E$ with $\|Y\|_2\leq T$ such that
$\sum\limits_{\ell=1}^{\q-1}|\langle Y_\ell,Y\rangle|^2\leq \delta^2$, and $|\langle Y_{\q},Y\rangle|=1$.
\end{lemma}
\begin{proof}
Let $B$ be a matrix with columns $U_1,\dots,U_{\q}$.
Then $Y_1,\dots,Y_{\q}$
are the rows of the complex conjugate of the inverse $\bar B^{-1}$. Conditions of the lemma can be rewritten as
$B{\bf a}=Y$,
for some ${\bf a}=(a_1,\dots,a_{\q})$ with $a_{\q}=-1$ and $\|(a_1,\dots,a_{\q-1})\|_2\leq \delta$
and $\|Y\|_2\leq T$.
Equivalently, ${\bf a}=B^{-1}Y$, and the result follows.
\end{proof}
\begin{cor}\label{cor: wy}
Let $T,\delta$, $\U_\q$, $\Y_\q$ be as in the lemma above,
and let $\W_\q$ be a $\delta$--perturbation of $\Y_\q$.
Assume additionally that $T\delta\leq \frac{1}{2}$.
Then there is a vector $Y\in E$ with $\|Y\|_2\leq 2T$ such that
$|\langle W_\ell,Y\rangle|\leq 2\delta +2\delta T$, $\ell\leq \q-1$, and $|\langle W_{\q},Y\rangle|=1$.
\end{cor}
\begin{proof}
It follows from Lemma \ref{l: y existence} that there exists $Y'\in E$ with $\|Y'\|_2\leq T$ such that
$|\langle W_\ell,Y'\rangle|\leq \delta +\delta T$, and $|\langle W_{\q},Y'\rangle|\ge 1- \delta T\ge 1/2$.
Hence, there is a number $0<c\leq 2$ such that $Y=cY'$ satisfies the conclusion of the corollary.
\end{proof}
As it follows from Corollary~\ref{cor: wy}, in probabilistic setting, in order to
estimate probability of the event $U_{\q}\in \delta\,\Ellips(\U_{\q-1})+T B_2$ it is sufficient to bound probability of the event of the form
\begin{equation}\label{eq: aux 1075090218415761}
\big\{\exists Y:\;\|Y\|_2\leq T'\mbox{ and
$|\langle W_\ell,Y\rangle|\leq \delta'$, $\ell\leq \q-1$; $|\langle W_{\q},Y\rangle|=1$}\big\}.
\end{equation}
Estimating probability of the latter
by taking the union bound over all possible choices of $Y$ would lead to summation over a continuum. To resolve this issue
we use Lemma \ref{l: y reduction} below, which allows to associate with a sequence $\W_\q$ {\it a finite}
collection $\disccollec(\W_\q,\delta',T)$ of vectors in $E$ which can be reconstructed from $\W_\q$ using certain relations for scalar products.
\begin{definition}[{Set $\disccollec(\W_\q,\delta',T)$}]
 \label{page: disccollec}
Given a sequence of linearly independent vectors $\W_\q$ in $E$ (over $\CorR=\C$) and parameters $\delta',T>0$, define
$\disccollec(\W_\q,\delta',T)$ as the set of all vectors $Z\in\spn(\W_\q)$ such that
\begin{itemize}
\item $\|Z\|_2\leq T$;
\item $\langle Z,W_{\q}\rangle=1$;
\item for any $i\in [\q]\setminus \{\sigma^{-1}_\W(\q)\}$, we have
$|\langle Z,W_{\sigma_\W(i)}\rangle|\leq 2\sqrt{2}(\q+1)\delta'$,
and\\
{\bf either}
$$\langle Z,W_{\sigma_\W(i)}\rangle/\delta'\in \Z+\cxi\Z$$
{\bf or}
$$\langle Z,W_{\sigma_\W(i)}\rangle=\langle Z,\Proj_{\spn\{W_{\sigma_\W(j)},\,j<i\}}W_{\sigma_\W(i)}\rangle.$$
\end{itemize}
\end{definition}
Here, $\sigma_\W$ is the permutation of $[\q]$ defined via relations \eqref{eq: sigma definition}.
When $\CorR=\R$, the definition of the vector collection is modified by replacing the lattice $\Z+\cxi\Z$ with $\Z$.
\smallskip

For example, in a simplest case when $W_i=W_{\sigma_\W(i)}={\bf e}_i$, $i\in[\q]$, and $\delta'=1$, we have
\begin{align*}
 \disccollec(\W_\q,1,T)=\big\{\mathbf{y}\in TB_2^\q(\C):
\,\, &|y_i|\le 2\sqrt{2}(\q+1)\,\,\text{ and}\,\,  y_i\in  \Z+\cxi\Z\,\,\forall i\in[\q-1];\;\;y_\q=1 \big\}.
\end{align*}

\medskip

The collection $\disccollec(\W_\q,\delta',T)$ can be viewed as a discretization of the set of vectors $Y$ from
Corollary~\ref{cor: wy}, up to some adjustment of parameters.
In particular, we allow the scalar products $\langle Z,W_{\sigma_\W(i)}\rangle$ (for $Z\in\disccollec(\W_\q,\delta',T)$)
to take only discrete values from a rescaled lattice $\Z+\cxi\Z$, {\bf unless} $Z$
is orthogonal to the vector $\Proj_{\spn\{W_{\sigma_\W(j)},\,j<i\}^\perp}W_{\sigma_\W(i)}$.
The ``unless'' part in the last statement is important: in the regime when a vector $W_{\sigma_\W(i)}$
is ``almost contained'' in the linear span of $W_{\sigma_\W(j)}$, $j<i$, i.e.\
the projection of $W_{\sigma_\W(i)}$ onto the orthogonal complement of $\spn\{W_{\sigma_\W(j)},\,j<i\}$
is very small, simultaneously fulfilling the requirements $\langle Z,W_{\sigma_\W(j)}\rangle\in \delta'\Z+\delta'\cxi\Z$, $j\leq i$,
and $\|Z\|_2= O(T')$ may not be possible even if a vector $Y$ in \eqref{eq: aux 1075090218415761}
does exist. Having ``simplified'' the definition of $\disccollec(\W_\q,\delta',T)$ leaving only conditions
$\langle Z,W_{\sigma_\W(j)}\rangle\in \delta'\Z+\delta'\cxi\Z$, $j\in [\q]\setminus \{\sigma^{-1}_\W(\q)\}$,
we would not be able to prove the
crucial Lemma~\ref{l: y reduction} below.
The special role of the scalar product $\langle Z,\W_\q\rangle$ is naturally related to the fact that $Z$ should serve as an ``approximation''
of a vector $Y$ from \eqref{eq: aux 1075090218415761}.

Note that given a sequence of linearly independent vectors $\W_{\q}$, a permutation $\sigma$,
a subset $J\subset [\q]\setminus\{\sigma^{-1}(\q)\}$, and numbers $(k_j)_{j\in J}$
from $\CorR$, the conditions
\begin{equation}\label{Z}
\begin{split}
\langle Z,\W_\q\rangle=1,\;\;
 & \langle Z,W_{\sigma(i)}\rangle=
\begin{cases}
k_i\delta',&\mbox{if }i\in J,
\\
\langle Z,\Proj_{\spn\{W_{\sigma(j)},\,j<i\}}W_{\sigma(i)}\rangle,&\mbox{if }i\in [\q]\setminus (J\cup \{\sigma^{-1}(\q)\}),
\end{cases}
\end{split}
\end{equation}
{\it uniquely} determine a vector $Z=Z_{\W_{\q},J,(k_j),\sigma}\in\spn(\W_{\q})$.
Thus we can say that if $\disccollec(\W_\q,\delta',T)$ is not empty then there exist
a permutation $\sigma$,
a subset $J\subset [\q]\setminus\{\sigma^{-1}(\q)\}$,
and complex numbers with integer real and imaginary parts,
\begin{equation}\label{eq: k definition}
k_i\in \{z\in \C:\, |z|\le 2\sqrt{2}(\q+1)\}\cap (\Z+\cxi \Z), \quad i\in J,
\end{equation}
such that vector $Z=Z_{\W_{\q},J,(k_j),\sigma}$ defined in \eqref{Z} satisfies
\begin{equation}\label{|Z|}
|\langle Z\,,\,W_{\sigma(i)}\rangle|\leq 2\sqrt{2}(\q+1)\delta',\, i\in [\q]\setminus\{\sigma^{-1}(\q)\}
\quad\mbox{ and }\quad  \|Z\|_2\le T.
\end{equation}
(If $\CorR=\R $ then $k_i\in (-2\sqrt{2}(\q+1),2\sqrt{2}(\q+1))\cap \Z$.)
We will return to discussing the set $\disccollec(\W_\q,\delta',T)$ in Remark~\ref{rem: more com on ups}.

The key statement below shows that every vector satisfying relations from \eqref{eq: aux 1075090218415761}
can be perturbed into a vector satisfying conditions \eqref{Z}, without increasing its length.
\begin{lemma}\label{l: y reduction}
Let $\W_\q$ be a sequence of linearly independent vectors in $E$. Assume that
$
\delta,\delta'>0$ and $Y\in E$ are such that
$$|\langle Y,W_{\q}\rangle|=1\quad \mbox{ and }\quad |\langle Y,W_i\rangle|\leq \delta',\;\;1\leq i\leq\q-1.$$
Assume further that $\d(W_{\q},\spn\{W_i,\,i\neq \q\})\geq \delta$.
Then the set $\disccollec(\W_\q,\delta',\|Y\|_2+\sqrt{2}\q \delta'/\delta)$ defined above is non-empty.
\end{lemma}
\begin{proof}
We will prove the statement when $\CorR=\C$; the real case follows by a straightforward adaptation of the argument.
Without loss of generality, $\langle Y,W_{\q}\rangle=1$.
For brevity, we will use notations
$$
W_i':=W_{\sigma_{\W}(i)},\,\, i\leq \q,\quad\text{and}\quad q:=\sigma_{\W}^{-1}(\q).
$$
Assuming conditions of the lemma, we will construct a vector $Z_\q\in\disccollec(\W_\q,\delta',\|Y\|_2+\sqrt{2}\q \delta'/\delta)$
in $\q$ steps via an inductive argument. At each step, the vector given by the induction hypothesis
will be perturbed in a special way, which will not increase its Euclidean norm, except for the step $q$
which has to be treated differently due to the special role of the inner product of the vector with $W'_q=W_\q$.

Set $Z_0:=Y$ and $F_{0}:=\{0\}$. For any $m=0,\dots,\q$, denote
$F_{m}:=\spn\{W_{1}',\dots,W_{m}'\}$.
For every $m\in[\q]$, at the $m$-th step of the induction argument, we assume that there exists a vector $Z_{m-1}\in E$ satisfying conditions:
\begin{itemize}
\item[a)] $\|Z_{m-1}\|_2\leq\|Y\|_2+{\bf 1}_{\{m-1\geq q\}}\sqrt{2}q \delta'/\delta$,
\item[b)] for any $i\in [\q]$,
$|\langle Z_{m-1},W'_{i}\rangle-{\bf 1}_{\{i=q\}}|\leq \sqrt{2}((m-1)+1)\delta'
+{\bf 1}_{\{m-1\geq q\}}\sqrt{2}q\delta'$, and
\item[c)] for any $i\in [m-1]\setminus\{q\}$, {\bf either}
$$\langle Z_{m-1},W'_{i}\rangle/\delta'\in \Z+\cxi\Z
\quad\mbox{ {\bf or} }\quad
\langle Z_{m-1},W'_{i}\rangle=\langle Z_{m-1},\Proj_{F_{i-1}}W'_{i}\rangle,$$
\item[d)] if $m-1\geq q$ then $\langle Z_{m-1},W'_q\rangle=\langle Z_{m-1},W_\q\rangle=1$,
\end{itemize}
and then construct vector $Z_m$ satisfying (a)--(d) with $m-1$ replaced with $m$.
Then, evidently, after $\q$ steps we get  $Z_\q\in \disccollec(\W_\q,\delta',\|Y\|_2+\sqrt{2}\q \delta'/\delta)$ and finish the proof.

Since $Z_0$ satisfies all of the above conditions with $m=1$, this allows us to start the inductive procedure.

Note that $Z_{m-1}$ satisfies all conditions for $Z_m$ except, possibly, (c) for $i=m$ and (d).
Hence to obtain $Z_m$ from $Z_{m-1}$,
we need to perturb $\langle Z_{m-1}, W'_m\rangle$ in a way which does not ``spoil'' conditions on
$\langle Z_{m-1}, W'_i\rangle$ for $i\neq m$.
To this end it is convenient to separate the projection of $Z_{m-1}$ on
$W'_m$ from the rest of projections and represent $Z_{m-1}$ in the form
\begin{align*}
  &Z_{m-1}=Z_{m-1}'+Z_{m-1}'',
\end{align*}
where
\begin{align*}
  & Z_{m-1}'=\Proj_{F_{m-1}\bigoplus F_m^\perp}Z_{m-1}\,\,\mbox{ and }
  \\
  &Z_{m-1}''=\Proj_{\Proj_{F_{m-1}^\perp}(W'_m)}Z_{m-1}=\langle Z_{m-1}, \Proj_{F_{m-1}^\perp}W_m'\rangle \frac{\Proj_{F_{m-1}^\perp}W_m'}{\|\Proj_{F_{m-1}^\perp}W_m'\|_2^2}.
\end{align*}
In words, we decompose $Z_{m-1}$ as a sum of orthogonal projections onto two orthogonal subspaces:
$F_{m-1}\bigoplus F_m^\perp$ and the (complex) one-dimensional subspace generated by the vector $\Proj_{F_{m-1}^\perp}W_m'$.
Clearly, such representation is unique.

Let us first treat the special case $m=q$. In this case we set
$$
Z_q:=Z_{q-1}'-\frac{(\langle Z_{q-1}',W'_q\rangle-1)\,\Proj_{F_{q-1}^\perp}W_q'}{\|\Proj_{F_{q-1}^\perp}W_q'\|_2^2}
=Z_{q-1}-\frac{(\langle Z_{q-1},W'_q\rangle-1)
\,\Proj_{F_{q-1}^\perp}W_q'}{\|\Proj_{F_{q-1}^\perp}W_q'\|_2^2},
$$
so that
$
\langle Z_q,W'_q\rangle=1.
$
Note that the inner products of $Z_q$ with $W'_1,\dots,W'_{q-1}$ remain the same as they were for the vector $Z_{q-1}$.
Thus, $Z_q$ satisfies (d) and (c) with $m-1$ replaced with $m=q$.
Further, the norm of the vector $Z_q$ can be estimated as
$$
\|Z_q\|_2\leq \|Z_{q-1}'\|_2+\frac{|\langle Z_{q-1}',W'_q\rangle-1|}{\|\Proj_{F_{q-1}^\perp}W_q'\|_2}
\leq \|Y\|_2+\frac{\sqrt{2}q\delta'}{\delta},
$$
where we have used the induction hypothesis for $Z_{q-1}$ and the assumption on the distance of $W_q'=W_\q$
to the linear span of the rest.
Thus, $Z_q$ satisfies (a), and it remains to check condition (b) for $i>q$.
Here is where our choice of the permutation $\sigma_{\W}$ becomes crucial.
Note that, by the definition of $\sigma_{\W}$, the distance of $W_q'$ to $F_{q-1}$ is at least as large as
the distance from $W'_i$ to $F_{q-1}$ for any $i>q$. Hence, from the above formula for $Z_q$
and the induction hypothesis we obtain
$$
|\langle Z_q,W'_i\rangle|\leq |\langle Z_{q-1},W'_i\rangle|+|\langle Z_{q-1},W'_q\rangle-1|\leq 2\sqrt{2} q\delta',\quad i>q,
$$
and (b) is verified.

\medskip

For the rest of the proof, we will consider the case $m\neq q$.
We will look for $Z_m$ of the form
$$
Z_m=Z_m(\zeta):=Z_{m-1}'+\zeta\,Z_{m-1}''=Z_{m-1}+(\zeta-1)Z_{m-1}'',\quad \zeta\in D_1:=\{\zeta\in\C:\,\,|\zeta|\le1\},
$$
Our goal is to show that we can choose parameter $\zeta\in D_1$ so that
$Z_m$ satisfies (a)--(d) with $m-1$ replaced with $m$.
The above representation
and the inductive hypothesis (conditions (a)--(d) for $Z_{m-1}$) guarantee that for any choice of $\zeta\in D_1$ we have
\begin{align*}
  &1.\,\,\,\|Z_{m}\|_2\leq\|Z_{m-1}\|_2\leq\|Y\|_2+{\bf 1}_{\{m-1\geq q\}}\sqrt{2}q \delta'/\delta,&
  \\
  &2.\,\,\,\text{for }i<m,\,\,\,\langle Z_{m},W'_{i}\rangle=\langle Z_{m-1},W'_{i}\rangle,
  \\
  &3.\,\,\,\text{for }i=m,\,\,\,\langle Z_{m},W'_{m}\rangle=\langle Z_{m},\Proj_{F_{m-1}}W'_{m}\rangle+\zeta\langle Z_{m-1},\Proj_{F_{m-1}^\perp}W'_{m}\rangle,
  \\
  &4.\,\,\,\text{for }i\geq m,\,\,\,|\langle Z_{m},W'_{i}\rangle-{\bf 1}_{\{i=q\}}|\leq \sqrt{2}m\delta'
+{\bf 1}_{\{m-1\geq q\}}\sqrt{2}q\delta'+|\zeta-1||\langle Z''_{m-1},W'_{i}\rangle|,
\end{align*}
so that, in particular, $Z_m(\zeta)$ satisfies (a), (d); (b) for every for every $i\le m-1$, and (c) for every $i\in[m-1]\setminus\{q\}$. 
Thus, it remains to provide (b) for $i\ge m$ and (c) for $i=m$, i.e.\ to find $\zeta\in D_1$ such that
\begin{equation}\label{b}
  |\langle Z_m,W'_{i}\rangle-{\bf 1}_{\{i=q\}}|\leq \sqrt{2}(m+1)\delta'+\sqrt{2}q\delta'\quad\text{for every }\, i\ge m
\end{equation}
and
\begin{equation}\label{c}
  \mbox{ {\bf either} } \langle Z_m,W'_{m}\rangle/\delta'\in \Z+\cxi\Z
\quad\mbox{ {\bf or} }\quad
\langle Z_m,W'_{m}\rangle=\langle Z_m,\Proj_{F_{m-1}}W'_{m}\rangle.
\end{equation}
Again, the definition of $\sigma(\W_\q)$ (see \eqref{eq: sigma definition}) plays a crucial role.
Namely, since by \eqref{eq: sigma definition},
$$
|\langle \Proj_{F_{m-1}^\perp}W_m',W_j'\rangle|\leq \| \Proj_{F_{m-1}^\perp}W_m'\|_2\| \Proj_{F_{m-1}^\perp}W_j'\|_2\leq d_m^2= |\langle \Proj_{F_{m-1}^\perp}W_m',W_m'\rangle|,\quad \forall j\ge m,
$$
we have
\begin{equation*}
 |\langle Z_{m-1}'',W_j'\rangle|\leq |\langle Z_{m-1}'',W_m'\rangle|, \quad \forall j\ge m.
\end{equation*}
This will allow us to ``adjust'' the inner product $\langle Z_m(\zeta), W_m'\rangle$
while not affecting $\langle Z_m(\zeta), W_j'\rangle$ ($j>m$) too much. In particular, now we can write instead of {\bf 4} above
\begin{equation}\label{4}
  |\langle Z_{m},W'_{i}\rangle-{\bf 1}_{\{i=q\}}|\leq \sqrt{2}m\delta'
+{\bf 1}_{\{m-1\geq q\}}\sqrt{2}q\delta'
+|\zeta-1||\langle Z''_{m-1},W'_{m}\rangle|,\quad \forall i\geq m.
\end{equation}

To get \eqref{b} -- \eqref{c} choosing appropriate $\zeta$, and so to finish the proof, we consider three cases.

\underline{Case 1}: $\langle Z_{m-1}, \Proj_{F_{m-1}^\perp}W_m'\rangle=0$. In this case $Z_m=Z_{m-1}$ (for any choice of $\zeta$) and
\eqref{b} follows from (b) for $Z_{m-1}$, while \eqref{c} is provided by the assumption.

\underline{Case 2}: $|\langle Z_{m-1}'',W_m'\rangle|\leq \sqrt{2}\delta'$.  In this case we set $\zeta:=0$, so
$Z_m=Z_m(0)$. Now \eqref{b} follows from \eqref{4} and \eqref{c} follows from {\bf 3} above.

\underline{Case 3}: $\langle Z_{m-1}, \Proj_{F_{m-1}^\perp}W_m'\rangle\neq 0$ and $|\langle Z_{m-1}'',W_m'\rangle|> \sqrt{2}\delta'$ .
Setting
$$
\Gamma:=\big\{z\in\C:\,\langle Z_m(z),W_m'\rangle\in \delta'\Z+\delta'\cxi \Z\big\},
$$
observe that necessarily $\Gamma\neq \emptyset$ and every  $Z_m(\zeta)$, $\zeta\in\Gamma$, satisfies \eqref{c}. So it remains to find
 $\zeta_0\in\Gamma\cap D_1$ such that $Z_m(\zeta_0)$ satisfies \eqref{b}.

Note that for every $\zeta_1, \zeta_2\in \Gamma$ we have
$$
\tau_i:=\frac{\langle Z_m(\zeta_i),W_m'\rangle}
{|\langle Z_{m-1}'',W_m'\rangle|}\in\frac{\delta'}{|\langle Z_{m-1}'',W_m'\rangle|}(\Z+\cxi\Z), \quad i=1,2,
$$
and $|\zeta_1-\zeta_2|=|\tau_1-\tau_2|$. Hence, $\Gamma$ is a {\it distance-preserving} transformation of the lattice
$(\delta'/|\langle Z_{m-1}'',W_m'\rangle|)(\Z+\cxi\Z)$, and, in particular, any disc in $\C$ of radius greater than or equal to $\delta'/(\sqrt{2}|\langle Z_{m-1}'',W_m'\rangle|)$
contains at least one point from $\Gamma$.
Thus, the condition $|\langle Z_{m-1}'',W_m'\rangle|> \sqrt{2}\delta'$ guarantees that the intersection of
the discs
$$D_1\quad\mbox{ and }\quad D:=\{\zeta\in\C:\,|1-\zeta|\leq \sqrt{2}\delta'/|\langle Z_{m-1}'',W_m'\rangle|\}$$
contains a disk of radius $\delta'/\sqrt{2}|\langle Z_{m-1}'',W_m'\rangle|$ which, in turn, contains at least one point, say $\zeta_0$, from $\Gamma$. Choose $Z_m=Z_m(\zeta_0)$,
$\zeta_0\in\Gamma\cap D_1\cap D$. Since $\zeta_0\in D$, we have
$
|\zeta_0-1| |\langle Z_{m-1}'',W_j'\rangle|\le 2\delta'.
$
This and \eqref{4} yield \eqref{b}.

\medskip

The above procedure produces a sequence of vectors $Z_0,Z_1,\dots,Z_\q$ satisfying (a)--(d).
Clearly,  $Z_\q\in \disccollec(\W_\q,\delta',\|Y\|_2+\sqrt{2}\q \delta'/\delta)$. This finishes the proof.
\end{proof}

\begin{rem}\label{rem: more com on ups}
We would like to note that some aspects in our definition of $\disccollec(\W_\q,\delta',T)$
are rather arbitrary.
In particular, in the condition that for any $Z\in \disccollec(\W_\q,\delta',T)$,
the inner product $\langle Z,W_{\sigma_\W(i)}\rangle$ belongs to
$\delta'\Z+\delta'\,\cxi\Z$ (unless $Z$
is orthogonal to $\Proj_{\spn\{W_{\sigma_\W(j)},\,j<i\}^\perp}W_{\sigma_\W(i)}$),
the lattice $\delta'\Z+\delta'\,\cxi\Z$ could be replaced with some other
(appropriate) discretizations of $\C$ without affecting the argument.

The properties of $\disccollec(\W_\q,\delta',T)$ of actual importance
are dictated by its role in estimating probabilities of events of the form \eqref{eq: aux 1075090218415761}
when $\W_\q$ is a collection of independent random vectors.
The three key elements in the estimation procedure are discretization, decoupling and a union bound argument.
As we noted above, the event $\{\disccollec(\W_\q,\delta',T)\neq\emptyset\}$
(hence the event \eqref{eq: aux 1075090218415761}, in view of Lemma~\ref{l: y reduction}),
implies that for some parameters $J$, $(k_j)$, and $\sigma$, the vector $Z=Z_{\W_{\q},J,(k_j),\sigma}$
defined by \eqref{Z} satisfies \eqref{|Z|}.
The crucial property of the vector is that for any $i\leq \q$, the projection of
$Z_{\W_{\q},J,(k_j),\sigma}$ onto $\spn\{W_\q;W_{\sigma(j)},\,j\leq i\}$ is measurable with respect to
the sigma-field generated by $W_\q$ and $W_{\sigma(j)},\,j\leq i$.
This property will be used in the proof of Lemma~\ref{l: aux decoupling} as a basis for the decoupling procedure.
The cardinality of the parametric space $(J,(k_j),\sigma)$ can be easily bounded from above by $(\widetilde C \q)^{\widetilde C \q}$,
which will enable us to apply a union bound in probability estimation.
\end{rem}

Summarising results of this section and
combining Lemma~\ref{l: y reduction} with Corollaries~\ref{cor: ellipsoid to distance type} and \ref{cor: wy} we get:

\begin{prop}\label{p: perturbed vector existence}
Let $\U_\q$, $\Y_\q$ be a biorthogonal system in $E$ over $\C$ and $\W_{\q}$ be a $\delta$--perturbation of $\Y_{\q}$
(for some $\delta\in(0,1]$), such that $\d(W_\q,\spn\{W_i,\,i\neq \q\})\geq \delta$.
Assume that for some ${\bf b}$ satisfying \eqref{b condition} and some $ R\geq 1$, $T> 0$ with $T\delta\leq 1/2$
and $R\leq \delta^{-1}N^{-1/2}$
we have
$$
\Ellips(\U_\q)\in\elclass(R,{\bf b})
\quad\mbox{ and }\quad U_{\q}\in \delta\,\Ellips(\U_{\q-1})+ T  B_2^E .
$$
Let $r <1$ be such that $d_i(\W_\q)\ge r$, $i\le\q$, where $d_i$'s are defined in \eqref{eq: sigma definition}.
Then there exists ${\bf p}$ satisfying \eqref{p condition} and
$2^{\sum_{i=1}^{\q}p_i}\leq (4\q )^{\q/2}\,2^{-\sum_{i=1}^{\q}b_i}$ such that
\begin{equation}\label{event D+disccollec}
  \W_\q\in \disclass(r ,{\bf p})
\quad\mbox{ and }\quad \disccollec(\W_\q,2\delta (T+1),2T+2\sqrt{2}\q(T+1) )\neq\emptyset.
\end{equation}
In particular, the second condition in \eqref{event D+disccollec} implies that there exist
a permutation $\sigma$ of $[\q]$,
a subset $J\subset [\q]\setminus\{\sigma^{-1}(\q)\}$, numbers
$$
k_i\in \{z\in \C:\, |z|\le 2\sqrt{2}(\q+1)\}\cap (\Z+\cxi \Z),\quad i\in J,
$$
such that vector $Z=Z_{\W_{\q},J,(k_j),\sigma}$ defined in \eqref{Z} with $\delta'=2\delta(T+1)$ satisfies
$$
|\langle Z\,,\,W_{\sigma(i)}\rangle|\leq \sqrt{2}(\q+1)\delta',\, i\in [\q]\setminus\{\sigma^{-1}(\q)\}
\quad\mbox{ and }\quad\|Z\|_2\le 2T+2\sqrt{2}\q(T+1).
$$
\end{prop}
Note that if $\CorR=\R $, then the statement remains valid with $k_i\in (-2\sqrt{2}(\q+1),2\sqrt{2}(\q+1))\cap \Z$.
In what follows we will take
$$
\U_\q=\big\{\Proj_F(-z{\bf e}_i)\big\}_{i=1}^\q,\quad \Y_\q=\big\{V_i\big\}_{i=1}^\q,\quad\text{and}\quad
\W_{\q}=\big\{\big(-\bar z^{-1}{\bf e}_i-\mbox{$\frac{\bar z^{-1}}{|z^{-1}|}$}\delta {\bf e}_i\big)\oplus D{\bf q}_i\big\}_{i=1}^\q
$$
(see Lemma \ref{l: V and proj biorth} and Lemma \ref{l: aux uaf,mlayg}).

\smallskip

\subsection{Probabilistic properties of the perturbed dual basis}\label{subs: two types}
The purpose of this section is to estimate probability of occurence of
\eqref{event D+disccollec} for vectors $\widetilde V_1,\dots,\widetilde V_\q$ defined below.

Let $m,n,\q$ satisfy $m=n-\q$.
Let $X_1,\dots,X_{\q}$ be i.i.d.\ random vectors in $\CorR^m$
satisfying \eqref{Asmp 1} or \eqref{Asmp 2}.
Given an $m\times m$ matrix $D$ and a parameter $\kappa\in\CorR\setminus\{0\}$,
define $n$--dimensional random vectors
 \begin{equation}\label{V definition}
   \widetilde V_i:=\kappa\,{\bf e}_i\oplus DX_i,\quad i\leq \q.
 \end{equation}
 (In what follows we will take $\widetilde V_i=V_i-({\bar z^{-1}}/{|z^{-1}|})\,\delta\,{\bf e}_i$
(see \eqref{eq: V act def}), so that $\V_\q$ is a $\delta$-perturbation of ${\mathcal V}_\q$.)
Here we want  to estimate probability of the event that $\V_\q$ 
falls into a given class $\disclass(r ,{\bf p})$ and that $\disccollec(\V_\q,\delta',T')\neq \emptyset$
for appropriately chosen parameters $\delta',T'$.
The importance of such an event is explained by Proposition~\ref{p: perturbed vector existence},
where relations of the form
$$
\V_\q\in \disclass(r ,{\bf p})
\quad\mbox{ and }\quad \disccollec(\V_\q,2\delta (T+1),2T+2\sqrt{2}\q(T+1))\neq\emptyset
$$
are shown to be {\it necessary} for the geometric
condition
$$
\Ellips(\U_\q)\in\elclass(R,{\bf b})
\quad\mbox{ and }\quad U_{\q}\in \delta\,\Ellips(\U_{\q-1})+ T  B_2^E
$$
to hold.

We distinguish two cases: first, when many singular values of $D$ are relatively large,
in which case $DX$ is well spread, and, second,
when $D$ has only a few large singular values in which case $DX$ is essentially contained in a fixed low-dimensional subspace. The corresponding two main statements are Lemmas \ref{l: aux a;kfnaori} and \ref{l: aux decoupling}, which we finally summarise in Proposition \ref{p: decoupling}.

Let us start with the following corollary of the Hanson--Wright inequality proved in \cite{RV sup norm}
which will be applied to matrices $D$ of the first type.
\begin{lemma}[{\cite[Theorem~4.1(i)]{RV sup norm}}]\label{l: first type lemma}
Let $D$ be an $m\times m$ fixed matrix, let $m\geq h>p\geq 1$, and let
$X_1,\dots,X_p$ be i.i.d random vectors in $\CorR^m$
satisfying \eqref{Asmp 1} or \eqref{Asmp 2} with a parameter $K$.
Then with probability at least $1-2p e^{-c(h-p)}$ we have
$$\d\big(DX_i,\spn\{DX_j:\;j\in[p]\setminus \{i\}\}\big)\geq c \Big(\sum\limits_{\ell=h}^m s_\ell(D)^2\Big)^{1/2}\;\quad\mbox{for all }i\in[p].$$
Here, $c>0$ may only depend on $K$.
\end{lemma}

\begin{lemma}\label{l: aux a;kfnaori}
Let $m,\q\in \N$, let
$D$ be an $m\times m$ matrix with
$\sum_{\ell=\q^2}^m s_\ell(D)^2\geq (\q T')^{-2}$, $T'\geq 1$,
and let $\V_\q$ be defined by \eqref{V definition} for some $\kappa\in\CorR\setminus\{0\}$.
Fix any  $0<r <1$ and any ${\bf p}$ satisfying \eqref{p condition}.
Then
$$
\Prob\big\{\V_\q\in\disclass(r ,{\bf p})\big\}
\leq  (C\q)^{C\q}\prod\limits_{i=1}^{\q}\big(2^{p_i}T'\big)^U+2e^{-c\q^2},
$$
with $U=1$ for $\CorR=\R$ and $U=2$ for $\CorR=\C$,
for some $C,c>0$ depending only on $K$.
\end{lemma}
\begin{proof}
In view of the assumptions on the singular values of $D$ and in view of Lemma~\ref{l: first type lemma}, we get
that with probability at least $1-2e^{-c_1\q^2}$
\begin{equation}\label{eq: aux apirua;lfam}
\d\big(\widetilde V_i,\spn\{\widetilde V_j:\;j\neq i\}\big)
\geq \d\big(DX_i,\spn\{DX_j:\;j\neq i\}\big)\geq {c_1}/(\q T')
\end{equation}
for all $i\leq \q$ for some constant $c_1\in(0,1]$
which may only depend on $K$.
Note that if $2^{p_i+1}\geq {c_1}/(\q T')$ for all $i\leq \q$ then,
for sufficiently large $C'$, we get $(C'\q)^{C'\q}\prod\limits_{i=1}^{\q}\big(2^{p_i}T'\big)^U\geq 1$,
and the required probability estimate is trivial.
On the other hand, if $2^{p_i+1}< {c_1}/(\q T')$ for some $i\leq \q$ then, conditioned on \eqref{eq: aux apirua;lfam},
we get $\min\big(\max(d_i,r ),1\big)>2^{p_i+1}$ (where $d_i$'s
are defined for the sequence $\V_\q$ by \eqref{eq: sigma definition}). Hence, $\V_\q\notin\disclass(r ,{\bf p})$, and
 probability of the corresponding event is bounded from above by probability
of the complement of the event \eqref{eq: aux apirua;lfam}, i.e.\ by $2e^{-c_1\q^2}$.
The result follows.
\end{proof}

The next lemma will be used while dealing with matrices of the second type (with few large singular values).
\begin{lemma}\label{l: second type lemma}
Let $D$ be an $m\times m$ fixed matrix, and let $X$ be a
random vector in $\CorR^m$ satisfying \eqref{Asmp 1} or \eqref{Asmp 2} with a parameter $K$.
Let $R>0$, $m\geq h\geq 1$, and assume that $ (\sum_{\ell=h}^m s_l(D)^2)^{1/2}\leq \psi$ for some $\psi>0$.
Then with probability at least $1-2m e^{-c \psi^{-2} R^{-2}}$ we have
$$\|D^* Y\|_2\geq {\psi R}/{h}\quad \mbox{ for all vectors $Y\in\CorR^m$ with $\|Y\|_2\leq R$ and $|\langle DX,Y\rangle|=1$}.$$
Here, $c>0$ may only depend on $K$.
\end{lemma}
\begin{proof}
Let $D=\sum\limits_{i=1}^m s_i(D) {\bf u}_i\, {\bf v}_i^*$
be the singular value decomposition of $D$ (with ${\bf u}_i$ and ${\bf v}_i$ being the normalized left and right singular vectors, respectively).
Applying a concentration inequality for sums of subgaussian random variables (see, for example, \cite[Proposition~5.10]{V 12}), we get
$$\Prob\big\{|\langle X,{\bf v}_i\rangle|\geq t\big\}\leq 2e^{-c_1t^2},\quad t>0,$$
for some $c_1=c_1(K)>0$.
Taking the union bound over all $i$, we obtain that for any $t>0$ the event
$$\Event_t:=\big\{|\langle X,{\bf v}_i\rangle|\leq t\;\;\mbox{ for all }i\leq m\big\}$$
has probability at least $1- 2m e^{-c_1t^2}$.
If we denote by $\Proj$ the orthogonal projection onto the span of vectors ${\bf v}_h,{\bf v}_{h+1},\dots,{\bf v}_m$ then, on the event $\Event_t$, we have $\|D\Proj X\|_2\leq \psi t$.

Let $Y$ be any vector in $\CorR^m$ with $\|Y\|_2\leq R$. Then
$$|\langle DX,Y\rangle|\leq R\|D\Proj X\|_2+|\langle D(\Idm-\Proj)X,Y\rangle|.$$
Setting $t:={1}/(2\psi R)$, we obtain from the above that everywhere on $\Event_t$ we have
$$|\langle D(\Idm-P)X,Y\rangle|\geq {1}/{2}\quad \mbox{ for all vectors $Y\in\CorR^m$ with $\|Y\|_2\leq R$ and $|\langle DX,Y\rangle|=1$}.$$
Note that condition $|\langle D(\Idm-\Proj)X,Y\rangle|\geq {1}/{2}$ immediately implies
$$
s_i(D)|\langle X,{\bf v}_i\rangle\, \langle {\bf u}_i,Y\rangle|\geq {1}/(2h)\quad\text{ for some}\quad i=i(Y)<h.
$$
On the other hand, everywhere on $\Event_t$ we have $|\langle X,{\bf v}_i\rangle|\leq t$ for all $i<h$. Thus, we get that, conditioned on $\Event_t$ (with the above choice of $t$), we have
$$s_i(D)|\langle {\bf u}_i,Y\rangle|\geq {\psi R}/{h}\quad \mbox{for some $i=i(Y)<h$ whenever $\|Y\|_2\leq R$ and $|\langle DX,Y\rangle|=1$}.$$
But the leftmost condition immediately gives $\|D^* Y\|_2\geq {\psi R}/{h}$. The result follows.
\end{proof}

The above lemma emphasizes an important property of anisotropic vectors with few principal components.
Consider first the opposite situation, when all singular values of $D$ are roughly comparable.
Say, if $D=\Idm$ then the relation $\langle DX,Y\rangle=\langle X,Y\rangle=1$
holds for vector $Y=\frac{X}{\|X\|_2^2}$ which has Euclidean norm of order $O(m^{-1/2})$ with large probability.
Thus, in this case there is no strong lower bound for $\|D^* Y\|_2$.
On the other hand, when $DX$ has few principal components, that is, when most singular values of $D$ are small,
the condition $\langle DX,Y\rangle=\langle X,D^* Y\rangle=1$ does guarantee (with large probability) that
$\|D^* Y\|_2$ is large. This happens because $DX$ is almost contained in a fixed low-dimensional subspace,
and therefore no dependence on $m$ appears in the lower bound for $\|D^* Y\|_2$.

\bigskip

In the next lemma we estimate probability of an event of type
\eqref{event D+disccollec} (with $\W_\q$ replaced with $\V_\q$)
in the case when the matrix $D$ from the definition of $\V_\q$ has only few principal components.
It encapsulates the decoupling procedure which was briefly mentioned in Remark~\ref{rem: more com on ups},
and in itself is the central (technical) element of the subsection.
Let us describe the essence of the argument leaving out all details.
The condition that the random set $\disccollec(\V_\q,\delta',T')$ is non-empty implies that for some
choice of (admissible) parameters $J$, $(k_j)$ and $\sigma$, the random vector
$Z=Z_{\V_\q,J,(k_j),\sigma}$ defined by \eqref{Z}, satisfies \eqref{|Z|}.
In certain sense, those conditions tell us that while the norm of $Z$
is well bounded from above, its scalar products with $\widetilde V_i$, $i\neq\q$, are ``small''.
Denoting by $P$ the coordinate projection of $Z$ onto the last $n-\q$ coordinates,
we may restate the conditions by saying that $|\langle PZ,D X_\q\rangle|$ is ``sufficiently large''
while $\langle PZ, DX_i\rangle=\langle D^*PZ, X_i\rangle$, $i\neq \q$, are ``relatively small''.
However, in view of Lemma~\ref{l: second type lemma}, the former implies (with very large probability)
that $\|D^* PZ\|_2$ is well bounded from below. Thus, we get that for vector $D^*PZ$ of not-too-small norm,
the inner products of $D^*PZ$ with the i.i.d.\ vectors $X_i$, $i\neq \q$, are small. This would
immediately imply a satisfactory upper bound on the probability of the event $\{\disccollec(\V_\q,\delta',T')\neq \emptyset\}$
via standard anti-concentration estimates, however, the principal problem
here is that the vector $D^*PZ$ depends on $X_i$'s (in fact, is an explicit function of $X_i$'s given by our choice of parameters).
The decoupling procedure is designed to deal with this issue.
In the actual proof, we will consider a collection of vectors $D^*PP_iZ$,
where $P_i$ is the projection onto
$\spn\{\widetilde V_\q;\widetilde V_{\sigma(j)},j< i\}$.
The property of the vector $Z$ which we use is that $P_i Z$ is measurable with respect to the sigma-field
generated by $\widetilde V_\q;\widetilde V_{\sigma(j)},j< i$, for every $i\leq \q$.
The probability estimate for the event in question
will be obtained by sequentially taking $\langle D^*PP_iZ,X_{\sigma(i)}\rangle$
after conditioning on a realization of $\widetilde V_\q;\widetilde V_{\sigma(j)},j< i$.
The relations between the scalar products $\langle D^*PP_iZ,X_{\sigma(i)}\rangle$
and $\langle Z,\widetilde V_{\sigma(i)}\rangle$ are determined, in particular, by the magnitudes of the distances
from $\widetilde V_{\sigma(i)}$ to the span of $\widetilde V_{\sigma(j)},j< i$,
which are in turn ``encoded'' within the condition ``$\V_\q\in\disclass(r ,{\bf p})$''.
This way, our probability bounds will be expressed in terms of the parametric vector ${\bf p}$.


\begin{lemma}\label{l: aux decoupling}
Let $n\in \N$, let $\q\geq \log n$, $m=n-\q$; let
$D$ be an $m\times m$ matrix
with $\sum_{\ell=\q^2}^m s_\ell(D)^2< (\q T')^{-2}$, and let $\kappa\in\CorR$,
$|\kappa|,\delta'\geq n^{-1}$, $T'\geq 1$ be some parameters
such that
\begin{equation}\label{param}
|\kappa| T'\leq {1}/{2}.
\end{equation}
Fix any $0<r <1$ and any ${\bf p}$ satisfying \eqref{p condition} and consider the event
\begin{align*}
  \Event :=\big\{\V_\q\in\disclass(r ,{\bf p})\quad\text{and}\quad &\disccollec(\V_\q,\delta',T')\neq\emptyset
  \big\}.
\end{align*}
Then for some $C,c>0$ depending only on $K$ we have
$$
\Prob(\Event)\leq
2e^{-c\q^2}+(C\q)^{C\q}\prod\limits_{i=1}^{\q-1}\sup\limits_{Y_i}
\Prob\big\{\min(|\langle X_{1},D^* Y_i \rangle|,T')\leq \eta_i\big\},
$$
where $\eta_i:=2\sqrt{2}(\q+1)\delta'+2^{p_i+1}T'+|\kappa|\, T'$,
and for each $i\leq \q-1$ the supremum is taken over all vectors $Y_i\in \CorR^m$
with $\|Y_i\|_2\leq T'$ and $\|D^* Y_i \|_2\geq \q^{-3}$.
\end{lemma}

\begin{proof}
Set $h:=\q^2$.
It follows from the definition of $\disccollec(\V_\q,\delta',T')$ (see (\ref{Z}) -- (\ref{|Z|}))
that there exists a universal constant $\tilde C>0$ such that
\begin{align*}
\Prob(\Event)\leq (\tilde C\q)^{\tilde C\q}\max_{\sigma}\max_{J\subset [\q]\setminus\{\sigma^{-1}(\q)\}}\max_{(k_j)_{j\in J}}
\Prob\big(\Event_{J,(k_j),\sigma}\big),
\end{align*}
where
\begin{equation*}
  \Event_{J,(k_j),\sigma}:=\big\{\V_\q\in\disclass(r ,{\bf p})\quad\text{and}\quad\mbox{$Z_{\V_\q,J,(k_j),\sigma}$ satisfies \eqref{|Z|}}\quad\text{and}\quad
  \sigma_{\V}=\sigma\big\},
\end{equation*}
vectors $Z_{\V_\q,J,(k_j),\sigma}$ are defined by \eqref{Z} and the maximum is taken over all
{\it deterministic} permutations $\sigma$ of $[\q]$,
all subsets $J\subset [\q]\setminus\{\sigma^{-1}(\q)\}$, and
all sequences of numbers $(k_j)_{j\in J}$ satisfying \eqref{eq: k definition}.

Fix any admissible $\sigma$, $J$ and $(k_j)_{j\in J}$.
We have:
\begin{align*}
\Event_{J,(k_j),\sigma}=\big\{&\V_\q\in\disclass(r ,{\bf p})\quad\text{and}\quad\|Z \|_2\leq T'\quad
\mbox{and }\;\;\\
&|\langle  Z ,\widetilde V_{\sigma(i)}\rangle|\leq 2\sqrt{2}(\q+1)\delta',\; i\in [\q]\setminus\{\sigma^{-1}(\q)\},\;\;
\mbox{and}\quad
\sigma_{\V}=\sigma\big\},
\end{align*}
where
$Z =Z _{\V_\q,J,(k_j),\sigma}$ is the random vector {\it uniquely} defined by the conditions
\begin{align*}
  &Z \in\spn(\V_{\q}),\;\langle Z,\widetilde V_\q\rangle =1,\\
  &\langle Z ,\widetilde{V}_{\sigma(i)}\rangle=
\begin{cases}
k_i\delta',&\mbox{if }i\in J,
\\
\langle Z ,\Proj_{\spn\{\widetilde{V}_{\sigma(j)},\,j<i\}}(\widetilde{V}_{\sigma(i)})\rangle,&\mbox{if }i\in [N]\setminus (J\cup
\{\sigma^{-1}(\q)\}).
\end{cases}
\end{align*}
Now to prove the lemma it is enough to show that
$$\Prob(\Event_{J,(k_j),\sigma})\le(C'\q)^{C'\q}\prod\limits_{i=1}^{\q-1}
\sup\limits_{Y_i}
\Prob\big\{\min(|\langle X_{1},D^* Y_i \rangle|,T')\leq \eta_i\big\}+2e^{-c\q^2}.$$

Let $P:\CorR^n\to \CorR^m$ be the coordinate projection onto the last $m$ coordinates and
for each $i\leq \q$ let $ P _i:\CorR^n\to\CorR^n$ be the orthogonal projection onto
$\spn\{\widetilde V_{\q};\widetilde V_{\sigma(j)},\,j<i\}$.
Define the events
\begin{align*}
\Event'&:=\big\{\|Z \|_2>T'\;\;\mbox{ or }\;\;
|\langle  Z ,\widetilde V_{\sigma(i)}\rangle|> 2\sqrt{2}(\q+1)\delta'\mbox{ for some } i\in [\q]\setminus\{\sigma^{-1}(\q)\}
\big\}\quad\text{ and}\\
\Event''&:=\big\{ \|D^* P P _iZ  \|_2\geq\q^{-3}\,\,\text{for all}\,\,i\neq \sigma^{-1}(\q)\big\},
\end{align*}
so that on
$
(\Event'\cup\Event'')^c$ we have $\|Z \|_2\le T'$,
$|\langle  Z ,\widetilde V_{\sigma(i)}\rangle|\leq \sqrt{2}(\q+1)\delta'$ for all $i\in[\q]\setminus\{\sigma^{-1}(\q)\}$ and
there exists a (random) index $i_0\neq \sigma^{-1}(\q)$ such that $\|D^* P P _{i_0}Z  \|_2<\q^{-3}$.
Since
$P\widetilde V_i=DX_i$ and $\langle  P _i Z ,\widetilde V_{\q}\rangle=\langle Z ,\widetilde V_{\q}\rangle$, we have
$$
\langle P P _iZ  ,DX_{\q}\rangle
=\langle Z ,\widetilde V_{\q}\rangle-\bar\kappa( P _iZ )_{\q},
$$
and by the above conditions and (\ref{param}), everywhere on $(\Event')^c$ we have
\begin{equation*}
|\langle P P _iZ  ,DX_{\q}\rangle|
\geq |\langle Z ,\widetilde V_{\q}\rangle|-|\kappa|\,\|Z \|_2\geq
1-|\kappa|T'\ge 1/2.
\end{equation*}
Hence, everywhere on $(\Event'\cup\Event'')^c$
the vector ${\bf y}:={P P _{i_0}Z  }/{\langle P P _{i_0}Z  ,DX_{\q}\rangle}$, satisfies
$$
\|{\bf y}\|_2\le 2T',\quad |\langle {\bf y} ,DX_{\q}\rangle|=1,\quad\text{and}\quad \|D^* {\bf y}\|< 2N^{-3}.
$$
Now  Lemma~\ref{l: second type lemma} with $\psi={1}/(\q T')$ and $R=2T'$ implies
$$
\Prob\big((\Event'\cup\Event'')^c\big)
\leq 2 e^{-c_2\q^2}
$$
for some $c_2>0$ depending only on the subgaussian moment.
Thus,
\begin{align*}
\Prob(\Event_{J,(k_j),\sigma})&\leq \Prob(\Event_{J,(k_j),\sigma}\cap \Event')
+\Prob(\Event_{J,(k_j),\sigma}\cap \Event'')+\Prob\big((\Event'\cup\Event'')^c\big)\leq \Prob(\Event_{J,(k_j),\sigma}\cap \Event'')+2 e^{-c_2\q^2}.
\end{align*}
On $\Event_{J,(k_j),\sigma}$, in view of Definition~\ref{page: disclass}
of the class $\disclass(r ,{\bf p})$, we have
\begin{equation}\label{bounds}
\begin{split}
&\|Z \|_2\le T',\quad |\langle \widetilde V_{\sigma(i)},Z \rangle|\leq 2\sqrt{2}(\q+1)\delta',\quad\text{ and }\\
&\min\big(\d\big(\widetilde V_{\sigma(i)},\spn\{\widetilde V_{\q};
\widetilde V_{\sigma(j)}:\;j<i\}\big),\,1\big)\leq 2^{p_i+1},\,\, \forall i\leq \q.
\end{split}
\end{equation}
Also by the definitions of $\widetilde V_{\sigma(i)}$, $P$, and $P_i$ we have
\begin{align*}
 \langle X_{\sigma(i)},D^* P P _iZ  \rangle&=\langle P\widetilde V_{\sigma(i)}, P _iZ  \rangle
 =\langle \widetilde V_{\sigma(i)}, P _iZ  \rangle-\kappa(P_iZ )_{\sigma(i)}
 \\
 &=\langle \widetilde V_{\sigma(i)}, Z  \rangle-\langle ({\rm I}-P_i)\widetilde V_{\sigma(i)}, Z  \rangle-\kappa(P_iZ )_{\sigma(i)}.
\end{align*}
Hence by (\ref{bounds}), for all $i\neq \sigma^{-1}(\q)$,
\begin{align*}
 |\langle X_{\sigma(i)},D^* P P _iZ  \rangle|\le 2\sqrt{2}(\q+1)\delta'+\d\big(\widetilde V_{\sigma(i)},\spn\{\widetilde V_{\q};
\widetilde V_{\sigma(j)}:\;j<i\}\big)
T'+\kappa T'
\end{align*}
and
\begin{align*}
&\min\big(|\langle X_{\sigma(i)},D^* P P _iZ  \rangle|,\,T'\big)\le 2\sqrt{2}(\q+1)\delta'+\kappa T'\\
&\hspace{2cm}+
\min\big(\d\big(\widetilde V_{\sigma(i)},\spn\{\widetilde V_{\q};
\widetilde V_{\sigma(j)}:\;j<i\}\big) T',\, T'\big)\le \eta_i.
\end{align*}
 Introducing events
$$
\Event_i:=\big\{\min(|\langle X_{\sigma(i)},D^* P P _iZ  \rangle|,T')\leq \eta_i
\;\mbox{ and }\;\|D^* P P _iZ  \|_2\geq \q^{-3}\;\mbox{ and }\;
\| P _iZ  \|_2\leq T'\big\},
$$
we get from the above
$$\Event_{J,(k_j),\sigma}\cap \Event''\subset \bigcap\limits_{i\neq \sigma^{-1}(\q)}
\Event_i.
$$
The following observation is crucial:
\smallskip

\noindent {\bf{Claim.}} {\it For every $i\in [\q]\setminus\{\sigma^{-1}(\q)\}$, the event
$\Event_i$ is measurable with respect to the sigma-field generated by $\{X_{\q};X_{\sigma(j)},\;j\leq i\}$.}
\smallskip

\noindent Indeed, it is obvious that $\Event_i$ is measurable with respect to the sigma-field generated by $X_{\sigma(i)}$
and $ P _iZ $. Further, the vector $ P _iZ $ satisfies
$\langle \widetilde V_{\sigma(j)}, P _i Z \rangle
=\langle \widetilde V_{\sigma(j)},Z \rangle$, $j\in\{1,\dots,i-1;\sigma^{-1}(\q)\}$,
whence, by the definition of $Z$,
\begin{equation*}
 \langle P _i Z ,\widetilde{V}_{\q}\rangle=1,\quad
  \langle P _i Z ,\widetilde{V}_{\sigma(j)}\rangle=
\begin{cases}
k_i\delta',&\mbox{if }j\in J\cap[i-1],
\\
\langle P _i Z ,\Proj_{\spn\{\widetilde{V}_{\sigma(\ell)},\,\ell<j\}}(\widetilde{V}_{\sigma(j)})\rangle,&\mbox{if }j\in
[i-1]\setminus (J\cup\{\sigma^{-1}(\q)\}).
\end{cases}
\end{equation*}
These conditions, together with the linear independence of $\widetilde V_j$'s,
imply that $ P _iZ $ is {\it uniquely} determined by $\{X_{\q};X_{\sigma(j)},\;j<i\}$, and the claim follows.

Applying the above claim, we get
\begin{align*}
\Prob(\Event_{J,(k_j),\sigma}\cap \Event'')&\le
 \Prob\Big(\bigcap\limits_{i\neq \sigma^{-1}(\q)}\Event_i\Big)
 \le
 \prod\limits_{i\neq \sigma^{-1}(\q)}\esssup
\Prob\big(\Event_i\mid X_{\q};X_{\sigma(1)},\dots,X_{\sigma(i-1)}\big).
\end{align*}
Finally, we estimate
$$\esssup
\Prob\big(\Event_i\mid X_{\q};X_{\sigma(1)},\dots,X_{\sigma(i-1)}\big).$$
Fix any realization of $X_{\q};X_{\sigma(1)},\dots,X_{\sigma(i-1)}$.
If $\| P _iZ  \|_2> T'$ then the conditional probability of $\Event_i$
given this realization of $X_{\q};X_{\sigma(1)},\dots,X_{\sigma(i-1)}$, is equal to zero.
Otherwise, the conditional probability can be bounded from above by
$$\sup\limits_{Y_i}\Prob\big\{\min(|\langle X_{\sigma(i)},D^* Y_i \rangle|,T')\leq \eta_i\big\},$$
where the supremum is taken over all vectors $Y_i\in \CorR^m$
with $\|Y_i\|_2\leq T'$ and $\|D^* Y_i \|_2\geq \q^{-3}$.
The result follows.
\end{proof}

The next proposition is the main result of the subsection, obtained by combining the above Lemmas~\ref{l: aux a;kfnaori}
and~\ref{l: aux decoupling} with anti-concentration statements for random vectors with independent components.

\begin{prop}\label{p: decoupling}
Let
$n\in \N$, let $\q\geq \log n$, $m=n-\q$; let
$D$ be an $m\times m$ matrix, and let $\kappa\in\CorR\setminus\{0\},\delta'\geq n^{-1}$, $T'\geq 1$ be some parameters
such that
$$
|\kappa| T'\leq {1}/{2}.
$$
Assume that 
for any vector $Y\in\CorR^m$ with $\|Y\|_2\leq T'$ and $\|D^* Y\|_2\geq \q^{-3}$,
we have
\begin{itemize}
\item $\LCD_{c_{\text{\tiny\ref{th: LCD of small}}} \sqrt{m},c_{\text{\tiny\ref{th: LCD of small}}}}\big(\frac{D^* Y}{\|D^*Y\|_2}\big)
\geq \frac{c_{\text{\tiny\ref{th: LCD of small}}}m}{NT'}$,
if $\CorR=\R$;
\item $\frac{D^* Y}{\|D^*Y\|_2}\in\Incomp_m(\zeta,\zeta)$ for some $\zeta>0$ if $\CorR=\C$.
\end{itemize}
Fix any  $0<r <1$ and any ${\bf p}$ satisfying \eqref{p condition} and consider the event
\begin{align*}
  \Event :=\big\{\V_\q\in\disclass(r ,{\bf p})\quad\text{and}\quad &\disccollec(\V_\q,\delta',T')\neq\emptyset
  \big\}.
\end{align*}
Then
$$
\Prob(\Event)\leq  (C\q)^{C\q}\prod\limits_{i=1}^{\q}\big(\delta'+|\kappa|\, T'+2^{p_i}T'+m^{-1/U}\big)^U+2e^{-c\q^2},
$$
where $U=1$ for $\CorR=\R$ and $U=2$ for $\CorR=\C$, and
$C,c>0$ depend only on the subgaussian moment $K$ and (in the complex case) the parameter $\zeta$.
\end{prop}
\begin{proof}
We consider two cases. First, if
$\sum_{\ell=\q^2}^m s_\ell(D)^2\geq (\q T')^{-2}$ then the statement immediately follows from Lemma~\ref{l: aux a;kfnaori}.

Otherwise, $\sum_{\ell=\q^2}^m s_\ell(D)^2< (\q T')^{-2}$.
Then, applying Lemma~\ref{l: aux decoupling}, we get
$$
\Prob(\Event)\leq
(C\q)^{C\q}
\prod\limits_{i=1}^{\q-1}\sup\limits_{Y_i}
\Prob\big\{\min(|\langle X_{1},D^* Y_i \rangle|,T')\leq \eta_i\big\}+2e^{-c\q^2},
$$
where
$\eta_i:=2\sqrt{2}(\q+1)\delta'+2^{p_i+1}T'+|\kappa|\, T'$,
and each supremum is taken over all $Y_i\in\CorR^m$ satisfying
$\|Y_i\|_2\leq T'$ and $\|D^* Y_i \|_2\geq \q^{-3}$.

Fix any $i\leq \q-1$.
If $T'\leq \eta_i$ then we have
$$\sup\limits_{Y_i}
\Prob\big\{\min(|\langle X_{1},D^* Y_i \rangle|,T')\leq \eta_i\big\}\le 1\le   (\eta_i /(T'))^U\le (\q\eta_i)^U.$$
If $T'\geq \eta_i$ then for every admissible $Y_i$ we have
$$
\Prob\big\{\min(|\langle X_{1},D^* Y_i \rangle|,T')\leq \eta_i\big\}
\leq \Prob\big\{|\langle X_{1},D^* Y_i \rangle|\leq \eta_i\big\}.
$$
We will bound the latter using anti-concentration estimates for isotropic vectors.

In the case $\CorR=\R$, by the assumptions of the proposition, for every admissible $Y_i$ the vector
$\frac{D^* Y}{\|D^* Y\|_2}$ has $LCD$ of order at least $\frac{m}{NT'}$.
Hence,
applying Theorem~\ref{th: 4.2},
$$\Prob\big\{|\langle X_{1},D^* Y_i \rangle|\leq \eta_i\big\}
\leq \widetilde C\frac{\eta_i}{\|D^* Y_i\|_2}+\widetilde CNT' m^{-1}.$$

In the case $\CorR=\C$, the vector
$\frac{D^* Y}{\|D^* Y\|_2}$ is incompressible.
Hence, applying Theorem~\ref{l: sb compl incomp},
we get
$$\Prob\big\{|\langle X_{1},D^* Y_i \rangle|\leq \eta_i\big\}
\leq \widetilde C\bigg(\frac{\eta_i}{\|D^* Y_i\|_2}\bigg)^2+\widetilde C m^{-1}$$
for some $\widetilde C>0$ depending only on $\zeta$ and the subgaussian moment $K$.

It remains to recall in both cases that $\|D^* Y_i \|_2\geq \q^{-3}$. The statement follows.
\end{proof}

\subsection{The cumulative estimate for a fixed coordinate subset
}
\label{subs: deloc estimates}

In this subsection we estimate probability of the geometric
relations (a) and (b) from the conclusion of Lemma \ref{l: two conditions},
using the results of the previous subsections. Our goal is to bound probability of the event
\begin{equation}\label{EventE}
\big\{\exists\,\n\in S^{n-1}(\CorR):\;\|(\Az )\n\|_2\leq n^{-2}\mbox{ and }|v_{\ell}|\leq\theta,\;\ell\leq \q-1;
|v_{\q }|\geq\beta\big\},
\end{equation}
given as an assumption of Lemma \ref{l: two conditions} (taking $\tau=n^{-2}$).
The estimate for \eqref{EventE} is obtained in Proposition~\ref{prop: mn}.

Here, we finalize the decoupling procedure whose objective was to separate analysis of conditions (a) and (b) in Lemma \ref{l: two conditions}.
In Subsection~\ref{subs: biorth}, we introduced a splitting of ellipsoids into classes $\elclass(R,{\bf b})$, and
thus reduced analysis of an event of the form
$$\{\mbox{(a) and (b) from Lemma~\ref{l: two conditions} hold}\}$$
to estimating probabilities of events
$$\{\mbox{(a) and (b) from Lemma~\ref{l: two conditions} hold, and $\Ellips(\U_\q)\in\elclass(R,{\bf b})$}\}$$
for all admissible vectors ${\bf b}$, where $\U_\q=\big(\Proj_F(-z {\bf e}_1),\dots,\Proj_F(-z {\bf e}_{\q})\big)$. The advantage of this step is that it allows us to bound the latter by the {\it product} of probability of
$\{\mbox{(a) and $\Ellips(\U_\q)\in\elclass(R,{\bf b})$}\}$
and the conditional probability of (b) given that $\Ellips(\U_\q)\in\elclass(R,{\bf b})$.

We start by computing probability for event (b) conditioning on a realization of the ellipsoid $\Ellips(\U_\q)$
(see Lemma \ref{l: aux aponfa;sa32r}).
This part does not depend on the previous two subsections, our proof of
Lemma \ref{l: aux aponfa;sa32r} is based on standard anti-concentration estimates which allow to represent the probability bound in terms of the ``measure'' of $\Ellips(\U_\q)$
given by the parameteric vector ${\bf b}$ for which we have $\Ellips(\U_\q)\in\elclass(R,{\bf b})$
(see also discussion in Subsection \ref{ss:intro5} of the introduction).
It is natural to expect that the probability should be roughly proportional to the ``size'' of ellipsoids from $\elclass(R,{\bf b})$;
in fact, disregarding small additive error terms, the estimate obtained in Lemma \ref{l: aux aponfa;sa32r} is proportional to the
quantity $2^{U\sum_{i=1}^\q b_i}$, where $U=1$ in the real and $U=2$ in the complex case.

At the second step, we bound probability of ``(a) and $\Ellips(\U_\q)\in\elclass(R,{\bf b})$''.
Significant preparatory work has been done in the previous two subsections to handle this event.
In Subsection~\ref{subs: biorth}, the condition was expressed in terms of a perturbation of the dual basis in Proposition~\ref{p: perturbed vector existence}; in particular,
it was shown that the condition implies that the perturbed dual basis belongs to a class $\disclass(r ,{\bf p})$
for a parametric vector ${\bf p}$ satisfying $2^{\sum_{i=1}^{\q}p_i}\leq (4\q )^{\q/2}\,2^{-\sum_{i=1}^{\q}b_i}$.
The latter condition was further analyzed in the probabilistic setting in Subsection~\ref{subs: two types}, and, for fixed parametric vector,
probability bounds roughly proportional to $2^{U\sum_{i=1}^{\q}p_i}$ (disregarding error terms) were obtained in Proposition~\ref{p: decoupling}.
Essentially by combining Proposition~\ref{p: perturbed vector existence} and Proposition~\ref{p: decoupling} with some simple computations,
we obtain a bound for the event in question in Lemma~\ref{l: aux uaf,mlayg}, which is roughly proportional to $2^{-U\sum_{i=1}^{\q}b_i}$.

The culminating point of the subsection --- Proposition~\ref{prop: mn} --- combines the estimates from Lemma~\ref{l: aux aponfa;sa32r}
and Lemma~\ref{l: aux uaf,mlayg}. The product of the probability bounds from those lemmas then cancels the terms depending on ${\bf b}$.
This cancellation, which is crucial in obtaining optimal delocalization results, is a consequence of the trivial duality relation
$\prod_{i=1}^\q s_i(\U_\q)\cdot \prod_{i=1}^\q s_i(\Y_\q)=1$ between semi-axes of ellipsoids of an arbitrary biorthogonal system.
Indeed, our definition of classes $\disclass(r ,{\bf p})$ for dual sequences is based on distances $d_i$ and a special permutation of the perturbed dual sequence (see Definitions~\ref{page: disclass} and~\ref{d:di}).
Those distances, in turn, are directly related to magnitudes of the semi-axes for the dual ellipsoid (see Lemma~\ref{l: W dist est}).
The fact that we are working with a perturbed dual sequence and that we ``truncate'' very large and very small semi-axes of ellipsoids in our analysis, inevitably
introduces error multiples to our variant of the duality relation. However, these error multiples can be controlled so that the resulting no-gap delocalization bound stays within a polylogarithmic range of optimal estimates.

\begin{lemma}\label{l: aux aponfa;sa32r}
Given $n\in\N$ and $\log^2 n\leq \q$, let $F$ be a fixed $\q $--dimensional subspace of $\CorR^n$,
and $\U_\q$ be a fixed sequence of vectors in $F$.
Let $n\geq R\geq 1$. Let
the
$\q $--dimensional ellipsoid
$\Ellips=\Ellips(\U_\q)\subset F$ be defined by \eqref{eq: aux el def}.
Assume that $\Ellips\in\elclass(R,{\bf b})$
for some admissible sequence of integers ${\bf b}$.
Further, let $Z$ be a random vector in $\CorR^n$ satisfying \eqref{Asmp 1} or \eqref{Asmp 2}.
Then for any $\delta\geq (R\q)^{-1}$, $t>0$ and any fixed vector $Y$ we have
\begin{align*}
\Prob\big\{\Proj_F(Z+Y)\in \delta\Ellips+t B_2^F\big\}
\leq \q^{C\q}
\bigg(\frac{t+\delta}{\delta}\bigg)^{U\q} 2^{U\sum_{i=1}^\q b_i}\,\cf(\Proj_F(Z),\delta)
+2e^{-c\q^2}
\end{align*}
for some $C,c>0$ depending only on $K$,
where $U=1$ for $\CorR=\R$ and $U=2$ for $\CorR=\C$.
\end{lemma}
\begin{proof}
Denote by $B$ the Euclidean ball in $F$ of radius $\q$ centered at $\Proj_F(Y)$. Applying the Hanson--Wright inequality
\cite{HW, RV HS}, we get that
$$
\Prob\{\Proj_F(Z+Y)\in B\}\ge 1-2e^{-c_1\q^2}
$$
for some $c_1>0$ depending only on $K$, whence
$$\Prob\big\{\Proj_F(Z+Y)\in \delta\Ellips+t B_2^F\big\}\leq
\Prob\big\{\Proj_F(Z+Y)\in (\delta\Ellips+t B_2^F)\cap B\big\}+2e^{-c_1\q^2}.
$$
To estimate $\Prob\big\{\Proj_F(Z+Y)\in (\delta\Ellips+t B_2^F)\cap B\big\}$, note that for any $\tau>0$
and any covering of $(\delta\Ellips+t B_2^F)\cap B$ by $L$ translates of $\tau B_2^F(\CorR)$,
we have
$$
\Prob\big\{\Proj_F(Z+Y)\in (\delta\Ellips+t B_2^F)\cap B\big\}\leq L\,\cf(\Proj_FZ ,\tau).
$$
The cardinality of a minimal covering of $(\delta\Ellips+t B_2^F)\cap B$ in the real case $\CorR=\R$
can be estimated from above with help of a standard volumetric argument by the ratio
$$\frac{\Vol_\q(\delta\Ellips\cap B+(t+\frac{\tau}{2})B_2^F(\R))}{\Vol_\q(\frac{\tau}{2}B_2^F(\R))},$$
which in turn is bounded above by $(C\q)^{\q/2}\tau^{-\q}\prod\limits_{i=1}^\q \big(\min(\delta s_i(\U_\q),\q)+t+\frac{\tau}{2}\big)$.
In the complex case $\CorR=\C$, applying Lemma~\ref{l: real ell},
we reduce the problem to estimating cardinality of a minimal covering of $(\delta\real(\Ellips)+t\real(B_2^F(\C)))\cap \real(B)$
by translates of $\tau\real(B_2^F(\C))$. Repeating the above argument, we get
at upper bound
$$(C\q)^{\q}\tau^{-2\q}\prod\limits_{i=1}^\q \Big(\min(\delta s_i(\U_\q),\q)+t+\frac{\tau}{2}\Big)^2.$$
Thus, for any $\tau>0$
\begin{align*}
\Prob\big\{&\Proj_F(Z+Y)\in (\delta\Ellips+t B_2^F(\CorR))\cap B\big\}\\
&\leq (C'\sqrt{\q}/\tau)^{U\q}\cf(\Proj_FZ ,\tau)\prod\limits_{i=1}^\q\max\big(\tau+t,\min(\delta s_i(\U_\q),\q)\big)^U,
\end{align*}
where $U=1$ for $\CorR=\R$ and $U=2$ for $\CorR=\C$.
Take $\tau:=\delta $. Then by condition $\delta\geq (R\q)^{-1}$ and  Definition \ref{page: elclass}, we have
\begin{align*}
  \max\big(\delta+t,\min(\delta s_i(\U_\q),\q)\big)&\le\delta \max\big((\delta+t)/\delta,\min(s_i(\U_\q),R\q^2)\big)
  \\
  &\le\delta \max\big((\delta+t)/\delta,\q^2\min(s_i(\U_\q),R)\big)
    \\
  &\le\delta \max\big((\delta+t)/\delta,\q^22^{b_i}\big)\le(\delta+t) \q^2 2^{b_i},
\end{align*}
thus
\begin{align*}
\Prob\big\{&\Proj_F(Z+Y)\in (\delta\Ellips+t B_2^F(\CorR))\cap B\big\}
\leq \q^{\widetilde C\q}\bigg(\frac{t+\delta}{\delta}\bigg)^{U\q}2^{U\sum_{i=1}^\q b_i}\,\cf(\Proj_FZ ,\delta).
\end{align*}
The result follows.
\end{proof}

Now we turn to estimating probability of event (a) in Lemma \ref{l: two conditions}.
\begin{lemma}\label{l: aux uaf,mlayg}
Let $n>2$, $\mathcal M\geq 1$, $\log^2 n\leq\q\leq n^{1/4}$,
$$
100\q ^2\leq |z|\leq \mathcal M\sqrt{n}\quad\text{and}\quad\sqrt{\q}n^{-1/\oo}\leq\delta\leq \q^{-1/2},
$$
where $\oo=2$ for $\CorR=\C$ and $\oo=1$ for $\CorR=\R$. Set $T:=3\q$ and $R:=\delta^{-1} \q^{-1/2}\geq 1$.
Let $A$ be an $n\times n$ random matrix with i.i.d columns satisfying \eqref{Asmp 1} or \eqref{Asmp 2}
with a parameter $K$ and with a bounded distribution density of the entries.
Let
$$
F:=\spn\{\col_{\q+1}(A-z),\dots,\col_n(A-z)\}^\perp.
$$
Denote $\U_\q:=\big(\Proj_F(-z {\bf e}_1),\dots,\Proj_F(-z {\bf e}_{\q})\big)$, and let
ellipsoids $\Ellips=\Ellips(\U_\q)$ and $\Ellips'=\Ellips(\U_{\q-1})$ be defined by
\eqref{eq: aux el def}.
Then
for any fixed (non-random) integer sequence ${\bf b}$ satisfying \eqref{b condition}, we have
\begin{align*}
\Prob\Big\{\Proj_F(-z {\bf e}_{\q })\in \delta\Ellips'+T B_2^F(\CorR)\;\;\mbox{and}\;\;
\Ellips\in\elclass(R,{\bf b})\Big\}\le\q ^{C\q }\,2^{-\oo\sum_{i=1}^{\q }b_i},
\end{align*}
where $C>0$ may only depend on $\mathcal M$ and $K$.
\end{lemma}
\begin{proof}
Without loss of generality, we can assume that $n$ is large. Further, since by Definition \ref{page: disclass},  $2^{-\oo\sum_{i=1}^{\q }b_i}\ge R^{-\q}$ then adjusting the constant $C$ we can make the right hand side of the relation
greater than one for any choice of $C'$ such that $R\le \q^{C'}$. Hence, we can assume that
$R\ge\q^2$, so that $\delta\le{\q}^{-5/2}$. It follows from the conditions on $\delta$, $\q$, and $z$, that
$$
\text{for }\,\,\kappa:=-\frac{\bar z^{-1}}{|z^{-1}|}(|z^{-1}|+\delta)\quad\text{we have }\,\,(7\q)^{-2}\ge|\kappa|\ge\delta\ge 1/n.
$$
Let $\V_\q\subset\CorR^n$ be defined as in \eqref{V definition}:
$$
\widetilde V_i=\kappa {\bf e}'_i\oplus D{\bf q}_i=
\big(-\bar z^{-1}-\mbox{$\frac{\bar z^{-1}}{|z^{-1}|}$}\delta\big) {\bf e}'_i\oplus D{\bf q}_i,\quad i\in [\q ].
$$
Here $D=\bar z^{-1}(\widetilde A^*-\bar z)^{-1}$, and
 $\widetilde A$ denotes the $(n-\q)\times (n-\q)$ submatrix of $A $ obtained by
crossing out the first $\q$ rows and columns (see also \eqref{eq: V act def}). Note that with probability $1$, matrix
$D$ is well defined  since in view of the assumption on the distribution density of the entries all square submatrices of $\Az $
are non-singular almost everywhere.
Let the vector sequence ${\mathcal V}_\q$ be defined by \eqref{eq: V act def}.
Clearly, $\|\widetilde V_i-V_i\|_2\le\delta$, $i\in[\q]$, so that $\V_\q$ is a $\delta$--perturbation of ${\mathcal V}_\q$.
Moreover, by the construction we have $\d(\widetilde V_i,\spn\{\widetilde V_j:\;j\neq i\})\geq |\kappa|$, $\forall i\in[ \q]$, whence
$$
d_i\geq |\kappa|\ge\delta,\quad i\in[ \q]
$$
almost everywhere on the probability space, with $d_i=d_i(\V_\q)$ defined according to Definition~\ref{d:di}.
Hence, applying
Proposition~\ref{p: perturbed vector existence} with $r=\delta$
we get that everywhere on the event under consideration the set
$\disccollec(\V_\q,2\delta (T+1),2T+2\sqrt{2}\q(T+1))$ is non-empty and
there is a (random) non-increasing sequence of integers ${\bf p}$ satisfying \eqref{p condition}
such that
$\V_\q\in \disclass(\delta,{\bf p})$ and
\begin{equation}\label{eq: aux sdognalafsa}
2^{\sum_{i=1}^{\q }p_i}\leq (4\q)^{\q /2}\,2^{-\sum_{i=1}^{\q }b_i}.
\end{equation}

By Lemma~\ref{l: disclass count}, the total number of possible realizations of  ${\bf p}$ is
bounded above by $\big(2-\log_2\delta\big)^{\q }$.
Hence, to prove the lemma, it is enough to show that
for every choice of a non-random sequence ${\bf p}$ satisfying \eqref{eq: aux sdognalafsa},
the event
\begin{align*}
\Event_{\bf p}:=\big\{\V_\q\in \disclass(\delta,{\bf p})\;\;\mbox{ and }\;\;\disccollec(\V_\q,2\delta (T+1),2T+2\sqrt{2}\q(T+1))\neq\emptyset\big\}
\end{align*}
has probability at most $\q ^{C'\q } 2^{\oo\sum_{i=1}^{\q }p_i}$.
Let us check the conditions of Proposition~\ref{p: decoupling}. To this end set
$$
T':=2T+2\sqrt{2}\q(T+1)\le24\q^2\quad\text{and}\quad\delta':=2\delta(T+1)\le 8\q^{-3/2}
$$
and note first that $|\kappa|T'\leq 1/2$. Also, it follows from the lower bounds for $\delta$ and $\q$ that $\delta'\geq 1/n$.

Take now any vector $Y\in \CorR^m$ such that $\|D^* Y\|_2\geq N^{-3}$, and $\|Y\|_2\leq T'$. Let
$$
Z:={D^* Y}/{\|D^*Y\|_2}\in S^{m-1}(\CorR),\quad m:=n-\q.
$$
Since $D^*= z^{-1}(\widetilde A-z)^{-1}$, using conditions on $z$ we get
$$
\|(\widetilde A-z)Z\|_2=\|(\widetilde A-z)z^{-1}(\widetilde A-z)^{-1}Y\|_2/{\|D^*Y\|_2}=|z^{-1}|\|Y\|_2/{\|D^*Y\|_2}\le NT'.
$$
 Consider two cases: 1. $\CorR=\R$. It follows then from  Theorem~\ref{th: LCD of small} with
$ B+M=\widetilde{A}-z$
that with probability at least $1-2e^{-cn}$ we have
$$
\LCD_{c_{\text{\tiny\ref{th: LCD of small}}} \sqrt{m},c_{\text{\tiny\ref{th: LCD of small}}}}(Z)
\geq {c_{\text{\tiny\ref{th: LCD of small}}}m}/{\|(\widetilde A-z)Z\|_2} \ge{c_{\text{\tiny\ref{th: LCD of small}}}m}/{NT'}.
$$

2.   $\CorR=\C$. Since by the condition on $\q$, $\|(\widetilde A-z)Z\|_2\le NT'\le 6\sqrt{n}$, then in view of Lemma~\ref{l: incomp in ker}  with probability at least $1-2e^{-cn}$ we have  $Y\in\Incomp_m(\zeta,\zeta)$,
where $\zeta>0$ depends only on $K$ and $\mathcal M$.

Now it follows from Proposition~\ref{p: decoupling}, that
$$\Prob(\Event_{\bf p})\le(C_1\q )^{C_1\q }\prod\limits_{i=1}^{\q }\big(\delta'+|\kappa|\, T'+2^{p_i}T'+n^{-1/\oo}\big)^\oo+2e^{-c\q ^2},$$
for some $c,C_1>0$ depending only on $K$ and $\mathcal M$.
In view of Definition \ref{page: disclass} and above remarks,
$$
2^{p_i+1}\ge \min(d_i,\, 1)\ge |\kappa|\ge \delta'/8\q,\quad i\in[\q].
$$
Also $2^{-\oo\sum_{i=1}^{\q }b_i}\geq R^{-UN}\gg e^{-c\q^2}$. Summarising we get that
$\Prob(\Event_{\bf p})\le\q ^{C_2\q }2^{\oo\sum_{i=1}^{\q }p_i}$
for some $C_2>0$ depending only on $K,\mathcal M$, and the result follows.
\end{proof}

In the next proposition, using Lemma \ref{l: two conditions} and previous two lemmas we get a probability bound for the event given by (\ref{EventE}).

\begin{prop}\label{prop: mn}
Let $\mathcal M\geq 1$, $n>2$, and consider an $n\times n$ random matrix $A$
with i.i.d columns satisfying \eqref{Asmp 1}
or \eqref{Asmp 2}
with a parameter $K$.
Let $\log^2 n\leq\q\leq n^{1/4}$, $z\in\CorR$ with $100\q ^2\leq |z|\leq \mathcal M\sqrt{n}$, and let $\beta,\theta>0$
be such that
$$
n^{-1/\oo}\leq{\theta}/{\beta}\leq 1\quad\text{and}\quad \beta\geq 1/n,
$$
where $\oo=1$ for $\CorR=\R$ and $\oo=2$ for $\CorR=\C$.
Consider the event
$$\Event:=\big\{\exists\,\n\in S^{n-1}(\CorR):\;\|(\Az )\n\|_2\leq n^{-2}\mbox{ and }|v_{\ell}|\leq\theta,\;\ell\leq \q-1;
|v_{\q }|\geq\beta\big\}.$$
Then $\Prob(\Event)\leq \big(\q ^{C }{\theta}/{\beta}\big)^{\oo\q}$
for some $C>0$
depending only on $K$ and $\mathcal M$.
\end{prop}
\begin{proof}
Without loss of generality, by adding an arbitrarily small Gaussian perturbation to the matrix entries,
we can assume that the distribution density of the entries is bounded.
We can also assume that $\theta/\beta\leq 1/\q ^4$ (see the proof of Lemma \ref{l: aux uaf,mlayg}).

Let $\U_\q$, $\Ellips=\Ellips(\U_\q)$, $\Ellips'=\Ellips(\U_{\q-1})$, $T$, and $F$ be as in Lemma \ref{l: aux uaf,mlayg}.
Observe that, by the Hanson--Wright inequality \cite{HW, RV HS},
$$
\Prob\big\{\|\Proj_F\col_\ell(A)\|_2\leq T\mbox{ for all }\ell\in[\q ]\big\}\ge 1-2e^{-c_1\q^2}
$$
for some $c_1>0$ depending only on $K$.
Hence, in view of Lemma~\ref{l: two conditions}, it is sufficient to show that the event
$$\Event':=\Big\{
\Proj_F(-z {\bf e}_{\q })\in \delta\Ellips'+ 3\q B_2^F\mbox{ and }\Proj_F\col_{\q }(\Az )\in \delta\Ellips'
+2\q^{3/2} \delta B_2^F
\Big\}$$
with $\delta={\theta\sqrt{\q}}/{\beta}$, has probability at most $\big(\q ^{C' }{\theta}/{\beta}\big)^{\oo\q}$,
for some $C'>0$ depending only on $K,\mathcal M$.

Set $R:={\beta}/({\theta\q })=(\delta\sqrt{\q})^{-1}$.
Let $\mathcal Q$ be a subset of $\q$--dimensional linear subspaces of $\CorR^n$ which will be defined later.
We have
\begin{align*}
\Prob(\Event')\leq \sum_{{\bf b}}\Prob\big(\Event'\cap \{\Ellips\in \elclass(R,{\bf b})\}\;\big\vert\;F\in\mathcal Q\big)
\Prob\big\{F\in \mathcal Q\big\}
+\Prob\big\{F\notin \mathcal Q\big\},
\end{align*}
where the summation is taken over all integer sequences ${\bf b}$ satisfying \eqref{b condition}.
Further, applying Lemma~\ref{l: elclass count}, we obtain
\begin{align*}
\Prob(\Event')\leq \big(2-\log_2(\sqrt{\q} \delta)\big)^{\q }\sup_{{\bf b}}\Prob\big(\Event'_1({\bf b})
\cap \Event'_2({\bf b})\;\big\vert\;F\in\mathcal Q\big)
+\Prob\big\{F\notin \mathcal Q\big\},
\end{align*}
where
\begin{align*}
 & \Event'_1({\bf b}):=\big\{
\Proj_F(-z {\bf e}_{\q })\in \delta\Ellips'+ 3T\q B_2^F\big\}\cap\big\{\Ellips\in \elclass(R,{\bf b})\big\},
 \\
& \Event'_2({\bf b}):=\big\{\Proj_F\col_{\q }(\Az )\in \delta\Ellips'
+2\q^{3/2} \delta B_2^F\big\}\cap\big\{\Ellips\in \elclass(R,{\bf b})\big\}.
\end{align*}
Let us fix for a moment any admissible sequence ${\bf b}$.
Note that 
the event $\Event'_1({\bf b})$ is measurable with respect to the sigma-field generated by $F$. Hence,
\begin{align*}
\Prob\big(\Event'_1({\bf b})\cap \Event'_2({\bf b})\;\big\vert\;F\in\mathcal Q\big)
&=\Exp\big({\bf 1}_{\Event'_1({\bf b})}\,{\bf 1}_{\Event'_2({\bf b})}\;\big\vert\;F\in\mathcal Q\big)\\
&=\Exp\big({\bf 1}_{\Event'_1({\bf b})}\,\Exp({\bf 1}_{\Event'_2({\bf b})}\;\vert\;F)\;\big\vert\;F\in\mathcal Q\big)\\
&\le
\Prob\big(\Event'_1({\bf b})\;\big\vert\;F\in\mathcal Q\big)
\sup\limits_{F_0\in\mathcal Q}\Prob\big(\Event'_2({\bf b})\;\big\vert\;F=F_0\big).
\end{align*}
By Lemma~\ref{l: aux aponfa;sa32r}, there exists some $\widetilde C=\widetilde C(K,\mathcal M)>0$
such that conditioned on any realization $F_0$ of $F$, we have
\begin{align*}
 \Prob\big(\Event'_2({\bf b})\;\big\vert\;F=F_0\big)
\le\q ^{\widetilde C\q } \sup\limits_{Y\in F_0}\Prob\big\{\|\Proj_{F_0}Z-Y\|_2\leq\delta\big\}\, 2^{\oo\sum_{i=1}^{\q }b_i},
\end{align*}
where $Z$ is a random vector equidistributed with columns of $A$.
On the other hand, in view of Lemma~\ref{l: aux uaf,mlayg},
$$
\Prob(\Event'_1({\bf b}))\leq \q^{C'_1\q}2^{-\oo\sum_{i=1}^\q b_i}.
$$
Summarising, we get for some $ C''>0$,
$$
\Prob(\Event')
\leq
\q ^{ C''\q } \sup\limits_{F_0\in\mathcal Q}
\sup\limits_{Y\in F_0}\Prob\big\{\|\Proj_{F_0}Z-Y\|_2\leq\delta\big\}+\Prob\big\{F\notin \mathcal Q\big\},
$$
assuming that, say, $\Prob\{F\in \mathcal Q\}\geq 1/2$.

It remains to define the subset $\mathcal Q$
depending on whether we are in the real or the compex setting.

If $\CorR=\R$ then we set
$$
\mathcal Q:=\big\{F_0\subset \R^n:\;\dim F_0=\q\;\;\mbox{ and
$\LCD_{c_{\text{\tiny\ref{th: 4.3}}} \sqrt{n},c_{\text{\tiny\ref{th: 4.3}}}}(F_0)\geq
\sqrt{n}e^{c_{\text{\tiny\ref{th: 4.3}}} n/\q}$}\big\}.
$$
Note that, by the definition of $F$ and by Theorem~\ref{th: 4.3}, we have
$$\Prob\{F\in \mathcal Q\}\geq 1-e^{-c_{\text{\tiny\ref{th: 4.3}}}n}.$$
On the other hand, for every $F_0\in\mathcal Q$ we get, by Theorem~\ref{th: 4.2}:
$$
\sup\limits_{Y\in F_0}\Prob\big\{\|\Proj_{F_0}Z-Y\|_2\leq\delta\big\}
\leq
\big({C_4 \delta}/{\sqrt{\q}}\big)^{\q}+C_4^{\q} e^{-c_4n},
$$
whence
$$\Prob(\Event')\leq
\q ^{\widetilde C_5\q }\big({\theta}/{\beta}\big)^\q.
$$

In the case $\CorR=\C$,
in view of Lemma~\ref{l: incomp in ker},
with probability at least $1-2e^{-c'n}$
all unit vectors in $F$ are $(\zeta,\zeta)$--incompressible,
for some $c',\zeta>0$ depending only on $\mathcal M,K$.
Define
$$\mathcal Q:=\{F_0:\;F_0\cap S^{n-1}(\C)\subset\Incomp_n(\zeta,\zeta)\}.$$
Applying Theorem~\ref{l: sb compl incomp},
we get that whenever $F_0\in\mathcal Q$,
we have
$$\sup\limits_{Y\in F_0}\Prob\big\{\|\Proj_{F_0}Z-Y\|_2\leq\delta\big\}
\leq C_6^\q(\delta+n^{-1/2})^{2\q}.$$
Hence,
$$
\Prob(\Event')
\leq
\q ^{\widetilde C_1\q }\big({\theta}/{\beta}\big)^{2\q}.
$$
The result follows.
\end{proof}

For  complex matrices, the following lemma supplements Proposition \ref{prop: mn} covering the case of small eigenvalues. We have:
\begin{lemma}\label{l: mn}
Let $\CorR=\C$, and let $n$, $\q$, $\beta$, $\theta$, $\mathcal M$  be as in Proposition \ref{prop: mn}.
Then for $z\in\C$ with $ |z|\leq 100\q^2$, the probability of the event $\Event$
defined in Proposition \ref{prop: mn} is less than $\big(\q ^{C }{\theta}/{\beta}\big)^{2\q}$ for some $C>0$
depending only on $K$ and $\mathcal M$.
\end{lemma}
\begin{proof}
Everywhere on $\Event$, by applying Lemma \ref{l:determ} with
$k=\q-1$, $r=1$,
$B=\Az $, $\tau=n^{-2}$, $I_\theta=[\q-1]$, $J_\beta=\{\q\}$,
$F=\spn\{\col_i(B):\,i>\q\}^\perp $,
$M:=\big(\Proj_F\col_{i}(B)\big)_{i\in [\q-1]}$,  $M':=(\Proj_F\col_{\q}(B))$, we get
$$
\beta\|\Proj_F\col_{\q}(B)\|_2=\beta s_{\min}(M')\leq \theta\sqrt{\q-1}\,s_{\max}(M)+n^{-2}.
$$
Further, the Hanson--Wright inequality (see \eqref{HR} with $h=\sqrt{\q}$) implies that
$$
\Prob\big\{s_{\max}(M)\geq C'\q\big\}\leq e^{-\q^2}
$$
for some $C'$ depending only on $K,\mathcal M$.
Hence,
\begin{align*}
\Prob(\Event)
\le\Prob\big\{ \|\Proj_F\col_{\q}(B)\|_2\leq C'\theta \q^{3/2}/\beta+n^{-1}\big\}+ e^{-\q^2}.
\end{align*}
Further,
Lemma~\ref{l: incomp in ker} implies that with probability at least $1-e^{-cn}$ we have
$F\cap S^{n-1}(\C)\subset \Incomp_n(\zeta,\zeta)$ for some $\zeta,c>0$
depending only on $K,\mathcal M$.
On the other hand, for any $\q$--dimensional subspace $F_0$
such that $F_0\cap S^{n-1}(\C)\subset \Incomp_n(\zeta,\zeta)$, we have, by Theorem~\ref{l: sb compl incomp}, that
$$
\Prob\big\{ \|\Proj_{F_0}\col_{\q}(B)\|_2\leq 2 C'\theta \q^{3/2}/\beta\big\}
\le \widetilde C^\q \big(2 C'\theta \q^{3/2}/\beta+n^{-1/2}\big)^{2\q},
$$
and the result follows.
\end{proof}

\subsection{Proof of Theorems~B and~C}\label{s: proofs}

\begin{proof}
Here we prove Theorem~C, the proof of Theorem~B follows the same lines with appropriate modifications
(see Remark~\ref{rem: Th B} below).
Without loss of generality, we can assume that $n$ is large.
For a random $n\times n$ matrix $A$ with i.i.d columns satisfying \eqref{Asmp 2},
there is $\M$ depending only on $K$ such that $\|A\|\leq \M\sqrt{n}$ with probability at least $2^{-n}$
(see, for example, $\varepsilon$-net argument in \cite{LPRT}).
Thus, it is sufficient to consider only eigenvectors corresponding to eigenvalues $z$
with $|z|\le \M\sqrt{n}$.
It is enough to prove the following statement:
\medskip

{\it
There exists constant $C>0$ depending only on $K$ and $\M$ such that for any $k\in [\log^{2C}{n}, n/2]$ the event
\begin{align*}
 \Event_k:=\Big\{\exists z\in\C\mbox{ with }|z|\le \M\sqrt{n}\mbox{ and }\n\in S^{n-1}(\C)\mbox{ s.t. }&A\n=z\n
\quad\text{and}\quad
v_{n-k+1}^*< \frac{\sqrt{k}}{n\log^C n}\Big\},
\end{align*}
has probability at most $n^{-1}e^{-\log^{2} n}$.
 }
\medskip

A standard $\varepsilon$-net argument (a discretization of the disc $\{w:\,|w|\leq \M\sqrt{n}\}$
in the complex plane) implies that
$\Prob(\Event_k)\leq 8\M^2 n^{5}\max_{|z|\le \M\sqrt{n}}\Prob(\Event_k(z))$, where 
$$
\Event_k(z):=\Big\{\exists\n\in S^{n-1}(\C):\,\|(\Az )\n\|_2\leq n^{-2}\mbox{ and }
v_{n-k+1}^*<\frac{\sqrt{k}}{n\log^C n}\Big\}.
$$
In what follows, we fix any $z$ with $|z|\le \M\sqrt{n}$.
Let $\n=(v_1,\dots,v_n)\in S^{n-1}(\C)$ be any random vector.
We will estimate probability of the event
$\big\{|I_\theta|\geq k\mbox{ and }\|(\Az )\n\|_2\leq n^{-2}\big\}$,
where, similarly to the argument in the proof of Theorem~\ref{th: normal lower},
we define
$$
\theta:=\frac{\sqrt{k}}{n\log^C n} \quad\text{and}\quad I_\theta:=\{i\leq n:\;|v_i|\leq \theta\}.
$$
It follows from Lemma~\ref{l: incomp in ker} that there is $\zeta>0$ depending only on $K$ and $\mathcal M$ such
that, setting
$$
\beta:=\zeta/\sqrt{n}\quad\text{and}\quad J_\beta:=\big\{i\leq n:\;|v_i|\geq\beta\big\},
$$
we get $\Prob\{|J_\beta|<\zeta n\}<e^{- c_{\text{\tiny\ref{l: incomp in ker}} } n}$. Fix $\q=\lceil\log^2 n\rceil$.
Let $\Event':=\{\|(\Az )\n\|_2\leq n^{-2}\}$.
Using Markov's inequality we obtain
\begin{align}
 \Prob\big\{|I_\theta|\geq k\mbox{ and }&\|(\Az )\n\|_2\leq n^{-2}\big\}=\Exp\big({\bf 1}_{\{|I_\theta|\geq k\}}{\bf 1}_{\Event'}\big)\nonumber
 \\
&\leq k^{-\q}\Exp\big(|I_\theta|^{\q}{\bf 1}_{\{|I_\theta|\geq k\}}{\bf 1}_{\Event'}
{\bf 1}_{\{|J_\beta|\geq \zeta n\}}\big)+(n/k)^{\q}e^{-c_{\text{\tiny\ref{l: incomp in ker}} }n}\nonumber\\
&\le (\zeta nk^{\q})^{-1}\Exp\big(|I_\theta|^{\q}|J_\beta|{\bf 1}_{\{|I_\theta|\geq k\}}{\bf 1}_{\Event'}
{\bf 1}_{\{|J_\beta|\geq \zeta n\}}\big)+(n/k)^{\q}e^{-c_{\text{\tiny\ref{l: incomp in ker}} }n}, \label{Markov1}
\end{align}
so we need to estimate $\Exp\big(|I_\theta|^{\q}|J_\beta|{\bf 1}_{\{|I_\theta|\geq k\}}{\bf 1}_{\Event'}
{\bf 1}_{\{|J_\beta|\geq \zeta n\}}\big)$.

Given a set of distinct indices $\{i_1,\dots,i_{\q}\}$, let
$\chi_{i_1,\dots,i_{\q}}$  be the indicator of the event
$$
|v_{i_1}|\le \theta,\,\dots,|v_{i_{\q-1}}|\le \theta\quad\text{and}\quad |v_{i_{\q}}|\geq \beta.
$$
Since for any distinct $i_1,\dots,i_{\q-1}\in I_\theta$ and $i_\q\in J_\beta$
we have $\chi_{i_1,\dots,i_{\q}}=1$, we get a {\it deterministic} relation valid on the entire probability space:
\begin{align*}
&\sum
\chi_{i_1,\dots,i_{\q}}\ge |I_\theta|(|I_\theta|-1)\dots(|I_\theta|-\q+1)|J_\beta|{\bf 1}_{\{|I_\theta|\ge \q\}}
\ge 3^{-\q}|I_\theta|^{\q}|J_\beta|{\bf 1}_{\{|I_\theta|\ge k\}},
\end{align*}
where the sum is taken over all ordered $\q$--tuples $(i_1,\dots,i_{\q})$ with distinct components. Hence,
$$
\Exp\big(|I_\theta|^{\q}|J_\beta|{\bf 1}_{\{|I_\theta|\geq k\}}{\bf 1}_{\Event'}
{\bf 1}_{\{|J_\beta|\geq \zeta n\}}\big)\leq  3^\q\sum
\Exp\big(\chi_{i_1,\dots,i_{\q}}{\bf 1}_{\Event'}\big)
$$
and, by \eqref{Markov1},
\begin{align*}
 \Prob&\big\{|I_\theta|\geq k\mbox{ and }\|(\Az )\n\|_2\leq n^{-2}\big\}\le (\zeta n k^\q)^{-1}3^\q\sum
\Exp\big(\chi_{i_1,\dots,i_{\q}}{\bf 1}_{\Event'}\big)
+(n/k)^{\q}e^{-c_{\text{\tiny\ref{l: incomp in ker}} }n}.
\end{align*}
Consider two cases. First, assume that $100\q^2\le |z|\le \M\sqrt{n}$.
If $k\ge \zeta^2 \log^{2C}n$ then $\theta/\beta\ge n^{-1/2}$, and
Proposition \ref{prop: mn} implies that for any sequence $(i_1,\dots,i_\q)$ of distinct indices
\begin{align*}
\Exp\big(\chi_{i_1,\dots,i_{\q}}{\bf 1}_{\Event'}\big)
\leq (N^{C'}\theta/\beta)^{2\q}=\Big(\frac{N^{2C'}k}{\zeta^2n\log^{2C}n}\Big)^\q,
\end{align*}
where $C'$ depends only on $K$ and $\M$.
Hence,
\begin{align*}
\Prob\big\{|I_\theta|\geq k\mbox{ and }\|(\Az )\n\|_2\leq n^{-2}\big\}
\leq (\zeta n)^{-1}(3/\zeta^2)^{\q}\q^{2C'\q}(\log n)^{-2C\q}+(n/k)^{\q}e^{-c_{\text{\tiny\ref{l: incomp in ker}} }n}.
\end{align*}
Taking $C>2C'$ big enough we get that the right hand side is less than $e^{-2\log^{2} n}$.
Since $\n$ was arbitrary, this implies $\Prob(\Event_k(z))\leq e^{-2\log^{2} n}$, and the result follows.

Now, assume that $|z|\le 100\q^2$. Then we repeat the above argument, replacing Proposition \ref{prop: mn}
with Lemma \ref{l: mn}. This finishes the proof of Theorem C.

To strengthen  the result and get  probability estimate $1-e^{-\log^{C'}n}$ for any constant $C'>2$, one can repeat the argument
above taking $\q:=[\log^{C''}n]$ with sufficiently large $C''$.
\end{proof}

\medskip

\begin{rem}\label{rem: Th B}
The proof of Theorem B in the case of real eigenvalues $\lambda$ such that
$100\log^4n\le|\lambda|\le \M \sqrt{n}$ follows the same scheme and is based on
Proposition \ref{prop: mn} with $U=1$. For $|\lambda|\le 100\log^4n$ the statement follows from Corollary \ref{cor: eigenv}.
\end{rem}

\section{Acknowledgments}

The research is partially supported by the Simons foundation. A.L. was supported by grant nr 2018/31/B/ST1/03937 National Science Centre, Poland.
The authors would like to thank the Referees of the paper for valuable suggestions.

\end{document}